\documentclass[12pt]{article}
\usepackage{a4wide}
\usepackage[french,english]{babel}
\usepackage[utf8]{inputenc}
\usepackage[T1]{fontenc}
\usepackage{amsthm}
\usepackage{amssymb}
\usepackage{amsmath}
\usepackage{stmaryrd}
\usepackage{makeidx}
\usepackage{dsfont}
\usepackage{graphicx}
\usepackage{caption}
\theoremstyle{plain}
\usepackage[colorlinks=true,breaklinks=true,linkcolor=black]{hyperref} 
\usepackage{amsfonts}
\usepackage{esvect}

\newtheorem{theorem}{Theorem}[section] 
\newtheorem{proposition}[theorem]{Proposition}
\newtheorem{corollary}[theorem]{Corollary}
\newtheorem{definition}[theorem]{Definition}
\newtheorem{lemma}[theorem]{Lemma}
\newtheorem{example}[theorem]{Example}
\newtheorem{remark}[theorem]{Remark}
\newtheorem{conjecture}[theorem]{Conjecture}
\newtheorem{notation}[theorem]{Notation}
\newtheorem{definition/proposition}[theorem]{Definition/Proposition}
\theoremstyle{definition}\newtheorem{thm*}{Theorem}
\date{}
\makeindex
\usepackage{mathrsfs}
\usepackage[all]{xy}
\theoremstyle{definition}

\usepackage{geometry}
\title{Decompositions of principal  series representations of Iwahori-Hecke algebras for Kac-Moody groups over local fields}
\author{Auguste \textsc{Hébert} \\École normale supérieure de Lyon\\ UMR 5669 CNRS, \\ auguste.hebert@ens-lyon.fr}
\usepackage{mathrsfs}

\makeatletter \@addtoreset{figure}{section}\makeatother

\newcommand{\R}{\mathbb{R}}
\newcommand{\A}{\mathbb{A}}

\newcommand{\N}{\mathbb{N}}
\newcommand{\Z}{\mathbb{Z}}
\newcommand{\Q}{\mathbb{Q}}
\newcommand{\C}{\mathbb{C}}
\newcommand{\Ne}{\mathbb{N}^*}

\newcommand{\I}{\mathcal{I}}
\newcommand{\T}{\mathcal{T}}
\newcommand{\Id}{\mathrm{Id}}

\newcommand{\supp}{\mathrm{supp}}

\newcommand{\AC}{{^{\mathrm{BL}}\mathcal{H}}}
\newcommand{\ATF}{\AC(T_\FC)}
\newcommand{\AF}{\AC_\FC}
\newcommand{\ATC}{\AC(T_\C)}
\newcommand{\vect}{\mathrm{vect}}

\newcommand{\Hom}{\mathrm{Hom}}

\newcommand{\ms}{\mathrm{ms}}

\newcommand{\DC}{\mathcal{D}}
\newcommand{\FC}{\mathcal{F}}
\newcommand{\GC}{\mathcal{G}}
\newcommand{\HC}{\mathcal{H}}
\newcommand{\KC}{\mathcal{K}}
\newcommand{\MC}{\mathcal{M}}
\newcommand{\PC}{\mathcal{P}}
\newcommand{\RC}{\mathcal{R}}
\newcommand{\UC}{\mathcal{U}}

\newcommand{\HCW}{\mathcal{H}_{W^v,\C}}

\newcommand{\HFW}{{\mathcal{\HC}_{W^v,\FC}}}

\newcommand{\MCC}{\mathscr{M}}
\newcommand{\RCC}{\mathscr{R}}
\newcommand{\SCC}{\mathscr{S}}
\newcommand{\TCC}{\mathscr{T}}

\newcommand{\End}{\mathrm{End}}
\newcommand{\Wt}{\mathrm{Wt}}
\newcommand{\Wta}{W_{(\tau)}}
\newcommand{\vb}{\mathbf{v}}

\usepackage{comment}

\newcommand{\ev}{\mathrm{ev}}
\newcommand{\irr}{\mathrm{irr}}
\makeindex

\usepackage[left,modulo]{lineno}

\hypersetup{
pdfauthor={Hebert, Auguste},
pdftitle={Decomposition principal series representations},
}

\begin{document}

\maketitle

\begin{abstract}
Recently, Iwahori-Hecke algebras were associated to Kac-Moody groups over non-Archimedean local fields. In a previous paper, we introduced principal series representations for these algebras and partially generalized Kato's irreducibility criterion. In this paper, we study how some of these representations decompose  when they are reducible and deduce information on the irreducible  representations of these algebras.
\end{abstract}

\section{Introduction}

\subsection*{The reductive case}

Let $G$ be a split reductive group over a non-Archimedean local field $\KC$. To each open compact subgroup $K$ of $G$ is associated a Hecke algebra $\HC(K)$.  Let $I$ be the Iwahori subgroup of $G$. Then the Hecke algebra $\HC_\C$ associated with $I$ is called the Iwahori-Hecke algebra of $G$ and plays an important role in the representation theory of $G$. Its representations have been extensively studied.

 Let $Y$ be the cocharacter lattice of $G$ and let $T_\C$ be the set of nonzero algebra morphisms from $\C[Y]$ to $\C$. By the Bernstein-Lusztig relations, $\HC_\C$ contains the group algebra $\C[Y]$ of $Y$. Thus if $\tau\in T_\C$, one can define the induced representation $I_\tau$ of $\HC_\C$. Let $W^v$ be the vectorial (i.e finite) Weyl group of $G$. Then $I_\tau$ admits a basis indexed by $W^v$ and has dimension $|W^v|$. Representations of the form $I_\tau$, for $\tau\in T_\C$ were introduced by Matsumoto in \cite{matsumoto77Analyse} and are called \textbf{principal series representations}. By \cite[(4.2.4) Th\'eor\`eme]{matsumoto77Analyse}, every irreducible representation of $\HC_\C$ is a quotient of $I_\tau$ and embeds in $I_{\tau'}$, for some $\tau,\tau'\in T_\C$ and thus studying principal series representations enables to get information on the irreducible representations of $\HC_\C$.

\subsection*{Iwahori-Hecke algebras in the Kac-Moody case}

Kac-Moody groups are infinite dimensional (if not reductive) generalizations of reductive groups. Let now $G$ be a split Kac-Moody group (for Tits definition) over a non-Archimedean local field $\KC$.  In \cite{braverman2011spherical} and \cite{braverman2016iwahori}, Braverman, Kazhdan and Patnaik defined the spherical Hecke algebra and the Iwahori-Hecke $\HC_\C$ of $G$ when $G$ is affine. Bardy-Panse, Gaussent and Rousseau generalized these constructions to the case where $G$ is a general Kac-Moody group. Very few is known on the representation theory of $\HC_\C$ (see \cite{gaussent2014spherical} and \cite{bardy2016iwahori}).

Let $Y$ be the cocharacter lattice of $G$ and $W^v$ be the Weyl group of $G$. The algebra $\HC_\C$ can be embedded in the Bernstein-Lusztig algebra $\AC_\C$. As a vector space $\AC_\C$ is $\HCW\otimes \C[Y]$, where $\HCW$ is the Hecke algebra of the Coxeter group $W^v$ and  $\C[Y]$ is the group algebra of $Y$. It is equipped with a product $*$ defined by some relations called the Bernstein-Lusztig relations. The algebra $\HC_\C$ is then the subalgebra $\HCW\otimes \C[Y^+]$, where $Y^+=Y\cap\T$, where $\T$  is some convex cone (the Tits cone) of $Y\otimes \R$ (in the reductive case, $Y^+=Y$).

\subsection*{Weighted representations of $\AC_\C$ and $\HC_\C$}  Let $M$ be a representation of $\AC_\C$ (resp. $\HC_\C$) and $\tau\in T_\C$ (resp. $T_\C^+=\Hom_{\mathrm{alg}}(\C[Y^+],\C)\setminus\{0\}$). We say that $\tau$ is a \textbf{weight} of $M$ if there exists $m\in M\setminus\{0\}$ such that $\theta.m=\tau(\theta).m$, for every $m\in M$ (resp. $\theta\in \C[Y^+]$). 

We call a representation $M$ (resp. $M^+$) of $\AC_\C$ (resp. $\HC_\C$) \textbf{weighted} if for every $m\in M$ (resp. $m\in M^+$), $\C[Y].m$ (resp. $\C[Y^+].m$) is finite dimensional. In the reductive case, every irreducible representation of $\AC_\C=\HC_\C$ is finite dimensional and is therefore weighted. In the Kac-Moody (non reductive) case however, there always exist infinite dimensional irreducible representations of $\AC_\C$ or $\HC_\C$ (see \cite[Remark 5.11]{hebert2018principal}). However, we do not know if there exist  non weighted irreducible representations of $\AC_\C$ or $\HC_\C$. In this paper, we are mainly interested in the weighted representations.

  As we shall see (see Proposition~\ref{propWeighted_representations_AC_HC})  if $M$ is a weighted representation of $\AC_\C$, then the $\HC_\C$-submodules of $M$ are exactly the restrictions to $\HC_\C$ of the $\AC_\C$-submodules of $M$. In particular, $M$ is $\AC_\C$-irreducible if and only if it is $\HC_\C$-irreducible. We give a characterization of the weighted representations of $\HC_\C$ that can be extended to a representation of $\AC_\C$ (see Proposition~\ref{propCharacterization_restriction_weighted_representations}). Depending on $G$, it may happen that every weighted representation of $\HC_\C$ extends to a representation of $\AC_\C$ (for example when $G$ is affine or associated to a size $2$ Kac-Moody matrix). In this case it is equivalent to study the weighted representations of $\AC_\C$ and the weighted representations of $\HC_\C$.   Note that we constructed in \cite[4.2.1]{hebert2018principal} examples of weighted representations of $\HC_\C$ which cannot be extended to representations of $\AC_\C$.
  
   We then restrict our study to the weighted representations of $\AC_\C$ and more specifically to the principal series representations of $\AC_\C$.

\subsection*{Principal series representations of $\AC_\C$}

In \cite{hebert2018principal}, we associated to each $\tau\in T_\C$ a representation $I_\tau$ called a \textbf{principal series representation}.  A motivation to study these representations if that every weighted irreducible representation of $\AC_\C$ is the quotient of  $I_\tau$, for some $\tau\in T_\C$ (see \cite[Proposition 3.8]{hebert2018principal}). In this paper, we study, under some assumptions on  $\tau\in T_\C$, the submodules of $I_\tau$ and the irreducible (weighted) representations admitting $\tau$ as a weight.

 The action of $W^v$ on $Y$ induces an action of $W^v$ on $T_\C$. Let $\tau\in T_\C$ and let $W_\tau$ be the fixator of $\tau$ in $W^v$. As we shall see (Lemma~\ref{lemDecomposition_Wtau_Rgroup}), $W_\tau$ decomposes as $W_\tau=\Wta\rtimes R_\tau$, where $\Wta$ is some reflection subgroup of $W_\tau$ and $R_\tau$ is a generalization of the ``$R$-group'' introduced by Knapp and Stein in \cite{knapp1972irreducibility}. Let $q$ be the residue cardinal of $\KC$ and $\Phi^\vee$ be the coroot system of $G$. Let $\UC_\C=\{\tau\in T_\C|\tau(\alpha^\vee)\neq q,\forall \alpha^\vee\in \Phi^\vee\}$. Then:

\begin{thm*}(see \cite[Introduction, Theorem 3, 4]{hebert2018principal})
Let $\tau\in T_\C$. Suppose that $I_\tau$ is irreducible. Then: \begin{enumerate}
\item\label{itCriterionR} $R_\tau=\{1\}$ (or equivalently $W_\tau=\Wta$)

\item\label{itCriterion_value_alphavee} $\tau\in \UC_\C$. 
\end{enumerate}

Moreover, if $G$ is associated with a size $2$ Kac-Moody matrix, then $I_\tau$ is irreducible if and only if $\tau$ satisfies (\ref{itCriterionR}) and (\ref{itCriterion_value_alphavee}).
\end{thm*}

When $G$ is reductive,  $I_\tau$ is irreducible if and only if $\tau$ satisfies (\ref{itCriterionR}) and (\ref{itCriterion_value_alphavee}) by \cite[Th{\'e}or{\`e}me 4.3.5]{matsumoto77Analyse} and \cite[Theorem 2.4]{kato1982irreducibility}.

\medskip

One says that $\tau$ is regular when $W_\tau=1$. We mainly focus on the following cases:\begin{itemize}
\item $\tau$ is regular,

\item $\tau\in \UC_\C$ and the Kac-Moody matrix defining $G$ has size $2$.
\end{itemize}

\subsection*{The case where $\tau$ is regular}

Let $\tau\in T_\C$ be regular. There exists a set $\SCC\subset W^v$ such that $(W^v,\SCC)$ is a Coxeter system.  Let $\GC$ be the non-oriented graph defined as follows. Its vertices are the $I_{w.\tau}$, for $w\in W^v$ and for $v,w\in W^v$ there is an edge between $v$ and $w$ if $w=sv$ for some $s\in \SCC$.  If $w\in W^v$ and $s\in \SCC$, then $\dim \Hom_{\AC_\C-\mathrm{mod}}(I_{w.\tau},I_{sw.\tau})=1$. We choose a nonzero intertwining map $A_{w,sw,\tau}:I_{w.\tau}\rightarrow I_{sw.\tau}$. 

A \textbf{path} $\Gamma$ in $\GC$ is a finite sequence $\Gamma=\big(\Gamma(1),\ldots,\Gamma(k)\big)=(I_{w_1.\tau},I_{w_2.\tau}, \ldots, I_{w_k.\tau})$ such that for all $i\in \llbracket 1,n-1\rrbracket$, $w_{i+1}w_i^{-1}\in \SCC$. Then we define an intertwining map $A_\Gamma:I_{w_1.\tau}\rightarrow I_{w_k.\tau}$ by  $A_\Gamma= A_{w_{k-1},w_k,\tau}\circ \ldots\circ A_{w_1,w_2,\tau}$. The path $\Gamma$ is said to be \textbf{reduced} if $k=\ell(w_kw_1^{-1})$.  Let $v,w\in W^v$ and $\Gamma$ be any reduced path between $I_{v.\tau}$ and $I_{w.\tau}$. Then $A_\Gamma\neq 0$ and  $\mathrm{Hom}_{\AC_\C-\mathrm{mod}}(I_{v.\tau},I_{w.\tau})=\C A_\Gamma$. 

Let $e=(I_{w.\tau},I_{sw.\tau})$, with $w\in W^v$ and $s\in \SCC$. Then $A_e:I_{w.\tau}\rightarrow I_{sw.\tau}$ is an isomorphism if and only if $w.\tau(\alpha_s^\vee)\in \C\setminus\{q,q^{-1}\}$.  Let $\widetilde{\GC}$ be the diagram obtained from $\GC$ by deleting the edges $e$ for which $A_e$ is not an isomorphism. We call a submodule $M$ of $I_\tau$ \textbf{strongly indecomposable} if for all family $(M_j)_{j\in J}$ of submodules such that $\sum_{j\in J} M_j=M$, there exists $j\in J$ such that $M_j=M$. Then we prove the following theorem  (see Proposition~\ref{propWeights_irreducible_components} and  Theorem~\ref{thmKrull_Schmidt_theorem}):

\begin{thm*}\label{thm*_decomposition_regular_tau}
\begin{enumerate}
\item Let $w\in W^v$. Then there exists (up to isomorphism) a unique irreducible representation $M_{w.\tau}^{\mathrm{irr}}$  of $\AC_\C$ admitting $w.\tau$ as a weight. Let $\tilde{C}(w)$ be the connected component of $\widetilde{\GC}$ containing $I_{w.\tau}$. Then $\dim M_{w.\tau}^{\mathrm{irr}}=|\tilde{C}(w)|=|\{v\in W^v|I_{v.\tau}\simeq I_{w.\tau}\}|$ (this cardinal can be infinite) and the set of weights of $M_{w.\tau}^{\mathrm{irr}}$ is $\{v.\tau\in W^v.\tau|\ I_{v.\tau}\in \tilde{C}(w)\}$. In particular, for all $v,w\in W^v$, $M_{v.\tau}^{\mathrm{irr}}\simeq M_{w.\tau}^{\mathrm{irr}}$ if and only if $\tilde{C}(v)=\tilde{C}(w)$.

\item For each connected component $\tilde{C}$ of $\tilde{\GC}$, choose a vertex $I_{w_{\tilde{C}.\tau}}$ of  $\tilde{C}$ and choose a reduced path $\Gamma_{\tilde{C}}$ from $I_{w_{\tilde{C}.\tau}}$ to $I_\tau$. Then the map $\tilde{C}\mapsto A_{\Gamma_{\tilde{C}}}(I_{w_{\tilde{C}}.\tau})$ is a bijection from the set of connected components of $\widetilde{\GC}$ to the set of strongly indecomposable submodules of $I_\tau$. 

\item Let $M$ be a submodule of $I_\tau$. Let $\mathrm{SI}(M)$ (resp. $\mathrm{MSI}(M)$) be the set of strongly indecomposable submodules (resp. maximal strongly indecomposable submodules) of $M$. Then $M=\sum_{N\in \mathrm{MSI}(M)} N$ and if $\mathcal{M}\subset \mathrm{SI}(M)$ is such that $M=\sum_{N\in \mathcal{M}} N$, then $\mathrm{MSI}(M)\subset \mathcal{M}$.

\end{enumerate}
\end{thm*}

\subsection*{The case where $\tau\in \UC_\C$}

We now assume that $G$ is associated with a size $2$ Kac-Moody matrix and we fix $\tau\in \UC_\C$. Then for all $w\in W^v$, $I_\tau$ is isomorphic to $I_{w.\tau}$. For $J\subset \End_{\AC_\C-\mathrm{mod}}(I_\tau)$ a right ideal and $M\subset I_\tau$ a submodule, we set: \[J(I_\tau)=\sum_{\phi\in J}\phi(I_\tau)\text{ and }J_M=\{\phi\in \End_{\AC_\C-\mathrm{mod}}(I_\tau)|\phi(I_\tau)\subset M\}.\] Then (see  Proposition~\ref{propDescription_endomorphism_algebra}, Theorem~\ref{thmBijection_modules_ideals}, Theorem~\ref{thmIrreducible_representations_tau_UC} and Lemma~\ref{lemList_possibilities_Wtau_Rtau_Wta}):

 \begin{thm*}\label{thm*_UC}\begin{enumerate}
\item The map $M\mapsto J_M$ is a bijection from the set of submodules of $I_\tau$ to the set of right ideals of $\End_{\AC_\C-\mathrm{mod}}(I_\tau)$. Its inverse is $J\mapsto J(I_\tau)$.

\item $\End_{\AC_\C}(I_\tau)$ is isomorphic to the group algebra $\C[R_\tau]$.

\item The set of possible $R_\tau$  is exactly $\{1\}$, $\Z/2\Z$, $\Z$, the infinite dihedral group $D_\infty$.

\item The map $M\mapsto I_\tau/M$ is a surjection from the set of maximal submodules of $I_\tau$ to the set of irreducible representations of $\AC_\C$ admitting the weight $\tau$. It is a bijection if and only if every maximal right ideal of $\End(I_\tau)$ is two-sided (which is the case when $R_\tau$ is commutative). In this case these representations have dimension $|\Wta||W^v/W_\tau|$ (it can be infinite).
\end{enumerate}

\end{thm*}

We conjecture that for the assumption on the size of the Kac-Moody matrix is useless for the points (1), (2) and (4). 

\paragraph{Frameworks}
Actually, following \cite{bardy2016iwahori} we study Iwahori-Hecke algebras associated to abstract masures. In particular our results also apply when $G$ is an almost-split Kac-Moody group over a non-Archimedean local field. In this case, most of the results of this introduction are true but the formulas are more complicated (they are given in the paper). Point (2) of Theorem~\ref{thm*_UC} can fail.  In sections~\ref{secWeighted_representations} and \ref{secRegular_representations} we work over a field $\FC$ which can be different from $\C$.

The paper is organized as follows. 

In section~\ref{secIH algebras}, we recall the definition of the Iwahori-Hecke algebras and of the principal series representations.

In section~\ref{secWeighted_representations}, we introduce the weighted representations and study the links between the weighted representations of $\HC_\C$ and those of $\AC_\C$.

In section~\ref{secRegular_representations}, we study $I_\tau$, for $\tau\in T_\C$ regular and prove Theorem~\ref{thm*_decomposition_regular_tau}.

In section~\ref{secTau_in_UC}, we study $I_\tau$, for $\tau\in \UC_\C$ and prove Theorem~\ref{thm*_UC}.

\paragraph{Funding}
The author was supported by the ANR grant ANR-15-CE40-0012.

\tableofcontents

\section{Iwahori-Hecke algebras}\label{secIH algebras}\label{secIH_algebras}
Let $G$ be a Kac-Moody group over a non-archimedean local field. Then Gaussent and Rousseau constructed a space $\I$, called a masure on which $G$ acts, generalizing the construction of the Bruhat-Tits buildings (see \cite{gaussent2008kac}, \cite{rousseau2016groupes} and \cite{rousseau2017almost}). In \cite{bardy2016iwahori} Bardy-Panse, Gaussent and Rousseau attached an Iwahori-Hecke algebra $\HC_\RC$ to each  masure satisfying certain conditions and to each ring $\RC$. They in particular attach an Iwahori-Hecke algebra to each almost-split Kac-Moody group over a local field. The algebra $\HC_\RC$ is an algebra of functions defined on some pairs of chambers of the masure, equipped with a convolution product. Then they prove that under some additional hypothesis on the ring $\RC$ (which are satisfied by $\R$ and $\C$),  $\HC_\RC$ admits a Bernstein-Lusztig presentation. In this paper, we will only use the Bernstein-Lusztig presentation of $\HC_\RC$ and we do not introduce masures. We however introduce the standard apartment of a masure. We restrict our study to the case where $\RC=\FC$ is a field. 

\subsection{Standard apartment of a masure}\label{subRootGenSyst}
A \textbf{ Kac-Moody matrix} (or { generalized Cartan matrix}) is a square matrix $A=(a_{i,j})_{i,j\in I}$ indexed by a finite set $I$, with integral coefficients, and such that :
\begin{enumerate}
\item[\tt $(i)$] $\forall \ i\in I,\ a_{i,i}=2$;

\item[\tt $(ii)$] $\forall \ (i,j)\in I^2, (i \neq j) \Rightarrow (a_{i,j}\leq 0)$;

\item[\tt $(iii)$] $\forall \ (i,j)\in I^2,\ (a_{i,j}=0) \Leftrightarrow (a_{j,i}=0$).
\end{enumerate}
A \textbf{root generating system} is a $5$-tuple $\mathcal{S}=(A,X,Y,(\alpha_i)_{i\in I},(\alpha_i^\vee)_{i\in I})$\index{$\mathcal{S}$} made of a Kac-Moody matrix $A$ indexed by the finite set $I$, of two dual free $\Z$-modules $X$ and $Y$ of finite rank, and of a free family $(\alpha_i)_{i\in I}$ (respectively $(\alpha_i^\vee)_{i\in I}$) of elements in $X$ (resp. $Y$) called \textbf{simple roots} (resp. \textbf{simple coroots}) that satisfy $a_{i,j}=\alpha_j(\alpha_i^\vee)$ for all $i,j$ in $I$. Elements of $X$ (respectively of $Y$) are called \textbf{characters} (resp. \textbf{cocharacters}).

Fix such a root generating system $\mathcal{S}=(A,X,Y,(\alpha_i)_{i\in I},(\alpha_i^\vee)_{i\in I})$ and set $\A:=Y\otimes \R$\index{$\A$}. Each element of $X$ induces a linear form on $\A$, hence $X$ can be seen as a subset of the dual $\A^*$. In particular, the $\alpha_{i}$'s (with $i \in I$) will be seen as linear forms on $\A$. This allows us to define, for any $i \in I$, an involution $r_{i}$ of $\A$ by setting $r_{i}(v) := v-\alpha_i(v)\alpha_i^\vee$ for any $v \in \A$. Let $\SCC=\{r_i|i\in I\}$\index{$\SCC$} be the (finite) set of \textbf{simple reflections}.  One defines the \textbf{Weyl group of $\mathcal{S}$} as the subgroup $W^{v}$\index{$W^v$} of $\mathrm{GL}(\A)$ generated by $\SCC$. The pair $(W^{v}, \SCC)$ is a Coxeter system, hence we can consider the length $\ell(w)$ with respect to $\SCC$ of any element $w$ of $W^{v}$. If $s\in \SCC$, $s=r_i$ for some unique $i\in I$. We set $\alpha_s=\alpha_i$ and $\alpha_s^\vee=\alpha_i^\vee$.

The following formula defines an action of the Weyl group $W^{v}$ on $\A^{*}$:  
\[\displaystyle \forall \ x \in \A , w \in W^{v} , \alpha \in \A^{*} , \ (w.\alpha)(x):= \alpha(w^{-1}.x).\]
Let $\Phi:= \{w.\alpha_i|(w,i)\in W^{v}\times I\}$\index{$\Phi,\Phi^\vee$} (resp. $\Phi^\vee=\{w.\alpha_i^\vee|(w,i)\in W^{v}\times I\}$) be the set of \textbf{real roots} (resp. \textbf{real coroots}): then $\Phi$ (resp. $\Phi^\vee$) is a subset of the \textbf{root lattice} $Q_\Z:= \displaystyle \bigoplus_{i\in I}\Z\alpha_i$ (resp. \textbf{coroot lattice} $Q^\vee_\Z=\bigoplus_{i\in I}\Z\alpha_i^\vee$). By \cite[1.2.2 (2)]{kumar2002kac}, one has $\R \alpha^\vee\cap \Phi^\vee=\{\pm \alpha^\vee\}$ and $\R \alpha\cap \Phi=\{\pm \alpha\}$ for all $\alpha^\vee\in \Phi^\vee$ and $\alpha\in \Phi$.

As in the reductive case, define the \textbf{fundamental chamber} as $C_{f}^{v}:= \{v\in \A \ \vert \ \forall s \in \SCC,  \alpha_s(v)>0\}$. 

 Let $\mathcal{T}:= \displaystyle \bigcup_{w\in W^{v}} w.\overline{C^{v}_{f}}$\index{$\T$} be the \textbf{Tits cone}. This is a convex cone (see \cite[1.4]{kumar2002kac}).
 
One sets $Y^+=Y\cap \T$.

\begin{remark}
By \cite[§4.9]{kac1994infinite} and \cite[§ 5.8]{kac1994infinite} the following conditions are equivalent:\begin{enumerate}
\item the Kac-Moody matrix $A$ is of finite type (i.e. is a Cartan matrix),

\item $\A=\T$

\item $W^v$ is finite.
\end{enumerate}
\end{remark}

\subsection{Recalls on Coxeter groups }\label{subReflection_subgroups}

\subsubsection{Bruhat order}

Let $(W_0,\SCC_0)$ be a Coxeter system. We equip it with the Bruhat order $\leq_{W_0}$ (see \cite[Definition 2.1.1]{bjorner2005combinatorics}). We have the following characterization (see \cite[Corollary 2.2.3]{bjorner2005combinatorics}): let $u,w\in W_0$. Then $u\leq_{W_0} w$ if and only if every reduced expression for $w$ has a subword that is a reduced
expression for $u$ if and only if there exists a reduced expression for $w$ whose subword is a reduced expression for $u$. By \cite[Proposition 2.2.9]{bjorner2005combinatorics}, $(W_0,\leq_{W_0})$ is a directed poset, i.e for every finite set $E\subset W_0$, there exists $w\in W_0$ such that $v\leq_{W_0} w$ for all $v\in E$. 

We write $\leq$ instead of $\leq_{W^v}$. For $u,v\in W^v$, we denote by $[u,v]$, $[u,v)$, $\ldots$ the sets $\{w\in W^v|u\leq w\leq v\}$, $\{w\in W^v|u\leq w <v\}$, $\ldots$. 

\subsubsection{Reflections  and coroots}

Let $\RCC=\{wsw^{-1}|w\in W^v, s\in \SCC\}$\index{$\RCC$} be the set of \textbf{reflections} of $W^v$. Let  $r\in \RCC$. Write $r=wsw^{-1}$, where $w\in W^v$, $s\in \SCC$ and $ws>w$ (which is possible because if $ws<w$, then $r=(ws)s(ws)^{-1}$). Then one sets $\alpha_r=w.\alpha_s\in \Phi_+$\index{$\alpha_r,\alpha_r^\vee$} (resp. $\alpha_r^\vee=w.\alpha_s^\vee\in\Phi^\vee_+$). This is well defined by the lemma below.

\begin{lemma}\label{lemDefinition_root_reflection} (see \cite[Lemma 2.2]{hebert2018principal})
Let $w,w'\in W^v$ and $s,s'\in \SCC$ be such that $wsw^{-1}=w's'w'^{-1}$ and $ws>w$, $w's'>w'$. Then $w.\alpha_s=w'.\alpha_{s'}\in \Phi_+$ and $w.\alpha_s^\vee=w'.\alpha_{s'}^\vee\in \Phi^\vee_+$.
\end{lemma}

\begin{lemma}\label{lemKumar1.3.11}(see \cite[Lemma 2.3]{hebert2018principal})
Let $r,r'\in \RCC$ and $w\in W^v$ be such that $w.\alpha_r=\alpha_{r'}$ or $w.\alpha_r^\vee=\alpha_{r'}^\vee$. Then $wrw^{-1}=r'$.
\end{lemma}

Let $r\in \RCC$. Then for all $x \in \A$, one has: \[r(x)=x-\alpha_r(x)\alpha_r^\vee.\]  Let $\alpha^\vee\in \Phi^\vee$. One sets $r_{\alpha^\vee}=wsw^{-1}$\index{$r_{\alpha^\vee}$} where $(w,s)\in W^v\times \SCC$ is such that $\alpha^\vee=w.\alpha_s^\vee$. This is well defined, by Lemma~\ref{lemKumar1.3.11}. Thus $\alpha^\vee\mapsto r_{\alpha^\vee}$ and $r\mapsto \alpha_r^\vee$ induce bijections $\Phi^\vee_+\rightarrow \RCC$ and $\RCC\rightarrow \Phi^\vee_+$.

For $w\in W^v$, set $N_{\Phi^\vee}(w)=\{\alpha^\vee\in \Phi^\vee_+|w.\alpha^\vee\in \Phi^\vee_-\}$\index{$N_{\Phi^\vee}(w)$}.

\begin{lemma}\label{lemKumar_1.3.14}(\cite[Lemma 1.3.14]{kumar2002kac})
Let $w\in W^v$.  Then $|N_{\Phi^\vee}(w)|=\ell(w)$ and if $w=s_1\ldots s_r$ is a reduced expression, then $N_{\Phi^\vee}(w)=\{\alpha_{s_r}^\vee,s_r.\alpha_{s_{r-1}}^\vee,\ldots,s_r\ldots s_{2}.\alpha_{s_1}^\vee\}$. 
\end{lemma}

\subsection{Iwahori-Hecke algebras}\label{subIH algebras}
In this subsection, we give the definition of  the Iwahori-Hecke algebra via its Bernstein-Lusztig presentation.
\subsubsection{The algebra $\AC(T_\FC)$}\label{subsubAlgebra_H(T_F)}

Let  $\RC_1=\Z[(\sigma_s)_{s\in \SCC},(\sigma'_s)_{s\in \SCC}]$, where $(\sigma_s)_{s\in \SCC}, (\sigma'_s)_{s\in \SCC}$ are two families of indeterminates satisfying the following relations:
\begin{itemize}
\item if $\alpha_{s}(Y) = \Z$, then $\sigma_{s} = \sigma'_{s}$\index{$\sigma_s,\sigma_s'$};  
\item if $s,t \in \SCC$ are conjugate (i.e. such that $\alpha_{s}(\alpha^{\vee}_{t}) = \alpha_{t}(\alpha^{\vee}_{s}) = -1$), then $\sigma_{s}=\sigma_{t}=\sigma'_{s}=\sigma'_{t}$.
\end{itemize}

\begin{definition} Let $\FC$ be a field of characteristic $0$ and $f:\RC_1\rightarrow \FC$ be a morphism such that $f(\sigma_s),f(\sigma_s')\in \FC^*$, for every $s\in \SCC$. We write $\sigma_s$ or $\sigma'_s$ instead of $f(\sigma_s)$, $f(\sigma_s')$.  Let $\HC_{W^v,\FC}$ be the \textbf{Hecke algebra of the Coxeter group $W^v$ over $\FC$}, that is: \begin{itemize}

\item  as a vector space, $\HC_{W^v,\FC}=\bigoplus_{w\in W^v} \FC H_w$, where $H_w$, $w\in W^v$ are symbols,

\item $\forall \ s \in \SCC, \forall \ w \in W^{v}$, $H_{s}*H_{w}=\left\{\begin{aligned} & H_{sw} &\mathrm{\ if\ }\ell(sw)=\ell(w)+1\\ & (\sigma_{s}-\sigma_{s}^{-1}) H_{w}+H_{s w} &\mathrm{\ if\ }\ell(sw)=\ell(w)-1 .\end{aligned}\right . \ $
\end{itemize}

\end{definition}

Let $\FC[Y]$ be the group algebra of $Y$ over $\FC$, that is:\begin{itemize}
\item as a vector space, $\FC[Y]=\bigoplus_{\lambda\in Y} \FC Z^\lambda$, where the $Z^\lambda$, $\lambda\in Y$ are symbols,

\item for all $\lambda,\mu\in Y$, $Z^\lambda*Z^\mu=Z^{\lambda+\mu}$.
\end{itemize}

 We denote by $\FC(Y)$ its field of fractions. For $\theta=\frac{\sum_{\lambda\in Y}{a_\lambda Z^\lambda}}{\sum_{\lambda\in Y} b_\lambda Z^\lambda}\in \FC(Y)$ and $w\in W^v$, set ${^w}\theta:=\frac{\sum_{\lambda\in Y}a_\lambda Z^{w.\lambda}}{\sum_{\lambda\in Y}b_\lambda Z^{w.\lambda}}$. 

\medskip

Let $\AC(T_\FC)$ be the algebra defined as follows: \begin{itemize}
\item as a vector space, $\AC(T_\FC)= \FC(Y)\otimes \HC_{W^v, \FC}$ (we write $\theta*h$ instead of $\theta\otimes h$ for $\theta\in \FC(Y)$ and $h\in \HC_{W^v,\FC}$),

\item $\AC(T_\FC)$ is equipped with the unique product $*$ which turns it into an associative algebra and such that, for  $\theta\in \FC(Y)$ and $s\in \SCC$, one has: \[H_{s}*\theta-{^s\theta}*H_{s} =Q_s(Z)(\theta-{^s\theta}),\] where $Q_s(Z)=\frac{(\sigma_s-\sigma_s^{-1})+(\sigma_s'-\sigma_s'^{-1})Z^{-\alpha_s^\vee}}{1-Z^{-2\alpha_s^\vee}}$\index{$Q_s(Z)$}. 

\end{itemize}

By \cite[Proposition 2.10]{hebert2018principal}, such an algebra exists and is unique.

\subsubsection{The Bernstein-Lusztig Hecke algebra and the Iwahori-Hecke algebra}

Let $C^v_f=\{x\in \A|\alpha_i(x)>0\forall i\in I\}$, $\T=\bigcup_{w\in W^v} w.\overline{C}^v_f$ be the \textbf{Tits cone} and $Y^+=Y\cap \T$.

\begin{definition}\label{defBernstein-Lusztig algebra}
 Let $\FC$ be a field of characteristic $0$ and $f:\RC_1\rightarrow \FC$ be a morphism such that $f(\sigma_s),f(\sigma_s')\in \FC^*$, for every $s\in \SCC$. The \textbf{Bernstein-Lusztig-Hecke algebra of } $(\A,(\sigma_s)_{s\in \SCC},(\sigma'_s)_{s\in \SCC})$ over $\FC$ is the subalgebra $\AC_\FC=\bigoplus_{\lambda\in Y,w\in W^v}\FC Z^\lambda*H_w=\bigoplus_{\lambda\in Y,w\in W^v}\FC  H_w* Z^\lambda$ of $\AC(T_\FC)$. The \textbf{Iwahori-Hecke algebra of $(\A,(\sigma_s)_{s\in \SCC},(\sigma'_s)_{s\in \SCC})$ over $\FC$} is the subalgebra $\HC_\FC=\bigoplus_{\lambda\in Y^+,w\in W^v}\FC Z^\lambda*H_w=\bigoplus_{\lambda\in Y^+,w\in W^v}\FC Z^\lambda*H_w$ of $\AC_\FC$. Note that for $G$ reductive, we recover the usual Iwahori-Hecke algebra of $G$, since $ \T=\A$.
\end{definition}

\begin{remark}\label{remIH algebre dans le cas KM deploye}

\begin{enumerate}
\item The algebra $\AC_\FC$  was first defined in \cite[Theorem 6.2]{bardy2016iwahori} without defining $\AC(T_\FC)$. Let $\mathcal{K}$ be a non-Archimedean local field and $q$ be its residue cardinal. Let $\mathbf{G}$ be the minimal Kac-Moody group associated with $\mathcal{S}=(A,X,Y,(\alpha_i)_{i\in I},(\alpha_i^\vee)_{i\in I})$\index{$\mathcal{S}$} and $G=\mathbf{G}(\mathcal{K})$ (see \cite[Section 8]{remy2002groupes} or \cite{tits1987uniqueness} for the definition). Let $\FC$ to be a field containing $\Z[\sqrt{q}^{\pm 1}]$ and  take $f(\sigma_s)=f(\sigma'_s)=\sqrt{q}$ for all $s\in \SCC$. Then $\HC_\FC$ is the Iwahori-Hecke algebra of $G$ (see \cite[Definition 2.5 and 6.6 Proposition]{bardy2016iwahori}). In the case where $G$ is an untwisted affine Kac-Moody group, these algebras were introduced in \cite{braverman2016iwahori}. Note also  that our frameworks is more general than the one of split Kac-Moody groups over local fields. It enables for example to study the Iwahori-Hecke algebras associated to almost split Kac-Moody groups over local fields, as in \cite{bardy2016iwahori}. In this case we do not have necessarily $\sigma_s=\sigma_s'=\sigma_t=\sigma_t'$ for all $s,t\in \SCC$. Most of our results remain true in this case (the only result where we need such an assumption is Proposition~\ref{propDescription_endomorphism_algebra}) but the formulas are slightly more complicated.

\item Let $s\in \SCC$. Then if $\sigma_s=\sigma'_s$, $Q_s(Z)=\frac{(\sigma_s-\sigma_s^{-1})}{1-Z^{-\alpha_s^\vee}}$.

\item\label{itPolynomiality_Bernstein_Lusztig} Let $s\in \SCC$ and $\theta\in \FC[Y]$. Then $Q_s(Z)(\theta-{^s\theta})\in \FC[Y]$ and if moreover $\theta\in \FC[Y^+]$, then $Q_s(Z)(\theta-{^s\theta})\in \FC[Y^+]$. Indeed, let $\lambda\in Y$. Then $Q_s(Z)(Z^\lambda-Z^{s.\lambda})=Q_s(Z).Z^\lambda(1-Z^{-\alpha_s(\lambda)\alpha_s^\vee})$. Assume that $\sigma_s=\sigma_s'$. Then \[ \frac{1-Z^{-\alpha_s(\lambda)\alpha_s^\vee}}{1-Z^{-\alpha_s^\vee}}=\left\{\begin{aligned} &\sum_{j=0}^{\alpha_s(\lambda)-1}Z^{-j\alpha_s^\vee}  &\mathrm{\ if\ }\alpha_s(\lambda)\geq 0\\ & -Z^{\alpha_s^\vee}\sum_{j=0}^{-\alpha_s(\lambda)-1}Z^{j\alpha_s^\vee}  &\mathrm{\ if\ }\alpha_s(\lambda)\leq 0,\end{aligned}\right.\] and thus $Q_s(Z)(Z^\lambda-Z^{s.\lambda})\in \FC[Y]$. Assume $\sigma_s'\neq \sigma_s$. Then $\alpha_s(Y)=2\Z$ and a similar computation enables to conclude. In particular, $\AC_\FC$ and $\HC_\FC$ are subalgebras of $\AC(T_\FC)$.

\end{enumerate}

\end{remark}

\begin{lemma}\label{lemCommutation relation}(see \cite[Lemma 2.8]{hebert2018principal})
Let $\theta\in \FC[Y]$ and $w\in W^v$. Then $\theta* H_w -H_w* \theta^{w^{-1}}\in \AF^{<w}:=\bigoplus_{v<w} H_v\FC[Y]$.  In particular, $\AF^{\leq w}:=\bigoplus_{v\leq w}H_v\C[Y]$ is a left finitely generated $\FC[Y]$-submodule of $\AF$.
\end{lemma}

\subsection{Principal series representations}\label{subPrincipal series representations}
In this subsection, we introduce the principal series representations of $\AF$. 

We now fix $(\A,(\sigma_s)_{s\in \SCC},(\sigma'_s)_{s\in \SCC})$ as in Subsection~\ref{subIH algebras} and a field $\FC$ as in Definition~\ref{defBernstein-Lusztig algebra}. Let $\HC_\FC$ and $\AC_\FC$ be the Iwahori-Hecke and the Bernstein-Lusztig Hecke algebras of $(\A,(\sigma_s)_{s\in \SCC},(\sigma'_s)_{s\in \SCC})$  over $\FC$.

Let  $T_\FC= \Hom_{\mathrm{Gr}}(Y,\FC^*)$\index{$T_\FC$} be the group of group morphisms from $Y$ to $\FC^*$. Let $\tau\in T_\FC$. Then $\tau$ induces an algebra morphism $\tau:\FC[Y]\rightarrow \FC$ by the formula $\tau(\sum_{\lambda\in Y} a_\lambda  Z^\lambda)=\sum_{\lambda\in Y} a_\lambda \tau(\lambda)$, for $\sum a_\lambda Z^\lambda\in \FC[Y]$. This equips $\FC$ with the structure of an  $\FC[Y]$-module.

Let $I_\tau=\mathrm{Ind}^{\AC_\FC}(\tau)=\AC_\FC\otimes_{\FC[Y]} \FC$\index{$I_\tau$}. As a vector space, $I_\tau=\bigoplus_{w\in W^v} \FC \vb_\tau$, where $\vb_\tau$ is some symbol. The actions of $\AC_\FC$ on $I_\tau$ is as follows. Let $h=\sum_{w\in W^v} H_w P_w\in \AC_\FC$, where $P_w\in \FC[Y]$ for all $w\in W^v$. Then $h.\vb_\tau=\sum_{w\in W^v} \tau(P_w)H_w\vb_\tau$. In particular, $I_\tau$ is a principal $\AC_\FC$-module generated by $\vb_\tau$.

 We regard the elements of $\FC[Y]$ as polynomial functions on $T_\FC$ by setting: \[\tau(\sum_{\lambda\in Y}a_\lambda Z^\lambda)=\sum_{\lambda\in Y}a_\lambda\tau(\lambda),\] for all $(a_\lambda)\in \FC^{(Y)}$. The ring $\FC[Y]$ is a unique factorization domain. Let $\theta\in \FC(Y)$ and $(f,g)\in \FC[Y]\times \FC[Y]^*$ be such that $\theta=\frac{f}{g}$ and $f$ and $g$ are coprime. Set $\DC(\theta)=\{\tau\in T_\FC|\theta(g)\neq 0\}$\index{$\DC(\theta)$}. Then we regard  $\theta$ as a map from $\DC(\theta)$ to $\FC$ by setting $\theta(\tau)=\frac{f(\tau)}{g(\tau)}$ for all $\tau\in \DC(\theta)$.

 For $w\in W^v$, let $\pi^H_w:\ATF\rightarrow \FC(Y)$\index{$\pi^H_w$} be defined by $\pi^H_w(\sum_{v\in W^v} H_w\theta_v)=\theta_w$. If $\tau\in T_\FC$, let $\FC(Y)_\tau=\{\frac{f}{g}|f,g\in \FC[Y]\text{ and } g(\tau)\neq 0\}\subset \FC(Y)$\index{$\FC(Y)_\tau, \C(Y)_\tau$}. Let $\ATF_\tau=\bigoplus_{w\in W^v} H_w \FC(Y)_\tau \subset \ATF$\index{$\ATF_\tau,\ATC_\tau$}. This is a not a subalgebra of $\ATF$ (consider for example $\frac{1}{Z^\lambda-1}*H_s=H_s*\frac{1}{Z^{s.\lambda}-1}+\ldots$ for some well chosen $\lambda\in Y$, $s\in \SCC$ and $\tau\in T_\FC$). It is however an $\HFW-\FC(Y)_\tau$ bimodule. For $\tau\in T_\FC$, we define $\ev_\tau:\ATF_\tau\rightarrow \HFW$\index{$ev_\tau$} by $\ev_\tau(h)=h(\tau)=\sum_{w\in W^v} H_w\theta_w(\tau)$ if $h=\sum_{w\in W^v}H_w \theta_w\in \HC(Y)_\tau$. This is a morphism of $\HFW-\FC(Y)_\tau$-bimodule.

\subsection{Weights and intertwining operators}

In this subsection, we recall results on intertwining operators and weights from \cite{hebert2018principal} and prove general facts on the weights of $\AC_\FC$-modules.

 If $M,M'$ is an $\AC_\FC$-module, we write $\Hom(M,M')$ the space of $\AC_\FC$-module morphisms from $M$ to $M'$, $\End(M)$ the algebra of $AC_\FC$-module endomorphisms ...

Let $M$ be a $\AC_\FC$-module. For $\tau\in T_\FC$, set \[M(\tau)=\{m\in M|\theta.m=\tau(\theta).m\ \forall \theta\in \FC[Y]\},\Wt(M)=\{\tau\in T_\FC| M(\tau)\neq\{0\}\}\] and  \index{$M(\tau)$, $I_\tau(\tau)$}\[M(\tau,\text{gen})=\{m\in M|\exists k\in \N|\forall \theta\in \FC[Y], (\theta-\tau(\theta))^k.m=0\}\supset M(\tau).\]\index{$M(\tau,\mathrm{gen})$, $I_\tau(\tau,\mathrm{gen})$}

Let $M$ be a $\AC_\FC$-module and $\tau\in T_\FC$. For $x\in M(\tau)$ define $\Upsilon_x:I_\tau\rightarrow M$\index{$\Upsilon$} by $\Upsilon_x(u.\vb_\tau)=u.x$, for all $u\in \AC_\FC$. Then $\Upsilon_x$ is well defined. Indeed, let $u\in \AC_\FC$ be such that $u.\vb_\tau=0$. Then $u\in \FC[Y]$ and $\tau(u)=0$. Therefore $u.x=0$ and hence $\Upsilon_x$ is well defined. The following lemma is then  easy to prove.

\begin{lemma}\label{lemFrobenius_reciprocity} (Frobenius reciprocity, see \cite[Proposition 1.10]{kato1982irreducibility})
Let $M$ be a $\AC_\FC$-module, $\tau\in T_\FC$ and $x\in M(\tau)$. Then the map $\Upsilon:M(\tau)\rightarrow \Hom(I_\tau,M)$  mapping each $x\in M(\tau)$ to $\Upsilon_x$ is a vector space isomorphism and $\Upsilon^{-1}(f)=f(\vb_\tau)$ for all $f\in  \Hom(I_\tau,M)$. 
\end{lemma}

\begin{proposition}\label{propIntertwining_nonzero}
Let $\tau,\tau'\in T_\FC$. Then:\begin{enumerate}
\item $\Hom(I_\tau,I_{\tau'})\neq \{0\}$ if and only if $\tau'\in W^v.\tau$.

\item If $\tau\in T_\FC$ is regular, then \[I_\tau=\bigoplus_{w\in W^v} I_\tau(w.\tau)\] and for $w\in W$, one has \[\dim \mathrm{Hom}(I_{w.\tau},I_\tau)=\dim I_\tau(w.\tau)=1.\]
\end{enumerate} 
\end{proposition}

\begin{proof}
This is a consequence of \cite[Propositions 3.4, 3.10 and 3.5 (2)]{hebert2018principal}.
\end{proof}

\begin{lemma}\label{lemStructure_Weights_module}
Let $M$ be a $\AC_\FC$-module and $\tau\in \Wt(M)$. Let $w\in W^v$ be such that $I_{w.\tau}$ is isomorphic to $I_\tau$. Then $w.\tau\in \Wt(M)$. 
\end{lemma}

\begin{proof}
Let $x\in M(\tau)\setminus \{0\}$. Let $\phi:I_\tau\twoheadrightarrow M$ be defined by $\phi(h.\vb_\tau)=h.x$ for all $h\in \AC_\FC$. By \cite[Lemma 3.6]{hebert2018principal}, $\phi$ is well defined. Let $\psi:I_{w.\tau}\rightarrow I_\tau$ be an  isomorphism. Then $\phi\circ\psi\neq 0$. One has $I_{w.\tau}=\AC_\FC.\vb_{w.\tau}$ and thus  $\phi\circ\psi(\vb_{w.\tau})\in M(w.\tau)\setminus\{0\}$, which proves the lemma.
\end{proof}

 \begin{proposition}\label{propMaximal_proper_submodules}
 \begin{enumerate}
\item  Let $\tau\in T_\FC$ and $M$ be a proper submodule of $I_\tau$. Then there exists a maximal submodule $M'$ of $I_\tau$ containing $M$.

\item There exists an irreducible representation  $M$ of $\AC_\FC$ such that $\tau\in \Wt(M)$.

\item The map $M\mapsto I_\tau/M$, from the set of maximal submodules of $I_\tau$ to the set of isomorphism classes of irreducible representations admitting $\tau$ as a weight is surjective.

 \end{enumerate}

 \end{proposition}
 
 \begin{proof}
 Let $\MCC(M)$ be the set of proper submodules of $I_\tau$ containing $M$. Let $J$ be a totally ordered set and $(M_j)_{j\in J}$ be an increasing family of $\MCC(M)$. Then $\bigcup_{j\in J} M_j$ is a submodule of $I_\tau$ containing $M$. Moreover, $\vb_\tau\notin M_j$ for all $j\in J$ and thus $\vb_\tau\notin \bigcup_{j\in J} M_j$: $\bigcup_{j\in J} M_j\neq I_\tau$. By Zorn's lemma we deduce that $\MCC(M)$ admits a maximal element, which proves (1).
 
 Let $M$ be a maximal  submodule of $I_\tau$. Let $M'=I_\tau/M$. Then $M'$ is irreducible and the image of $\vb_\tau$ is a nonzero element of $M'(\tau)$, which proves (2).
 
 Let $M$ be an irreducible representation of $\AC_\FC$ admitting $\tau$ as a weight. By  \cite[Proposition 3.8]{hebert2018principal}, there exists a surjective morphism of $\AC_\FC$-modules $\phi:I_\tau\twoheadrightarrow M$. Then $M\simeq I_\tau/\ker(\phi)$ and $\ker(\phi)$ is a maximal submodule of $I_\tau$. 
 \end{proof}

 Set $B_s=\sigma_s H_s-\sigma_s^2\in \HC_{W^v,\FC}$\index{$B_s$}. One has $B_s^2=-(1+\sigma_s^2)B_s$. Let  $\zeta_{s}=-\sigma_sQ_s(Z)+\sigma_s^2\in \FC(Y)\subset \AC(T_\FC)$\index{$\zeta_s$}.  When $\sigma_s=\sigma_s'=\sqrt{q}$ for all $s\in \SCC$, we have $\zeta_{s}=\frac{1-qZ^{-\alpha_s^\vee}}{1-Z^{-\alpha^\vee_s}}\in \FC(Y)$. Let \[F_s=B_s+\zeta_{s}\in \AC(T_\FC)\index{$F_s$}.\]

Let  $\alpha^\vee\in \Phi^\vee$. Write $\alpha^\vee=w.\alpha_s^\vee$ for $w\in W^v$ and $s\in \SCC$. We set $\zeta_{\alpha^\vee}=(\zeta_s)^w$\index{$\zeta_{\alpha^\vee}$}.  

Let $\alpha^\vee\in\Phi^\vee$. Write $\alpha=w.\alpha_s^\vee$, with $w\in W^v$ and $s\in \SCC$. We set $\sigma_{\alpha^\vee}=\sigma_s$\index{$\sigma_{\alpha^\vee}$, $\sigma'_{\alpha^\vee}$} and $\sigma_{\alpha^\vee}'=w.\sigma_s'$. This is well defined by Lemma~\ref{lemKumar_1.3.14} and by the relations on the $\sigma_t$, $t\in \SCC$ (see Subsection~\ref{subIH algebras}).

 Let $w\in W^v$. Let $w=s_1\ldots s_r$ be a reduced expression of $w$. Set \[F_w=F_{s_r}\ldots F_{s_1}=(B_{s_r}+\zeta_{s_r})\ldots (B_{s_1}+\zeta_{s_1})\in \AC(T_\FC).\]\index{$F_w$} By the lemma below, this does not depend on the choice of the reduced expression of $w$.

\begin{lemma}\label{lemReeder 4.3}(see \cite[Lemma 5.14]{hebert2018principal}) Let $w\in W^v$. \begin{enumerate}

\item\label{itWell_definedness_Fw} The element $F_w\in \AC(T_\FC)$ is well defined, i.e it does not depend on the choice of a reduced expression for $w$.

\item\label{itLeading_coefficient_Fw} There exists $a\in \FC^*$ such that $F_w-aH_w\in \ATF^{<w}=\bigoplus_{v<w}H_v\FC(Y)$.

\item\label{itCommutation_relation} If $\theta\in \FC(Y)$, then $\theta*F_w=F_w*{^{w^{-1}}}\theta$.

\item\label{itDomain_Fw} If $\tau\in T_\FC$ is such that $\zeta_{\beta^\vee}\in \FC(Y)_\tau$ for all $\beta^\vee\in N_{\Phi^\vee}(w)$, then $F_w\in \AC(T_\FC)_\tau$ and $F_w(\tau).\vb_\tau \in I_\tau(w.\tau)$.

\item\label{itF_w_well_defined_at_regular} Let $\tau\in T^{\mathrm{reg}}_\FC$. Then $F_w\in \AC(T_\FC)_\tau$.

\end{enumerate}
\end{lemma}

\section{Weighted representations of $\AC_\FC$ and $\HC_\FC$}\label{secWeighted_representations}

In this section, the field $\FC$ is not necessarily $\C$. We set $\AC_\FC^\emptyset=\AC_\FC$ and $\AC_\FC^+=\HC_\FC$. Let $\epsilon\in \{+,\emptyset\}$.  A $\FC[Y^\epsilon]$-module $M$ is called \textbf{weighted} if for all $x\in M$, $\FC[Y^\epsilon].x$ is a finite dimensional. A $\AC_\FC^\epsilon$-module is called \textbf{weighted} if the induced $\FC[Y^\epsilon]$-module is. 
 
 In this section, we   characterize the weighted representations of $\HC_\FC$ which can be extended to a representation of $\AC_\FC$ (see Proposition~\ref{propCharacterization_restriction_weighted_representations}) . We also prove that if $M$ is a weighted representation of $\AC_\FC$, then the $\HC_\FC$ submodules of $M$ are exactly the restrictions to $\HC_\FC$ of the $\AC_\FC$-submodules of $M$ (see Proposition~\ref{propWeighted_representations_AC_HC}).

\begin{proposition}\label{propWeighted_representations_AC_HC}
Let $M$ be a weighted $\AC_\FC$-module. Then a subset $M'\subset M$ is an $\HC_\FC$-submodule of $M$ if and only if it is a $\AC_\FC$-submodule of $M$.
\end{proposition}

\begin{proof}
Let $M'\subset M$ be an $\HC_\FC$-submodule of $M$. Let $x\in M'$ and $M'_x=\FC[Y^+].x$. For $\lambda\in Y^+$, define $\phi_{\lambda,x}:M'_x\rightarrow M'_x$ by $\phi_{\lambda,x}(y)=Z^\lambda.y$, for $y\in M'_x$. Then $\phi_{\lambda,x}$ is injective and as $M'_x$ is finite dimensional, $\phi_{\lambda,x}$ is an isomorphism. Let $y=(\phi_{\lambda,x})^{-1}(x)$.  Then $Z^\lambda.y=x$ and $y=Z^{-\lambda}.x\in M'_x$. Let $\mu\in Y$. By writing $\mu=\lambda_+-\lambda_-$, with $\lambda_+,\lambda_-\in Y^+$, we deduce that $Z^\mu.x\in M'_x\subset M$. Therefore $M$ is stable under the action of $\HFW$ and of $\FC[Y]$ and hence $M$ is a $\AC_\FC$-module.
\end{proof}

\begin{proposition}\label{propCharacterization_restriction_weighted_representations}
Let $M$ be a weighted representation of $\HC_\FC$. Then the following are equivalent:\begin{enumerate}
\item $M$ is the restriction of a representation of $\AC_\FC$,

\item $Z^\lambda.x\neq 0$, for all $\lambda\in Y^+$ and $x\in M\setminus\{0\}$,

\item for every $\tau\in \Wt(M)$, $\tau(Y^+)\subset \FC^*$.
\end{enumerate}
\end{proposition}

The condition is necessary because if $M$ is a $\AC_\FC$-module, one has $x=Z^{-\lambda}.Z^\lambda.x$, for all $\lambda\in Y$. In the sequel of this section, we prove that this condition is indeed sufficient. The idea of our proof is to extend the action of $\FC[Y^+]$ to an action of $\FC[Y]$ and then to define an action of $\AC_\FC$. The difficulty is then to prove that it is indeed an action, i.e that $(h*h').x=h.(h'.x)$ for every $h,h'\in \AC_\FC$ and $x\in M$.

\begin{lemma}\label{lemDecomposition_weighted_modules}
Let $\epsilon\in \{\emptyset,+\}$ and $M$ be a weighted $\FC[Y^\epsilon]$-module. Then $M=\bigoplus_{\tau\in \Wt(M)} M(\tau,\mathrm{gen})$.
\end{lemma}

\begin{proof}
One has $\sum_{\tau\in \Wt(M)} M(\tau,\mathrm{gen})=\bigoplus_{\tau\in \Wt(M)} M(\tau,\mathrm{gen})$.

 Let $x\in M$ and $M_x=\FC[Y^\epsilon].x$. Then by \cite[Lemma 3.1]{hebert2018principal}, \[x\in M_x=\bigoplus_{\tau\in \Wt(M_x)}M_x(\tau,\mathrm{gen})\subset   \bigoplus_{\tau\in \Wt(M)}M(\tau,\mathrm{gen}).\] Thus $\bigoplus_{\tau\in \Wt(M)}M(\tau,\mathrm{gen})=M$, which proves the lemma.
\end{proof}

\begin{lemma}\label{lemInjectivity_multiplication_weighted_modules}
Let $M$ be a weighted $\FC[Y^+]$-module. Suppose that there exists $\lambda\in Y^+$ and $x\in M\setminus\{0\}$ such that $Z^\lambda.x=0$. Then there exists $\tau\in \Wt(M)$ such that $\tau(\lambda)=0$.
\end{lemma}

\begin{proof}
By Lemma~\ref{lemDecomposition_weighted_modules}, one can assume that $x\in M(\tau,\mathrm{gen})$, for some $\tau\in \Wt(M)$. Let $k\in \Ne$ be such that $(Z^\lambda-\tau(\lambda)\Id)^k.x=0$. Then $(Z^\lambda-\tau(\lambda)\Id)^k.x=\sum_{j=0}^k {k\choose{j}}\tau(\lambda)^j Z^{(k-j)\lambda}.x=\tau(\lambda)^k.x=0$ and thus $\tau(\lambda)=0$.
\end{proof}

\begin{lemma}\label{lemExtension_action_CY+}
Let $M$ be a weighted $\FC[Y^+]$-module. Suppose that for all $\tau\in \Wt(M)$, $\tau(Y^+)\subset \FC^*$. Then there exists a unique action of $\FC[Y]$ on $M$ which induces the action of  $\FC[Y^+]$ on $M$.
\end{lemma}

\begin{proof}
We begin by proving the uniqueness of such an action. Suppose that we can extend $.$ to $\FC[Y]$. For $\lambda\in Y$, define $\phi_\lambda:M\rightarrow M$ by $\phi_\lambda(m)=Z^\lambda.m$, for $\lambda\in Y^+$ and $m\in M$. Let $\lambda\in Y^+$. Then $\phi_\lambda$ is a bijection and its inverse is $\phi_{-\lambda}$. Let now $\mu\in Y$ and $m\in M$. Write $\mu=\lambda_+-\lambda_-$, with $\lambda_+,\lambda_-\in Y^+$. Then $Z^\mu.m=\phi_{\lambda_+}\big((\phi_{\lambda_-})^{-1}(m)\big)$, which proves the uniqueness of such an action.

Suppose now that for every $\tau\in \Wt(M)$, $\tau(Y^+)\subset \FC^*$.  Let $\lambda\in Y^+$ and $x\in M$. Let $M_x=\FC[Y^+].x$ and $\phi_{\lambda,x}:M_x\rightarrow M_x$ be defined by $\phi_{\lambda,x}(y)=Z^\lambda.y$ for all $y\in M_x$. Then by Lemma~\ref{lemInjectivity_multiplication_weighted_modules}, $\phi_{\lambda,x}$ is injective and by assumption, $M_x$ is finite dimensional. Thus $\phi_{\lambda,x}$ is an isomorphism. Thus the map $\phi_\lambda:M\rightarrow M$ defined by $\phi_\lambda(x)=Z^\lambda.x$ for all $x\in M$ is surjective. By Lemma~\ref{lemInjectivity_multiplication_weighted_modules}, $\phi_\lambda$ is an isomorphism. One sets $\phi_{-\lambda}=\phi_\lambda^{-1}$. Then $(\phi_\mu)_{\mu\in Y^+}$ is commutative and thus $(\phi_{\pm \mu})_{\mu\in Y^+}$ is commutative.  If $\mu\in Y$, $\mu=\mu_+-\mu-$, with $\mu_-,\mu_+\in Y^+$, one sets $\phi_{\mu}=\phi_{\mu_+}\circ\phi_{-\mu_-}$. Then $\phi_{\mu}$ does not depend on the choice of $\mu_-,\mu_+$ such that $\mu_+-\mu_-=\mu$ and $(\phi_{\mu})_{\mu\in Y}$ is commutative. One has $\phi_{\lambda}\circ \phi_{\mu}=\phi_{\lambda+\mu}$ for all $\lambda,\mu\in Y$. For $\mu\in Y$, one sets $Z^\mu.x=\phi_{\mu}(x)$. Then $(Z^\lambda.Z^\mu).x=Z^\lambda.(Z^\mu.x)$ for all $\lambda,\mu\in Y$, and thus this defines an action of $\FC[Y]$ on $M$. 
\end{proof}

We now fix a weighted representation $M$ of $\HC_\FC$ such that $Z^\lambda.x\neq 0$ for all $\lambda\in Y^+$ and $x\in M\setminus\{0\}$. Using Lemma~\ref{lemExtension_action_CY+}, we equip $M$ with the structure of an $\FC[Y]$-module. For $h=\sum_{\lambda\in Y,w\in W^v}a_{w,\lambda} H_wZ^\lambda\in \AC_\FC$, and $x\in M$, one sets $h.x=\sum_{\lambda\in Y,w\in W^v} a_{w,\lambda} H_w.(Z^\lambda.x)$. We now prove that $M$ is a $\AC_\FC$-module by proving that for all $h,h'\in \AC_\FC$ and all $x\in M$, one has $(h*h').x=h.(h'.x)$.

Let $w\in W^v$ and $\lambda\in Y$. Using Lemma~\ref{lemCommutation relation}, we write  $Z^\lambda*H_w=\sum_{v\leq w} H_v *R_{v,w}^\lambda$, where $(R_{v,w}^\lambda)_{v\leq w}\in \FC[Y]^{[1,w]}$.

\begin{lemma}\label{lemTranslation_associativity_R}
Let  $\lambda,\nu\in Y$ and $w\in W^v$. Then: \[\sum_{u\leq v\leq w}H_u* R_{u,v}^{\nu}*R_{v,w}^\lambda=\sum_{v\leq w} H_v* R_{v,w}^{\nu+\lambda}.\]
\end{lemma}

\begin{proof}
This follows from the associativity of $*$: $Z^{\nu}*(Z^\lambda*H_w)=Z^{\nu+\lambda}*H_w$.
\end{proof}

\begin{lemma}\label{lemFirst_step_associativity_extension}
Let $x\in M$, $w\in W^v$ and $\lambda\in Y$. Then $(Z^\lambda*H_w).x=Z^\lambda.(H_w.x)$.
\end{lemma}

\begin{proof}
Let $\nu\in Y^+$. Then: \[\begin{aligned} Z^{\nu}.\big((Z^\lambda*H_w).x\big) & = Z^\nu.\big(\sum_{v\leq w} H_v .(R_{v,w}^\lambda.x)\big)\\
& = \sum_{v\leq w} (Z^\nu*H_v).(R_{v,w}^\lambda .x)\\
&= \sum_{ v\leq w}\sum_{u\leq v} (H_u*R_{u,v}^\nu*R_{v,w}^\lambda).x\\
&= \sum_{v\leq w} (H_v*R_{v,w}^{\nu+\lambda} ).x.\end{aligned}\] 

We now assume that  $\nu+\lambda\in Y^+$. Such a $\nu$ exists. Indeed, one can choose $\nu'\in Y\cap C^v_f$ and take $\nu=N\nu'$, for $N\in \N$ large enough. Then \[Z^{\nu}.\big(Z^{\lambda}.(H_w.x)\big)=Z^{\nu+\lambda}.(H_w.x)=(\sum_{v\leq w} H_v*R_{v,w}^{\nu+\lambda}) .x=Z^{\nu}.\big((Z^\lambda*H_w).x\big).\] Therefore $(Z^\lambda.H_w).x=Z^\lambda.(H_w.x)$.
\end{proof}

We can now prove Proposition~\ref{propCharacterization_restriction_weighted_representations}.

We have to prove that for all $h,h'\in \AC_\FC$, and $x\in M$, one has $(h*h').x=h.(h'.x)$. Let $u,v\in W^v$, $\lambda,\mu\in Y$ and $x\in M$. Write $Z^\lambda*H_v=\sum_{w\in W^v,\nu\in Y} a_{w,\nu} H_wZ^\nu$. Then $(H_uZ^\lambda)*(H_vZ^\mu)=\sum_{w\in W^v,\nu\in Y}a_{w,\nu} H_u*H_w*Z^{\mu+\nu}$. Therefore \[\begin{aligned} (H_uZ^\lambda*H_vZ^\mu).x & =\sum_{w\in W^v,\nu\in Y}a_{w,\nu}H_u*H_w.(Z^{\nu+\mu}.x) \\
 & =  H_u.\bigg(\sum_{w\in W^v,\mu\in Y} a_{w,\nu} H_w.\big(Z^\nu.(Z^\mu.x)\big)\bigg)\\
 &= H_u.\big((Z^\lambda*H_v).(Z^\mu .x)\big)\\
 &=H_u.\bigg(Z^\lambda.\big(H_v.\big(Z^\mu.x)\big)\bigg)\text{ by Lemma~\ref{lemFirst_step_associativity_extension}}\\
 &= (H_u*Z^\lambda).\big( H_v.\big(Z^\mu.x)\big) \\
&=  (H_u*Z^\lambda).\big((H_v*Z^\mu).x\big),\end{aligned}\] which proves the proposition.

\begin{remark}
By \cite[Lemma 4.5]{hebert2018principal}, if $\T=\mathring{\T}\cup \bigcap_{s\in \SCC}\ker(\alpha_s)$, then for every nonzero algebra morphism $\tau:\FC[Y^+]\rightarrow \FC$,  one has $\tau(Y^+)\subset \FC^*$. Therefore in this case, every weighted representation of $\HC_\FC$ extends to a representation of $\AC_\FC$. This is the case for example when $\HC_\FC$ is associated to an affine Kac-Moody group or to a size $2$ Kac-Moody matrix. By \cite[Lemma 4.9]{hebert2018principal}, there exist Kac-Moody matrices for which there exist weighted representations of $\HC_\FC$ which do not extend to representations of $\AC_\FC$. 
\end{remark}

\section{Decomposition of regular principal series representations}\label{secRegular_representations}

In this section, the field $\FC$ is not necessarily $\C$. Let $\tau\in T_\FC$. We call $\tau$ \textbf{regular} if $W_\tau=\{1\}$, that is if for all $w\in W^v$, $w.\tau=\tau$ implies $w=1$. Let $\tau\in T_\FC$ be regular. In this section, we describe the submodules of $I_\tau$ and prove that there exists a unique  irreducible representation of $\AC_\FC$ admitting $\tau$ as a weight. The main tools that we use are the weights of the sumodules and the intertwining operators $I_{w.\tau}\rightarrow I_{w'.\tau}$, for $w,w'\in W^v$.

In subsection~\ref{subDiagram_and_distance}, we introduce, for $\tau\in T_\FC$ regular, the graph of $\tau$, whose vertices are the $I_{w.\tau}$, for $w\in W^v$ and a semi-distance on it.

In subsection~\ref{subIrreducible_constituents}, we study the irreducible representations admitting $\tau$ as a weight.

In subsection~\ref{subIndecomposable_submodules}, we study the strongly indecomposable submodules of $I_\tau$ and prove that the sumbodules of $I_\tau$ can be written as sums of strongly indecomposable submodules.

In subsection~\ref{subWeights_submodules}, we give a way to compute the weights of a submodule.

In subsection~\ref{subExamples_regular_case}, we apply the results of this section to some examples.

\subsection{Graph and semi-distance associated to $I_\tau$}\label{subDiagram_and_distance}

Let $\tau\in T_\FC$ be regular. By Lemma~\ref{lemFrobenius_reciprocity} and Proposition~\ref{propIntertwining_nonzero}, one has $\dim \Hom(I_{w.\tau},I_{w'.\tau})=1$, for all $w,w'\in W^v$. For every $w,w'\in W^v$, we fix $A_{w,w',\tau}\in \Hom(I_{w.\tau},I_{w'.\tau})\setminus\{0\}$. 

The \textbf{graph of morphisms} $\GC_\tau$ of $\tau$ is the non-oriented graph defined as follows. Its vertices are the $I_{w.\tau}$, for $w\in W^v$.  Two vertices $I_{w.\tau}$, $I_{w'.\tau}$ are  joined by an edge if and only if $\ell(w'^{-1}w)=1$.

A \textbf{path in $\GC_\tau$} is a finite  sequence $(I_{w_i.\tau})_{i\in \llbracket 1,n\rrbracket}\in (\GC_\tau)^n$, where $n\in \N$ and $w_iw_{i+1}^{-1}\in \SCC$ for all $i\in \llbracket 1,n-1\rrbracket$. Let $n\in \N$ and $\Gamma=(I_{w_1.\tau},\ldots,I_{w_n.\tau})$ be a path in $\GC_\tau$. We say that $\Gamma$ \textbf{is an intertwining path} if $A_\Gamma:=A_{w_{n-1},w_{n},\tau}\circ \ldots \circ A_{w_{1},w_2,\tau} :I_{w_1.\tau}\rightarrow I_{w_n.\tau}$ is nonzero.

The \textbf{graph of isomorphisms} $\tilde{\GC}_\tau$ is the graph obtained  from $\GC_\tau$ by deleting the edges $(I_{w.\tau},I_{sw.\tau})$,  $w\in W^v$, $s\in \SCC$ such that $A_{w,sw,\tau}$ is not an isomorphism (this is equivalent to assuming that $A_{sw,w,\tau}$ is not an isomorphism since $\dim \Hom(I_{w.\tau},I_{sw.\tau})=1$). Note that by  \cite[Lemma 5.4]{hebert2018principal}, $A_{w,sw,\tau}$ is an isomorphism if and only $w.\tau(\zeta_s)w.\tau(^s\zeta_s)\neq 0$ (or equivalently $w.\tau(\alpha_s^\vee)\notin\{q,q^{-1}\}$ in the split case).

If $\Gamma$ is an path in $\GC_\tau$, we set \[\ell_{\not\simeq}(\Gamma) =|\{i\in \llbracket 1,n-1\rrbracket|A_{w_{i},w_{i+1}.\tau})\text{ is not an isomorphim}\}|.\]

If $P_1,P_2$ are two vertices of $\GC_\tau$, then we set $d(P_1,P_2)=\ell_{\not\simeq}(\Gamma)$, where $\Gamma$ is any intertwining path joining $P_1$ to $P_2$. The aim of this subsection is to prove that this is well defined. For this we prove the following:\begin{itemize}
\item there exists an intertwining path $\Gamma$  joining $P_1$ to $P_2$ (see  Proposition~\ref{propExistence_intertwining_paths})

\item $\ell_{\not\simeq}(\Gamma)$ is independent of the choice of such a path (see Proposition~\ref{propDefinition_tau_distance}).
\end{itemize}

 Our proof is based on the ``word property'' in Coxeter groups. Note that  we will prove that $d$ is symmetric and satisfies the triangle inequality (see Proposition~\ref{propDefinition_tau_distance}), but in general, it is not a distance (for example if $I_{w.\tau}\simeq I_{\tau}$ for every $w\in W^v$, then $d(P,P')=0$, for every $P,P'\in \GC_\tau$). However it induces a distance on the set of connected components of $\tilde{\GC}_\tau$. This semi-distance will enable us to study the strongly indecomposable submodules of $I_\tau$.

\subsubsection{Existence of intertwining paths between two vertices}

We begin by proving the existence of intertwining paths between any two vertices of $\GC_\tau$. 

Recall that $\HFW=\bigoplus_{w\in W^v} \FC H_w\subset \AC_\FC$. For $w\in W^v$, set $\HC_{W^v,\FC}^{\leq w}=\bigoplus_{v\leq w} \FC H_v$ and $\HC_{W^v,\FC}^{< w}=\bigoplus_{v<w}\FC H_v$.

\begin{lemma}\label{lemProduct in the Coxeter IH algebra} Let $w\in W^v$ and $s\in \SCC$ be such that $ws>w$. Then: \[(\HC_{W^v,\FC}^{\leq w}\setminus\HC_{W^v,\FC}^{<w})*(\HC_{W^v,\FC}^{\leq s}\setminus\HC_{W^v,\FC}^{<s})\subset \HC_{W^v,\FC}^{\leq ws}\setminus\HC_{W^v,\FC}^{<ws}.\]
\end{lemma}

\begin{proof}
This follows from the fact that $[1,w].[1,s]\subset [1,ws]$ and that $[1,w).s\cup [1,w]\subset [1,ws]$.
\end{proof}

\begin{proposition}\label{propExistence_intertwining_paths}(see  \cite[(1.21)]{kato1982irreducibility}).
Let $\tau\in T_\FC$ be regular. Let $w\in W^v$ and $w=s_k\ldots s_1$ be a reduced writing of $w$, where $k\in \N$ and $s_1,\ldots,s_k\in \SCC$. For $j\in \llbracket 1,k\rrbracket$, set $w_j=s_{j-1}\ldots s_1$ (where we set $s_0\ldots s_1=1$) and $\tau_j=w_j.\tau$. Then $\Gamma=(I_{\tau_1},I_{\tau_2},\ldots,I_{\tau_k})$ is an intertwining path joining $I_\tau$ to $I_{w.\tau}$.
\end{proposition}

\begin{proof}
 
Let $j\in \llbracket 1,k-1\rrbracket$. By Lemma~\ref{lemReeder 4.3} (\ref{itDomain_Fw}) and (\ref{itF_w_well_defined_at_regular}), \[x_j:= F_{s_j}(s_j.\tau_j)\vb_{s_j.\tau_j}=F_{s_j}(\tau_{j+1})\vb_{\tau_{j+1}}\in I_{\tau_{j+1}}(\tau_j).\]

For $\tau'\in T_\FC$ and $w\in W^v$, set $I_{\tau'}^{\leq w}=\bigoplus_{v\leq w} \FC \vb_{\tau'}$ and $I_{\tau'}^{< w}=\bigoplus_{v< w} \FC \vb_{\tau'}$.

Set $f_j=\Upsilon_{x_j}\circ \ldots \circ \Upsilon_{x_1}\in \Hom(I_{\tau},I_{\tau_{j+1}})$
 (where the $\Upsilon_{x_j}:I_{\tau_j}\rightarrow I_{\tau_{j+1}}$ are defined in Lemma~\ref{lemFrobenius_reciprocity}). Let $\PC_j:$ ``$f_j(\vb_\tau)\in I_{\tau_{j+1}}^{\leq w_{j+1}^{-1}}\setminus I_{\tau_{j+1}}^{<w_j^{-1}}$''.
  Then $\PC_1$ is true by Lemma~\ref{lemReeder 4.3} (\ref{itLeading_coefficient_Fw}). Let $j\in \llbracket 1,k-2\rrbracket$ and assume that $\PC_j$ is true. Write $f_j(\vb_\tau)=h.\vb_{\tau_{j+1}}$, where $h\in \HC_{\FC,W^v}^{\leq w_{j+1}^{-1}}\setminus\HC_{\FC,W^v}^{< w_{j+1}^{-1}}$. Then one has $\Upsilon_{x_{j+1}}\big(f_j(\vb_\tau)\big)=h.\Upsilon_{x_{j+1}}(\vb_{\tau_{j+1}})$. Write $x_{j+1}=h'. \vb_{\tau_{j+2}}$, where $h'\in (\HC_{\FC,W^v}^{\leq s_{j+1}}\setminus \FC). \vb_{\tau_{j+1}}$. Then $f_{j+1}(\vb_{\tau})=h.h'.\vb_{\tau_{j+1}}$.
   By Lemma~\ref{lemProduct in the Coxeter IH algebra}, we deduce that $\PC_{j+1}$ is true. Thus $\PC_{k-1}$ is true and in particular, $f_{k-1}(\vb_\tau)\neq 0$, which proves the lemma.
\end{proof}

\subsubsection{Independence of the choice of a path}

We now prove that if $I_{\tau_1},I_{\tau_2}$ are two vertices of $\GC_\tau$ and $\Gamma, \Gamma'$ are intertwining paths joining them, then $\ell_{\not\simeq}(\Gamma)=\ell_{\not\simeq}(\Gamma')$. 

Let $(W^v)^*=\SCC^{(\N)}$. For $w^*=(s_1,\ldots,s_k)\in (W^v)^*$, we set $\pi(w^*)=s_1\ldots s_k\in W^v$. For $s,t\in \SCC$ denote by $m(s,t)$ the order of $st$ in $W^v$. If $m(s,t)$ is finite, we denote by $w_0^*(s,t)$ the $m(s,t)$-tuple $(s,t,s,t\ldots )$. One has $\pi\big(w^*_0(s,t)\big)=\pi\big(w_0^*(t,s)\big)=:w_0(s,t)$. If $w\in W^v$ and $\Gamma=(I_{w_1.\tau},I_{w_2.\tau},\ldots,I_{w_n.\tau})$ is a path,  we set $\Gamma^*=(w_2w_1^{-1},\ldots,w_nw_{n-1}^{-1})$. Let  $w^*,\tilde{w}^*\in (W^v)^*$. We say that $\tilde{w}^*$\textbf{ is obtained from $w^*$ by a braid-move} if there exist $s,t\in \SCC$ such that $m(s,t)$ is finite 	and $u^*,v^*\in (W^v)^*$ such that \[w^*=(u^*,w_0^*(s,t),v^*)\text{ and }\tilde{w}^*=(u^*,w_0^*(t,s),v^*).\]

\begin{lemma}\label{lemConservation_tau_distance_braid_moves1}
Let $\tau\in T_\FC$ be regular.  Let $s,t\in \SCC$ be such that $s\neq t$ and $m(s,t)$ is finite. Let  \[\Gamma_s=(I_\tau,I_{s.\tau},I_{ts.\tau},\ldots,I_{w_0(s,t).\tau})\mathrm{\ and\ }\Gamma_t=(I_{\tau},I_{t.\tau},I_{st.\tau},\ldots,I_{w_0(s,t).\tau}).\] Then $\ell_{\not\simeq}(\Gamma_s)=\ell_{\not\simeq}(\Gamma_t)$ and   $A_{\Gamma_s}\in \FC^* A_{\Gamma_t}$.
\end{lemma}

\begin{proof}
Write $\Gamma_s=(I_{\tau_1^{(s)}},\ldots,I_{\tau_k^{(s)}})$ and $\Gamma_t=(I_{\tau_1^{(t)}},\ldots,I_{\tau_k^{(t)}})$. 
Let us prove that  \begin{equation}\label{eqCondition_nonisomorphism}
I_\tau=I_{\tau_1^{(s)}}\simeq I_{\tau_2^{(s)}}=I_{s.\tau}\text{ if and only if } I_{\tau_{k}^{(t)}}\simeq I_{\tau_{k-1}^{(t)}}.\end{equation}
For $u\in \SCC$, set \[R_u(T)=-\sigma_u\frac{(\sigma_u-\sigma_u^{-1})+(\sigma'_u-\sigma_u'^{-1})T}{1-T^2}+\sigma_u^2\in \FC(T),\] where $T$ is an indeterminate.  
By \cite[Lemma 5.4]{hebert2018principal} , for $\tau'\in T_\FC$ and $u\in \SCC$, $I_{\tau'}$ is not isomorphic to $I_{u.\tau'}$ 
if and only if $R_u\big(\tau'(\alpha_u^\vee)\big)R_u\big(\tau'(-\alpha_u^\vee)\big)=0$.

By \cite[1.3.21 Proposition]{kumar2002kac}, $m(s,t)\in \{2,3,4,6\}$. Suppose $m(s,t)=3$.  Then $k=4$, $\tau_k^{(t)}=tst.\tau$ and $\tau_{k-1}^{(t)}=st.\tau$. Thus $I_{\tau_{k}^{(t)}}\simeq I_{\tau_{k-1}^{(t)}}$ if and only if  $R_t\big(st.\tau(\alpha_t^\vee)\big)R_t\big(st.\tau(-\alpha_t^\vee)\big)=0$. Moreover $sts=tst$, thus $s$ and $t$ are conjugate and hence $R_s=R_t$, by assumptions on the $\sigma_u,\sigma_u'$, $u\in \SCC$. By Lemma~\ref{lemDefinition_root_reflection}, $ts.\alpha_t^\vee=\alpha_s^\vee$, which proves~(\ref{eqCondition_nonisomorphism}).

Suppose $m(s,t)$ is even. Set $w=sw_0(s,t)=w_0(s,t)s$. Then $k=m(s,t)+1$, $\tau_k^{(t)}=w_0(s,t).\tau$ and $\tau_{k-1}^{(t)}=w.\tau$.   Thus $I_{\tau_{k}^{(t)}}\simeq I_{\tau_{k-1}^{(t)}}$ if and only if $R_s\big(w.\tau(\alpha_s^\vee)\big)R_s\big(w.\tau(-\alpha_s^\vee)\big)=0$.  Moreover, $w=w^{-1}$ and $ws=sw=w_0(s,t)$. Thus $wsw^{-1}=s$ and thus by Lemma~\ref{lemDefinition_root_reflection}, $w.\alpha_s^\vee=\alpha_s^\vee$, which proves~(\ref{eqCondition_nonisomorphism}).

We deduce that~(\ref{eqCondition_nonisomorphism}) holds in both cases. By applying~(\ref{eqCondition_nonisomorphism}) to \[\Gamma_s(v):=(I_{v.\tau},I_{sv.\tau},I_{tsv.\tau},\ldots,I_{w_0(s,t)v.\tau})\text{ and }\Gamma_t(v):=(I_{v.\tau},I_{tv.\tau},I_{stv.\tau},\ldots,I_{w_0(s,t)v.\tau}),\] for every $v\in \langle s,t\rangle$, we deduce that $\ell_{\not\simeq }(\Gamma_s)=\ell_{\not\simeq }(\Gamma_t)$. By Proposition~\ref{propExistence_intertwining_paths}, $\Gamma_s$ and $\Gamma_t$ are intertwining paths and as $\dim \Hom (I_{v.\tau}, I_{w_0(s,t)v.\tau})=1$, one has   $\FC^* A_{\Gamma_s}=\FC^* A_{\Gamma_t}$.
\end{proof}

We deduce the following lemma:

\begin{lemma}\label{lemConservation_tau_distance_braid_moves}
Let $w,w'\in W^v$ and $\Gamma,\tilde{\Gamma}$ be two paths joining $I_{w.\tau}$ to $I_{w'.\tau}$. We assume that $\tilde{\Gamma}^*$ is obtained from $\Gamma^*$ by a braid-move. Then $\ell_{\not\simeq}(\Gamma)=\ell_{\not\simeq}(\tilde{\Gamma})$ and $\Gamma$ is an intertwining path if and only if $\tilde{\Gamma}$ is an intertwining path.
\end{lemma}

Let $w^*,\tilde{w}^*\in (W^v)^*$. We say that $\tilde{w}^*$ \textbf{is obtained from $w^*$ by a nil-move} if there exist $u^*,v^*\in (W^v)^*$ and $s\in \SCC$ such that $w^*=(u^*,s,s,v^*)$ and $\tilde{w}^*=(u^*,v^*)$.

\begin{lemma}\label{lemComposition_intertwinning_morphisms_nonisomorphism}
Let $\tau\in T_\FC$ be regular. Let $u,v\in W^v$ be such that $I_{u.\tau}$ and $I_{v.\tau}$ are not isomorphic. Then $A_{u,v,\tau}\circ A_{v,u,\tau}=0$.
\end{lemma}

\begin{proof}
We have $A_{u,v,\tau}\circ A_{v,u,\tau}\in \mathrm{End}(I_{v.\tau})=\FC \Id$ and $A_{v,u,\tau}\circ A_{u,v,\tau}\in \mathrm{Hom}(I_{u.\tau})=\FC \Id$. Write $A_{u,v,\tau}\circ A_{v,u,\tau}=\gamma \Id$ and $A_{v,u,\tau}\circ A_{u,v,\tau}=\gamma'\Id$, with $\gamma,\gamma'\in \FC$. As $I_{u.\tau}$ and $I_{v.\tau}$ are not isomorphic, we have $\gamma\gamma'=0$. Exchanging $u$ and $v$ if necessary, we may assume $\gamma=0$. Then $A_{v,u,\tau}\circ A_{u,v,\tau}\circ A_{v,u,\tau}=0=\gamma' A_{v,u,\tau}$ and hence $\gamma'=0$, which proves the lemma.
\end{proof}

\begin{lemma}\label{lemConservation_tau_distance_nil_move}

Let $w,w'\in W^v$ and $\Gamma,\tilde{\Gamma}$ be two paths joining $I_{w.\tau}$ to $I_{w'.\tau}$. We assume that $\Gamma$ is an intertwining path and that $\tilde{\Gamma}^*$ is obtained from $\Gamma^*$ by a nil-move. Then $\tilde{\Gamma}$ is an intertwining path and  $\ell_{\not\simeq}(\Gamma)=\ell_{\not\simeq}(\tilde{\Gamma})$. 
\end{lemma}

\begin{proof}
Write $\Gamma^*=(u^*,v^*)$ and $\tilde{\Gamma}^*=(u^*,s,s,v^*)$, with $u^*,v^*\in (W^v)^*$ and $s\in \SCC$. Set $u=\pi(u^*)$ and $\pi(v^*)$. 

As $\Gamma$ is an intertwining path, one has: \[A_{u^{-1}w,v^{-1}u^{-1}w,\tau}\circ A_{su^{-1}w,u^{-1}w,\tau}\circ A_{u^{-1}w,su^{-1}w,\tau}\circ A_{w,u^{-1}w,\tau}\neq 0.\]  By Lemma~\ref{lemComposition_intertwinning_morphisms_nonisomorphism} we deduce that $A_{su^{-1}w,u^{-1}w,\tau}$ and $A_{u^{-1}w,su^{-1}w,\tau}$ are isomorphisms and that $A_{su^{-1}w,u^{-1}w,\tau}\circ A_{u^{-1}w,su^{-1}w,\tau}\in \FC^* \Id$. Therefore $A_{u^{-1}w,v^{-1}u^{-1}w,\tau}\circ A_{w,u^{-1}w,\tau}\neq 0,$ and the lemma follows.
\end{proof}

If $w\in W^v$, we denote by $\tilde{C}(w)$ the connected component of $\tilde{\GC}_\tau$ containing $I_{w.\tau}$.

\begin{proposition}\label{propDefinition_tau_distance}
\begin{enumerate}

\item\label{itWell_definedness_distance} Let $w,w'\in W^v$. Then if $\Gamma_1,\Gamma_2$ are two intertwining paths joining $I_{w'.\tau}$ to $I_{w.\tau}$, one has $\ell_{\not\simeq}(\Gamma_1)=\ell_{\not\simeq}(\Gamma_2)$.

\item If $w\in W^v$ then $\tilde{C}(w)=\{I_{w'.\tau}|w'\in W^v| I_{w'.\tau}\simeq I_{w.\tau}\}$.

\item If $w,w'\in W^v$, then every path $\Gamma$ in $\GC_\tau$ joining $I_{w.\tau}$ to $I_{w'.\tau}$ satisfies $\ell_{\not\simeq}(\Gamma)\geq d(I_{w.\tau},I_{w'.\tau})$.

\item\label{itDistance_indeed_distance} The map $d:\GC_\tau\times \GC_\tau\rightarrow \N$ is symmetric and satisfies the triangle inequality. Moreover it induces a distance on the set $\tilde{\mathcal{C}}$ of connected components of $\tilde{\GC}_\tau$, by setting $d\big(\tilde{C}(w),\tilde{C}(w')\big)=d(I_{w.\tau},I_{w'.\tau})$ for $w,w'\in W^v$. 
\end{enumerate}
\end{proposition}

\begin{proof}
(1) By the word property (\cite[Theorem 3.3.1]{bjorner2005combinatorics}), there exist $n_1,n_2\in \N$ and sequences $\Gamma_1^{(1)}=\Gamma_1,\ldots,\Gamma_1^{(n_1)}$, $\Gamma_2^{(1)}=\Gamma_2,\ldots,\Gamma_2^{(n_2)}$ of paths such that $\Gamma_1^{(n_1)}=\Gamma_2^{(n_2)}$ and for all $i\in \{1,2\}$ and $j\in \llbracket 1,n_i-1\rrbracket$, $\Gamma_i^{(j+1)}$ is obtained from $\Gamma_i^{(j)}$ by a nil-move or a braid-move and such that $\Gamma_i^{(n_i)}$ has length $\ell (w'w^{-1})$. Then by Lemma~\ref{lemConservation_tau_distance_braid_moves} and Lemma~\ref{lemConservation_tau_distance_nil_move}, $\ell_{\not\simeq}(\Gamma_1)=\ell_{\not\simeq}(\Gamma_1^{(n_1)})=\ell_{\not\simeq}(\Gamma_2)=\ell_{\not\simeq}(\Gamma_2^{(n_2)})$, which  proves~(\ref{itWell_definedness_distance}).

(2) Let $I_{w_1.\tau}\in \tilde{C}(w)$. Then there exists a path $\Gamma$ from $I_{w_1.\tau}$ to $I_{w.\tau}$ composed uniquely of isomorphisms and thus $I_{w_1.\tau}\simeq I_{w.\tau}$. Let $w_1\in W^v$ be such that $I_{w_1.\tau}$ is isomorphic to $I_{w.\tau}$. Let $\Gamma$ be an intertwining path joining $w_1.\tau$ to $w.\tau$, which exists by Proposition~\ref{propExistence_intertwining_paths}. Write $\Gamma=(I_{w_1.\tau},I_{w_2.\tau},\ldots,I_{w_k.\tau})$. Then $A_\Gamma=A_{w_{k-1},w_k,\tau}\circ \ldots \circ A_{w_1,w_2,\tau}$ is an isomorphism and thus for all $i\in \llbracket 1,k-1\rrbracket$, $A_{w_i,w_{i+1},\tau}$ is an isomorphism. Therefore $\Gamma$ is contained in $\tilde{C}(w)$ and thus $I_{w_1.\tau}\in \tilde{C}(w)$, which proves (2).

(3), (4) Let $w,w'\in W^v$. Let us prove that $d(I_{w.\tau},I_{w'.\tau})=d(I_{w'.\tau},I_{w.\tau})$. Maybe considering $\tilde{\tau}=w'.\tau$ and $ww'^{-1}$, we may assume that $w'=1$. Let  $w=s_k\ldots s_1$ be a reduced writing of $w$, with  $s_1, \ldots,s_k\in \SCC$. Then by Proposition~\ref{propExistence_intertwining_paths},  $\Gamma=(I_{\tau},I_{s_1.\tau},\ldots,I_{s_k\ldots s_1.\tau})$ is an intertwining path joining $I_\tau$ to $I_{w.\tau}$. By Proposition~\ref{propExistence_intertwining_paths}, $\Gamma':=(I_{s_k\ldots s_1.\tau},I_{s_{k-1}\ldots s_1.\tau},\ldots,I_{\tau})$ is an intertwining path from $I_{w.\tau}$ to $I_\tau$. As for all $w\in W^v$ and $s\in \SCC$, $\dim \Hom(I_{w.\tau},I_{sw.\tau})=1$, $A_{w,sw,\tau}$ is an isomorphism if and only if $A_{sw,w,\tau}$ is an isomorphism. Therefore $\ell_{\not\simeq}(\Gamma)=\ell_{\not\simeq}(\Gamma')$ and hence $d(I_{\tau},I_{w.\tau})=d(I_{w.\tau},I_\tau)$: $d$ is symmetric.

Let $w,w'\in W^v$ and $\Gamma$ be a path from $I_{w.\tau}$ to $I_{w'.\tau}$. We may assume that $w'=1$. Then using the word property we can transform $\Gamma$ into a path $\Gamma'$ of length $\ell(w)$, by using nil-moves and braid-moves. By Proposition~\ref{propExistence_intertwining_paths}, $\Gamma'$ is then an intertwining path. For each braid-move, $\ell_{\not\simeq}$ remains unchanged (by Lemma~\ref{lemConservation_tau_distance_braid_moves}) and for each nil-move $\ell_{\not\simeq}$ either remain unchanged or decrease by $2$. Thus  $\ell_{\not\simeq}(\Gamma)\geq d(I_\tau,I_{w.\tau})$. 

Let $w,w'\in W^v$. Let $\Gamma$ (resp. $\Gamma'$) be an intertwining path between $I_\tau$ and $I_{w'.\tau}$ (resp. between $I_{w'.\tau}$ and $I_{w.\tau}$). Then the concatenation $\Gamma''$ of $\Gamma$ and $\Gamma'$ is a path between $I_\tau$ and $I_{w'.\tau}$ and thus $d(I_{\tau},I_{w.\tau})\leq \ell_{\not\simeq}(\Gamma'')= d(I_\tau,I_{w'.\tau})+d(I_{w'.\tau},I_{w.\tau})$, which proves that $d$ satisfies the triangle inequality. By (2), for $w,w'\in W^v$, $d(I_{w.\tau},I_{w'.\tau})=0$ if and only if $I_{w.\tau}\simeq I_{w'.\tau}$ which proves that $d$ induces a distance on the set of  connected components of $\tilde{\GC}_\tau$. 
\end{proof}

\subsection{Irreducible representation admitting $\tau$ as a weight}\label{subIrreducible_constituents}

Let $\tau\in T_\FC$ be regular. In this section, we prove the existence of  a unique irreducible representation $M$ of $\AC_\FC$ admitting $\tau$ as a weight. We describe it as  a quotient of $I_\tau$.

\begin{lemma}\label{lemMaximal_submodule_regular_case}(see \cite[Corollary 3.3]{rogawski1985modules})
There exists a unique maximal  submodule $M_\tau^{\max}$ of $I_\tau$.
\end{lemma}

\begin{proof}
Using Proposition~\ref{propIntertwining_nonzero}~(2) we choose a basis $(\xi_w)_{w\in W^v}$ of $I_\tau$ such that $\xi_w\in I_\tau(w.\tau)$ for all $w\in W^v$. Let $\pi^1:I_\tau\rightarrow \FC$ be defined by $\pi^1(\sum_{w\in W^v} a_w \xi_w)=a_1$, for all $(a_w)\in \FC^{(W^v)}$. Let $M$ be a  submodule of $I_\tau$.  Then    $M(\tau)\subset I_\tau(\tau)=\FC \vb_\tau$. Thus $M$ is a proper submodule of $I_\tau$ if and only if $M(\tau)=\{0\}$. Therefore the sum of all the proper submodules of $I_\tau$ is a proper submodule of $I_\tau$, which proves the lemma.
\end{proof}

Let $w\in W^v$ and $\tau'=w.\tau$. Then $M_{\tau'}^\irr:=I_{\tau'}/M_{\tau'}^{\max}$ is an irreducible $\AC_\FC$-module.

We define $\sim$ on $W^v$ by $w\sim w'$ if $I_{w.\tau}\simeq I_{w'.\tau}$ for all $w,w'\in W^v$. This is an equivalence relation. If $w\in W^v$ we denote its class by $[w]$.

\begin{proposition}\label{propWeights_irreducible_components}(see \cite[Proposition 3.5]{rogawski1985modules})

\begin{enumerate}
\item Let $M$ be an irreducible $\AC_\FC$-module such that $M(\tau)\neq \{0\}$. Then $M
\simeq M_\tau^\irr$.

\item\label{itWeights_irreducible_constituents} The set of weights of $M_\tau^\irr$ is $\Wt(M_\tau^\irr)=[1].\tau$ and  $\dim M_\tau^\irr=|[1]|$. In particular, if $w\in W^v$, then $M_\tau^{\irr}$ is isomorphic to $M_{w,\tau}^{\irr}$ if and only if $w\sim 1$.

\end{enumerate}

\end{proposition}

\begin{proof}
Let $x\in M(\tau)\setminus\{0\}$. By Lemma~\ref{lemFrobenius_reciprocity} there exists  $\phi:\in \Hom(I_\tau,M)$ such that $\phi(\vb_\tau)=x$.  By Lemma~\ref{lemMaximal_submodule_regular_case}, $\ker\phi\subset M_\tau^{\mathrm{max}}$ and thus  $\phi$ induces a nonzero map $\overline{\phi}:I_\tau/M_\tau^{\mathrm{max}}=M_\tau^{\mathrm{irr}}\rightarrow M$. As $M_\tau^{\mathrm{irr}}$ and $M$ are irreducible, $\overline{\phi}$ is an isomorphism, which proves (1).

By Lemma~\ref{lemStructure_Weights_module}, $[1].\tau\subset \Wt(M^{\irr}_\tau)$. Let $w\in W^v$ be such that $I_{w.\tau}$ is not isomorphic to $I_\tau$. Let $\phi=A_{w,1,\tau}:I_{w.\tau}\rightarrow I_\tau$. Then by Lemma~\ref{lemComposition_intertwinning_morphisms_nonisomorphism}, $\phi(I_{w.\tau})(\tau)=0$. Therefore, $\phi(I_{w.\tau})\subset M_\tau^{\mathrm{max}}$ and hence $M_\tau^{\mathrm{max}}(w.\tau)\neq 0$. By Proposition~\ref{propIntertwining_nonzero}~(2) we deduce that $M_\tau^{\mathrm{max}}(w.\tau)=I_\tau(w.\tau)$ and hence $w.\tau\notin\Wt(M_\tau^{\mathrm{irr}})$, which proves (2).
  \end{proof}

\begin{remark}
Let $M$ be an irreducible $\AC_\FC$-module such that for some $w\in W^v$, $M\subset I_{w.\tau}$. Then there exists $w'\in W^v$ such that $M$ is isomorphic to $M^{\irr}_{w'.\tau}$. However, there can exist $w\in W^v$ such that $M^{\irr}_{w.\tau}$ is not contained in any $I_{w'.\tau}$. This is the case for example if $M^{\irr}_{w.\tau}$ is finite dimensional, by \cite[Proposition 3.12]{hebert2018principal}.
\end{remark}

\subsection{Strongly indecomposable submodules of $I_\tau$}\label{subIndecomposable_submodules}

\begin{definition}
Let $M$ be a submodule of $\AC_\FC$. One says that $M$ is  \textbf{indecomposable} if for all submodules $M_1,M_2$ of $M$ such that $M_1\oplus M_2=M$, one has $M_1=M$ or $M_2=M$.

We say that $M$ is \textbf{strongly indecomposable} if for every family $(M_j)_{j\in J}$ of submodules of $M$,  \[\sum_{j\in J} M_j=M\implies \exists j\in J|\ M_j=M.\]
\end{definition}

A $\AC_\FC$-module $M$ is strongly indecomposable if and only if  there exists a proper submodule $M_{\mathrm{max}}$ containing every proper submodule of $M$.

Let $\tau\in T_\FC$ be regular. Recall that if $w,w'\in W^v$, $A_{w,w',\tau}$ is an (arbitrary) element of $\Hom (I_{w.\tau},I_{w'.\tau})\setminus\{0\}$. If $w\in W^v$, we set $M_{w,\tau}=A_{w,1,\tau}(I_{w.\tau}) \subset I_\tau$. 

In this subsection, we prove that   the strongly indecomposable submodules of $I_\tau$ are exactly the  $M_{w,\tau}$, for $w\in W^v$ (see Lemma~\ref{lemIndecomposability_Mwtau} and Lemma~\ref{lemDescription_indecomposable_submodules_image_morphism}). We then study how a submodule of $I_\tau$ can be decomposed as a sum of strongly indecomposable submodules (see Theorem~\ref{thmKrull_Schmidt_theorem}).

\subsubsection{Characterization  of the strongly indecomposable submodules of $I_\tau$}

\begin{lemma}\label{lemIndecomposability_Mwtau}
Let $w\in W^v$. Then $M_{w,\tau}=A_{w,1,\tau}(I_{w.\tau})$ is strongly indecomposable.
\end{lemma}

\begin{proof}
Let $(M_j)_{j\in J}$ be a family of submodules of $M_{w,\tau}$ such that $M_{w,\tau}=\sum_{j\in J} M_j$.   For $j\in J$, set $N_j=(A_{w,1,\tau})^{-1}(M_j)$. Let $x\in I_{w.\tau}$ and $y=A_{w,1,\tau}(x)$. Write $y=\sum_{j\in J} y_j$, where $y_j\in M_j$ for all $j\in J$. For  $j\in J$ such that $y_j\neq 0$, choose $x_j\in N_j$ such that $A_{w,1,\tau}(x_j)=y_j$. For   $j\in J$ such that $y_j=0$, set $x_j=0$. Then $x-\sum_{j\in J}x_j\in \mathrm{ker}(A_{w,1,\tau})$. Let $j\in J$. Then $\mathrm{ker}(A_{w,1,\tau})\subset N_j$ and thus $x\in \sum_{j\in J} N_j$. Therefore $\sum_{j\in J} N_j=I_{w.\tau}$. By Lemma~\ref{lemMaximal_submodule_regular_case}, there exists $j\in J$ such that $N_j=I_{w.\tau}$. Then $M_j=M_{w,\tau}$, which proves the lemma.
\end{proof}

\begin{lemma}\label{lemCondition_inclusion_Mwtau}
Let $M\subset I_\tau$ and $w\in W^v$ be such that $w.\tau\in \Wt(M)$. Then $M_{w,\tau}\subset M$.
\end{lemma}

\begin{proof}
By Lemma~\ref{lemFrobenius_reciprocity}, there exists a nonzero intertwiner $f:I_{w.\tau}\rightarrow M$. Then $f\in \Hom(I_{w.\tau},I_\tau)=\FC A_{w,1,\tau}$, which proves the lemma.
\end{proof}

\begin{lemma}\label{lemPreimage_weight_subspaces} Let $w,w'\in W^v$ be such that $w'.\tau\in \Wt(M_{w,\tau})$. Then $M_{w,\tau}(w'.\tau)=A_{w,1,\tau}\big(I_{w.\tau}(w'.\tau)\big)$. In particular, $ A_{w',1,\tau}\in \FC^* A_{w,1,\tau}\circ A_{w',w,\tau}$.
\end{lemma}

\begin{proof}
As $A_{w,1,\tau}$ is a $\AC_\FC$-module morphism, it is an $\FC[Y]$-module morphism and thus $M_{w,\tau}(w'.\tau)\supset A_{w,1,\tau}\big(I_{w.\tau}(w'.\tau)\big)$. Let $y\in M_{w,\tau}(w'.\tau)$ and $x\in (A_{w,1,\tau})^{-1}(\{y\})$.   Using Proposition~\ref{propIntertwining_nonzero}, write $x=\sum_{v\in W^v} x_v$, where for every $v\in W^v$, $x_v\in I_{w.\tau}(v.\tau)$. Then for all $v\in W^v\setminus\{w'\}$, $A_{w,1,\tau}(x_v)\in I_{\tau}(v.\tau)$ and thus $A_{w,1,\tau}(x_v)=0$. Consequently, $y=A_{w,1,\tau}(x)=A_{w,1,\tau}(x_{w'})\in A_{w,1,\tau}\big(I_{w.\tau}(w'.\tau)\big)$ and thus $M_{w,\tau}(w'.\tau)=A_{w,1,\tau}\big(I_{w.\tau}(w'.\tau)\big)$.

 By Proposition~\ref{propIntertwining_nonzero}~(2), $M_{w,\tau}(w'.\tau)=I_\tau(w'.\tau)$. Let $y=A_{w',1,\tau}(\vb_{w'.\tau})$. Then there exists $x\in I_{w.\tau}(w'.\tau)$ such that $y=A_{w,1,\tau}(x)$. Then there exists $\gamma\in \FC^*$ such that  $x=\gamma A_{w',w,\tau}(\vb_{w'.\tau})$. Therefore $y=A_{w,1,\tau}\circ A_{w',w,\tau}(\vb_{w'.\tau)})$. In particular, $A_{w,1,\tau}\circ A_{w',w,\tau}\neq 0$ and thus $ A_{w',1,\tau}\in \FC^* A_{w,1,\tau}\circ A_{w',w,\tau}$.
\end{proof}

\begin{lemma}\label{lemInjectivity_map_modules}
Let $w,w'\in W^v$.  Then $I_{w.\tau}\simeq I_{w'.\tau}$ if and only if $M_{w,\tau}=M_{w',\tau}$. 
\end{lemma}

\begin{proof}
One has $w'.\tau\in \Wt(M_{w,\tau})$ and $w.\tau\in \Wt(M_{w'.\tau})$. Thus by Lemma~\ref{lemPreimage_weight_subspaces} one has: \[\FC^* A_{w',1,\tau}=\FC^* A_{w,1,\tau}\circ A_{w',w,\tau}=\FC^*  A_{w',1,\tau}\circ A_{w,w',\tau}\circ A_{w',w,\tau}.\] By Lemma~\ref{lemComposition_intertwinning_morphisms_nonisomorphism} we deduce that  $A_{w,w',\tau}$ and $A_{w',w,\tau}$ are isomorphisms, which proves the lemma.
\end{proof}


\begin{lemma}\label{lemChasles_distance_weights_Mwtau}
Let $w'\in W^v$ and $w.\tau\in \Wt(M_{w',\tau})$ (i.e $M_{w,\tau}\subset M_{w',\tau}$). Then $d(I_{w.\tau},I_{\tau})=d(I_{w.\tau},I_{w'.\tau})+d(I_{w'.\tau},I_{\tau})$. In particular, $d(I_{w.\tau},I_{\tau})\geq d(I_{w'.\tau},I_\tau)$ and the equality holds if and only if $M_{w,\tau}=M_{w',\tau}$
\end{lemma}

\begin{proof}
By Lemma~\ref{lemPreimage_weight_subspaces}, one has $\FC^* A_{w,1,\tau}=\FC^* A_{w',1,\tau}\circ A_{w,w',\tau}$. Therefore if $\Gamma_1$ is an intertwining path from $I_{w.\tau}$ to $I_{w'.\tau}$ and $\Gamma_2$ is an intertwining path from $I_{w'.\tau}$ to $I_\tau$, the concatenation of $\Gamma_1$ and $\Gamma_2$ is an intertwining path from $I_{w.\tau}$ to $I_\tau$. Thus $d(I_{w.\tau},I_\tau)=d(I_{w.\tau},I_{w'.\tau})+d(I_{w'.\tau},I_\tau)$. 

Thus $d(I_{w.\tau},I_\tau)\geq d(I_{w'.\tau},I_\tau)$ and the equality holds if and only if $d(I_{w.\tau},I_{w'.\tau})=0$ if and only if $M_{w,\tau}=M_{w',\tau}$, by Lemma~\ref{lemInjectivity_map_modules}.
\end{proof}

\begin{lemma}\label{lemDecomposition_weight_subspaces_regular_case}
\begin{enumerate} \item Let $M$ be a submodule of $I_\tau$. Then $M=\bigoplus_{\tau'\in \Wt(M)}  M(\tau')=\bigoplus_{\tau'\in \Wt(M)} I_\tau(\tau')$. 

\item Let $\MC$ be a family of submodules of $I_\tau$. Then $\Wt(\sum_{N\in\mathcal{M}} N)=\bigcup_{N\in \mathcal{M}} \Wt(N)$.

\end{enumerate}
\end{lemma}

\begin{proof}
(1) By \cite[Lemma 3.3 2.]{hebert2018principal}, one has $M=\bigoplus_{\tau'\in \Wt(M)} M(\tau',\text{gen})$. By Proposition~\ref{propIntertwining_nonzero}, for all $\tau'\in \Wt(M)$, $M(\tau',\mathrm{gen})\subset I(\tau',\mathrm{gen})=I(\tau')$ and thus $M(\tau',\mathrm{gen})=M(\tau')$, which proves (i). 

(2) By (1), $\bigoplus_{\tau'\in W^v.\tau}I_\tau(\tau')\supset \sum_{N\in \mathcal{M}} N=\sum_{N\in \mathcal{M}} \sum_{\tau'\in \Wt(N)} I_\tau(\tau')=\sum_{\tau'\in \bigcup_{N\in \mathcal{M}}\Wt(N)} I_\tau(\tau')$, which proves (2).
\end{proof}

\begin{lemma}\label{lemDescription_indecomposable_submodules_image_morphism}
Let  $M\subset I_\tau$ be a strongly indecomposable submodule. Then there exists $w\in W^v$ such that $M=M_{w,\tau}$. More precisely, let $n=\min\{d(I_{v.\tau},I_{\tau}) |v.\tau\in \Wt(M)\}$ and  $w.\tau\in \Wt(M)$ be such that  $d(I_{w.\tau},I_\tau)=n$. Then $M=M_{w,\tau}$.
\end{lemma}

\begin{proof}
By Lemma~\ref{lemCondition_inclusion_Mwtau},  
 $M_{w,\tau}+\sum_{w'.\tau\in \Wt(M)\setminus \Wt(M_{w,\tau})} M_{w',\tau}\subset M$ and thus  by Lemma~\ref{lemDecomposition_weight_subspaces_regular_case} and Proposition~\ref{propIntertwining_nonzero}, \[M_{w,\tau}+\sum_{w'.\tau\in \Wt(M)\setminus \Wt(M_{w,\tau})} M_{w',\tau}=M.\] 
 
Let $w'.\tau\in \Wt(M)$ be such that $\Wt(M_{w',\tau})\ni w.\tau$. Then by  Lemma~\ref{lemCondition_inclusion_Mwtau}, $M_{w',\tau}\supset M_{w,\tau}$. By  Lemma~\ref{lemChasles_distance_weights_Mwtau} and by definition of $w$, 
one has $d(I_{w'.\tau},I_\tau)\leq d(I_{w.\tau},I_\tau)\leq d(I_{w'.\tau},I_\tau)$, thus $M_{w',\tau}=M_{w,\tau}$ and in particular, $w'.\tau\in \Wt(M_{w,\tau})$. 
As \[\Wt(\sum_{w'.\tau\in \Wt(M)\setminus \Wt(M_{w,\tau})} M_{w',\tau})=\bigcup_{w'.\tau\in \Wt(M)\setminus \Wt(M_{w,\tau })}\Wt(M_{w',\tau})\]
 we deduce that $\sum_{w'.\tau\in \Wt(M)\setminus \Wt(M_{w,\tau})} M_{w',\tau}$ does not contain $M_{w,\tau}$. As $M$ is strongly indecomposable we deduce that $M=M_{w,\tau}$, which proves the lemma.
\end{proof}

\subsubsection{Semi-distance on $\GC_\tau$ and ascending chains of strongly indecomposable submodules}

\begin{proposition}
Let $w'\in W^v$ and  $w.\tau\in \Wt(M_{w',\tau})$ (i.e $M_{w,\tau}\subset M_{w',\tau}$). Let  $n=d(I_{w.\tau},I_{w'.\tau})$ and $M_1,\ldots,M_k$ be  a sequence of strongly  indecomposable submodules of $I_\tau$ such that \[M_{w,\tau}=M_1\subsetneq M_2\subsetneq \ldots \subsetneq M_k=M_{w',\tau}.\] 
Then $k\leq n+1$ and there exist strongly indecomposable submodules $M_1',\ldots,M'_n$ and $\sigma:\llbracket 1,k\rrbracket \rightarrow \llbracket 1,n\rrbracket$ strictly increasing such that 
\[M_1'\subsetneq M_2'\subsetneq \ldots \subsetneq M_{n+1}',\ \sigma(1)=1, \sigma(k)=n+1\text{ and }M_i=M'_{\sigma(i)}\text{ for }i\in \llbracket 1,k\rrbracket.\] 
\end{proposition}

\begin{proof}
By Lemma~\ref{lemDescription_indecomposable_submodules_image_morphism} there exist $w_1,\ldots,w_k\in W^v$ such that for all $i\in \llbracket 1,k\rrbracket$, $M_i=M_{w_i,\tau_i}$. Then $w_{k-1}.\tau\in \Wt(M_{k-1})\subset \Wt(M_k)$ and by Lemma~\ref{lemPreimage_weight_subspaces},  $ A_{w_k,1,\tau}\circ A_{w_{k-1},w_k,\tau}\in \FC^* A_{w_{k-1},1,\tau} $. By induction, \begin{equation}\label{eqIntertwiner}
0\neq A_{w_k,1,\tau}\circ A_{w_{k-1},w_k,\tau}\circ\ldots A_{w_1,w_2,\tau}\in \FC^* A_{w_1,1,\tau}.
\end{equation} Set $w_{k+1}=1$. For $i\in \llbracket 1,k \rrbracket$, choose an intertwining path $\Gamma_i$ from $I_{w_i.\tau}$ to $I_{w_{i+1}.\tau}$, whose existence is provided by Proposition~\ref{propExistence_intertwining_paths}. Let $\Gamma$ (resp. $\Gamma'$) be the concatenation of $\Gamma_1$, $\Gamma_2,\ldots, \Gamma_{k-1}$ (resp. $\Gamma_1$, $\Gamma_2,\ldots, \Gamma_{k}$). Then by~(\ref{eqIntertwiner}), $\Gamma'$ and thus $\Gamma$ are intertwining paths from $I_{w_1.\tau}$  to $I_{w_k.\tau}$. Therefore, \begin{equation}\label{eqChasles_distance}
n=d(I_{w_1.\tau},I_{w_k.\tau})=d(I_{w_1.\tau},I_{w_2.\tau})+\ldots+d(I_{w_{k-1}.\tau},I_{w_k.\tau}).
\end{equation} By Lemma~\ref{lemChasles_distance_weights_Mwtau} we deduce that $n\geq k$.

 Write $\Gamma=(I_{v_1.\tau},\ldots,I_{v_m.\tau})$, where $m\in \N$ and $v_1,\ldots,v_m\in W^v$. Let $K=\{i\in \llbracket 1,m-1\rrbracket|\ I_{v_i.\tau}\not \simeq I_{v_{i+1}.\tau}\}$. Then by definition, $|K|=n$.  For $i\in K$, set $\tilde{M_i}=M_{v_i,\tau}$ and set $\tilde{M}_m=M_{v_m,\tau}$. 
 Write $K\cup\{m\}=\{k_1,\ldots,k_{n+1}\}$, $k_1<\ldots <k_{n+1}$ and for $i\in \llbracket 1,n+1\rrbracket$, set $M_i'=\tilde{M}_{k_i}$. 
 As $\Gamma'$ is an intertwining path, one has  $M'_1\subset M'_2\subset \ldots \subset M'_{n+1}$ and by Lemma~\ref{lemChasles_distance_weights_Mwtau}, the inclusions are strict. By definition, $M'_1\simeq M_{v_1,\tau}=M_1$ and $M'_{n+1}\simeq M_{v_m,\tau}=M_k$. Set $\sigma(1)=1$ and for $i\in \llbracket 1,k-1\rrbracket$, $\sigma(i+1)=\sigma(i)+d(I_{w_i.\tau},I_{w_{i+1}.\tau})$. By~(\ref{eqChasles_distance}), $\sigma(k)=n+1$. Let $i\in \llbracket 1,k\rrbracket$ and assume that $M'_{\sigma(i)}=M_i$. As $\ell_{\not\simeq}(\Gamma_i)=d(I_{w_i.\tau},I_{w_{i+1}.\tau})$,  $M'_{\sigma(i+1)}=M_{i+1}$ and the proposition follows.
\end{proof}

\subsubsection{Decomposition of submodules as sums of strongly indecomposable submodules}

\begin{lemma}\label{lemNoetherianity_indecomposable_submodules}
Let $J$ be  a totally ordered set and $(M_j)_{j\in J}$ be an increasing family of strongly indecomposable submodules of $I_\tau$. Then $(M_j)_{j\in J}$ is stationary.
\end{lemma}

\begin{proof}
Let $M=\sum_{j\in J} M_j$. Let $n=\min\{d(I_{v.\tau},I_{\tau})|v.\tau\in \Wt(M)\}$ and $w.\tau\in \Wt(M)$ be  such that $d(I_{w.\tau},I_\tau)=n$. One has $\Wt(M)=\bigcup_{j\in J}\Wt(M_j)$ and thus there exists  $k\in J$ such that $w.\tau\in \Wt(M_k)$. By  Lemma~\ref{lemDescription_indecomposable_submodules_image_morphism}, $M_k=M_{w,\tau}$. Let $k'\in J$ be such that $k'\geq k$ and $w'\in W^v$ be such that $M_{k'}=M_{w',\tau}$, which exists by Lemma~\ref{lemDescription_indecomposable_submodules_image_morphism}. Then $M_{w,\tau}\subset M_{w',\tau}$ and by Lemma~\ref{lemChasles_distance_weights_Mwtau}, $M_{w,\tau}=M_{w',\tau}=M_k=M_{k'}$, which proves the lemma. 
\end{proof}

\begin{theorem}\label{thmKrull_Schmidt_theorem}
Let $\tau\in T_\FC$ be regular. Let $W^v(\tau)=W^v/\sim$ where $w\sim w'$ if and only if $I_{w.\tau}\simeq I_{w'.\tau}$, for $w,w'\in W^v$. Then:\begin{itemize}

\item[(i)] The map from $W^v(\tau)$ to the set of strongly indecomposable submodules of $I_\tau$, which maps each $[w]_\tau\in W^v(\tau)$ to $M_{w,\tau}=A_{w,1,\tau}(I_{w.\tau})$ is well defined and is a bijection.

\end{itemize}

Let  $M$ be a submodule of $I_\tau$ and $\mathrm{SI}(M)$ (resp. $\mathrm{MSI}(M)$) be the set of (resp. maximal) strongly indecomposable submodules of $I_\tau$. Then:
\begin{itemize}

\item[(ii)] One has $M=\sum_{N\in \mathrm{MSI}(M)}N$.

\item[(iii)] Suppose that $M=\sum_{N\in \mathcal{M}}N$, where $\mathcal{M}\subset \mathrm{SI}(M)$. Then $\mathcal{M}\supset \mathrm{MSI}(M)$.
\end{itemize}
\end{theorem}

\begin{proof}
(i) is a consequence of  Lemma~\ref{lemIndecomposability_Mwtau},  Lemma~\ref{lemIndecomposability_Mwtau} and Lemma~\ref{lemInjectivity_map_modules}.

(ii) Let $w.\tau\in \Wt(M)$. Then by definition of $M_{w,\tau}$ and   Lemma~\ref{lemCondition_inclusion_Mwtau}, $M(w.\tau)\subset M_{w,\tau}\subset M$. Thus by Lemma~\ref{lemIndecomposability_Mwtau} and  Lemma~\ref{lemDecomposition_weight_subspaces_regular_case}, $M\supset \sum_{N\in \mathrm{SI}(M)}N\supset \sum_{w.\tau\in \Wt(M)}M(w.\tau)\supset M$.

(iii) Let $N\in \mathrm{MSI}(M)$. By Lemma~\ref{lemDescription_indecomposable_submodules_image_morphism}, there exists $w\in W^v$ such that $N=M_{w,\tau}$. Then $M_{w,\tau}\subset M$, thus by Lemma~\ref{lemDecomposition_weight_subspaces_regular_case} (ii), $w.\tau\in \Wt(M)=\bigcup_{N\in \mathcal{M}}\Wt(N)$. Let $N\in \mathcal{M}$ be such that $w.\tau\in \Wt(N)$. Then $M_{w,\tau}\subset N\subset M$ and thus $N=M_{w,\tau}$. Therefore $M_{w,\tau}\in \mathcal{M}$, which completes the proof of the theorem.
\end{proof}

\subsection{Weights of the submodules of $I_\tau$}\label{subWeights_submodules}

Let $\tau\in T_\FC$ be regular. We proved in  Lemma~\ref{lemDecomposition_weight_subspaces_regular_case} that a submodule of $I_\tau$ is completely determined by its weights. In this subsection, we give a method to determine the weights of the submodules $M_{w,\tau}$, for $w\in W^v$, from the graph  $\tilde{\GC}_\tau$.

\begin{lemma}\label{lemWeights_intersection_modules regular_case}
Let $M$ and $M'$ be submodules of $I_\tau$. Then $\Wt(M\cap M')=\Wt(M)\cap \Wt(M')$. 
\end{lemma}

\begin{proof}
One has $\Wt(M\cap M')\subset \Wt(M)\cap \Wt(M')$. Let $w.\tau\in \Wt(M)\cap \Wt(M')$. Then by Proposition~\ref{propIntertwining_nonzero}, $1\leq \dim M(w.\tau)\leq \dim I_\tau(w.\tau)=1$ and $1\leq \dim M'(w.\tau)\leq \dim I_\tau(w.\tau)=1$. Thus $M'(w.\tau)=M(w.\tau)=I_\tau(w.\tau)\subset M\cap M'$. Hence $w.\tau\in \Wt(M\cap M')$, which proves the lemma. 
\end{proof}

\begin{lemma}\label{lemRank_theorem_weights}
Let $w\in W^v$. Then $\Wt\big(\mathrm{Ker}(A_{w,1,\tau})\big)\sqcup \Wt\big(\mathrm{Im}(A_{w,1,\tau})\big)=W^v.\tau$. 
\end{lemma}

\begin{proof}
Using Proposition~\ref{propIntertwining_nonzero}, we write $I_\tau=\bigoplus_{v\in W^v} \FC \xi_v$ (resp. $I_{w.\tau}=\bigoplus_{v\in W^v} \FC \xi_v'$) where for all $v\in W^v$, $\xi_v\in I_\tau(v.\tau)\setminus\{0\}$ (resp. $\xi'_v\in I_{w.\tau}(v.\tau)\setminus\{0\}$). Let $v.\tau\in W^v.\tau$. Suppose $v.\tau\in \Wt\big(\mathrm{Im}(A_{w,1,\tau})\big)$. Then $\xi_v\in \mathrm{Im}(A_{w,1,\tau})$ and by Lemma~\ref{lemPreimage_weight_subspaces}, there exists $x\in I_{w.\tau}(v.\tau)$ such that $A_{w,1,\tau}(x)=\xi_v$. Then $x\in \FC^* \xi_v'$ and thus $\xi_v'\notin \mathrm{Ker}(A_{w,1,\tau})$. Suppose now $v.\tau\notin \Wt\big(\mathrm{Im}(A_{w,1,\tau})\big)$. As $A_{w,1,\tau}(\xi'_v)\in \FC \xi_v$ we necessarily have $A_{w,1,\tau}(\xi_v')=0$ and thus $v.\tau\in \mathrm{Ker}(A_{w,1,\tau})$, which proves the lemma.
\end{proof}


\begin{proposition}\label{propWeights_intertwiners}
\begin{enumerate}
\item\label{itWeight_simple_noniso} Let $w\in W^v$ and $s\in \SCC$. We assume that $I_{w.\tau}$ is not isomorphic to $I_{sw.\tau}$. Let $f=A_{w,sw.\tau}:I_{w.\tau}\rightarrow I_{sw.\tau}$. Then: \[\Wt \big(\mathrm{Im}(f)\big)=\{uw.\tau|u\in W^v|us>u\}\text{ and }\Wt\big(\mathrm{Ker}(f)\big)=\{uw.\tau|u\in W^v|us<u\}.\]

\item Let $w_1,\ldots,w_{n+1}\in W^v$. For $i\in \llbracket 1,n-1\rrbracket$, set $f_i=A_{w_i,w_{i+1},\tau}$.  Then: \[\Wt\big(\mathrm{Im}(f_n\circ \ldots \circ f_1)\big)=\bigcap_{i=1}^n \Wt\big(\mathrm{Im}(f_i)\big)\text{ and }\Wt\big(\mathrm{Ker}(f_n\circ \ldots \circ f_1)\big)=\bigcup_{i=1}^n \Wt\big(\mathrm{Ker}(f_i)\big).\]
\end{enumerate}

\end{proposition}

\begin{proof}
Maybe considering $\tau'=w.\tau$, we may assume that $w=1$. Let $f'=A_{s,1,\tau}:I_{s.\tau}\rightarrow I_\tau$. Let $u\in W^v$ be such that $us>u$. Let $\tilde{\tau}=u.\tau$. Let $u^{-1}=s_1\ldots s_k$ be a reduced writing of $u^{-1}$, with $s_1,\ldots,s_k\in \SCC$. Then $su^{-1}=ss_1\ldots s_k$ is a reduced writing.
 Let $\Gamma=(I_{\tilde{\tau}},I_{s_k.\tilde{\tau}}, \ldots,I_{s_1\ldots s_k.\tilde{\tau}},I_{ss_1\ldots s_k.\tilde{\tau}})$. Then by Proposition~\ref{propExistence_intertwining_paths}, $\Gamma$ is an intertwining path from $I_{\tilde{\tau}}=I_{u.\tau}$ to $I_{s.\tau}$ and $\Gamma$ contains $I_\tau$. Thus \[0\neq A_\Gamma=A_{1,s,u^{-1}.\tilde{\tau}}\circ A_{(I_{\tilde{\tau}},\ldots,I_{s_1\ldots s_k.\tilde{\tau}})}=f\circ A_{(I_{\tilde{\tau}},\ldots,I_{s_1\ldots s_k.\tilde{\tau}})}\in \FC^* f\circ A_{u,1,\tau} .\]
 
 As $I_{u.\tau}=\AC_\FC.\vb_{u.\tau}$ we deduce that $\{0\}\neq f\big(A_{u,1,\tau}( \vb_{u.\tau})\big)\in I_{s.\tau}(u.\tau)$.

In particular $u.\tau\in \Wt \big(\mathrm{Im}(f)\big)$.  By Lemma~\ref{lemRank_theorem_weights} we deduce that $\Wt\big(\ker(f)\big)\subset \{u.\tau|u\in W^v, us<u\}$.
 
 Let now $u\in W^v$ be such that $us<u$. Let $u'=us$ and $\tau'=s.\tau$. Then by the result we just proved applied to $\tau'$, we have  $u'.\tau\in \mathrm{Im}(f')$. Moreover by Lemma~\ref{lemComposition_intertwinning_morphisms_nonisomorphism}, $f\circ f'=0$. Thus $u'.\tau'=u.\tau\in \Wt\big(\mathrm{Ker}(f)\big)$, which proves the reverse inclusion and proves~(\ref{itWeight_simple_noniso}). 
 
Let $i\in \llbracket 1,n+1\rrbracket$.  Using Proposition~\ref{propIntertwining_nonzero}, we write $I_{w_i.\tau}=\bigoplus_{v\in W^v}\FC \xi_v^i$, where $\xi_v^i\in I_{w_i.\tau}(v.\tau)$, for $v\in W^v$. Let $v.\tau\in \bigcup_{i=1}^n \Wt\big(\mathrm{Ker} (f_i)\big)$ and let $i\in \llbracket 1,n\rrbracket $ be such that $v.\tau\in \Wt\big(\mathrm{Ker}(f_i)\big)$.  Then $f_{i-1}\circ\ldots f_1(\xi_v^1)\in I_{w_i.\tau}(v.\tau)$ and thus $f_i\circ  f_{i-1}\circ\ldots f_1(\xi_v^1)=0$. Hence $f_n\circ\ldots\circ f_1(\xi_v^1)=0$ and thus $v.\tau\in \Wt\big(\mathrm{Ker}(f_n\circ\ldots\circ f_1)\big)$. 

Let $v.\tau\in W^v.\tau \setminus \bigcup_{i=1}^n \Wt\big(\mathrm{Ker} (f_i)\big)$. Let $i\in \llbracket 1,n\rrbracket$, 
$f_i\circ \ldots\circ f_1(\xi_v^1)\in \FC \xi_v^i$. Suppose that
 $f_{i-1}\circ\ldots \circ f_1 (\xi_v^1)\in \FC^* \xi_v^{i}$. Then by assumption, $f_i\circ \ldots \circ f_1(\xi_v^1)\neq 0$ and thus $v.\tau\notin \Wt\big(\mathrm{Ker}(f_n\circ\ldots\circ f_1)\big)$. Consequently $\Wt\big(\mathrm{Ker}(f_n\circ \ldots \circ f_1)\big)=\bigcup_{i=1}^n \Wt\big(\mathrm{Ker}(f_i)\big)$ and we conclude with Lemma~\ref{lemRank_theorem_weights}.
\end{proof}

\begin{remark}\label{rkRight-angled_case}
Suppose that for all $s,t\in \SCC$ such that $s\neq t$, the order of $st$ is infinite (this is the case if and only if for all $s,t\in \SCC$ such that $s\neq t$, the coefficients of the Kac-Moody matrix satisfy $a_{s,t}a_{t,s}\geq 4$,  by \cite[1.3.21 Proposition]{kumar2002kac}). Then for all strongly  indecomposable submodules $M,M'$ of $I_\tau$, one has   $M\cap M'=\{0\}$,   $M\subset M'$ or $M'\subset M$. Therefore the strongly indecomposable submodules of $I_\tau$ are exactly the indecomposable submodules of $I_\tau$ and one can replace the sums by direct sums in Theorem~\ref{thmKrull_Schmidt_theorem}. 

Indeed, let $w\in W^v$. Let $w=s_k\ldots s_1$ be the (unique) reduced writing of $w$, where $s_1,\ldots,s_k\in \SCC$. For $i\in \llbracket 1,k\rrbracket$, set $f_i=A_{s_i\ldots s_1,s_{i-1}\ldots s_1,\tau}$. Then $A_{w,1,\tau}=f_1\circ \ldots\circ f_k$. If for all $i\in \llbracket 1,k\rrbracket$, $f_i$ is an isomorphism, then $M_{w,\tau}=I_\tau$. Otherwise, let $n$ be the maximum of the $i\in \llbracket 1,k\rrbracket$ such that $f_i$ is not an isomorphism. Then $\Wt(M_{w,\tau})$ is the set of  $v.\tau\in W^v.\tau$ such that the reduced writing of $v$ ends up with $s_n\ldots s_1$. Indeed, let $w'=s_ns_{n-1}\ldots s_1$. One has $M_{w,\tau}=f_k\circ\ldots\circ f_1(I_{w.\tau})=f_n\circ\ldots \circ f_1 (I_{w'.\tau})$. By Proposition~\ref{propWeights_intertwiners}, $\Wt(M_{w,\tau})=\bigcap_{i=1}^n \Wt\big(\mathrm{Im}(f_i)\big)$. Moreover by Proposition~\ref{propWeights_intertwiners}, if $i\in \llbracket 1,n\rrbracket$ is such that $f_i$ is not an isomorphism, then $\Wt\big(\mathrm{Im}(f_i)\big)$ is the set of $v.\tau$ such that the reduced writing of $v\in W^v$ ends up with $s_i\ldots s_1$.

 Note that it is not true in general, see Example~\ref{exSL3}.
\end{remark}

\subsection{Examples}\label{subExamples_regular_case}

\subsubsection{The trivial and the Steinberg representations}

Assume for simplicity that $\FC=\C$ and that there exists $\sigma\in \R_{>1}$ such that $\sigma_s=\sigma_s'=\sigma$, for all $s\in \SCC$. Let $\epsilon\in \{-1,1\}$. Let $\tau_\epsilon\in T_\C$ be such that $\tau_\epsilon(\alpha_s^\vee)=\sigma^{2\epsilon}$, for all $s\in \SCC$ (such a $\tau_\epsilon$ exists by \cite[Lemma 6.2]{hebert2018principal}). Then $\tau_\epsilon$ is regular, as proved in the proof of \cite[Lemma A.1]{hebert2018principal}. By \cite[Lemma A.1]{hebert2018principal}, $I_{\tau_\epsilon}$ admits a unique maximal proper submodule $M_\epsilon$. Moreover, $M_\epsilon$ has codimension $1$. Then $I_{\tau_\epsilon}/M_{\epsilon}$ is the \textbf{trivial representation} if $\epsilon=1$ and the \textbf{Steinberg representation} if $\epsilon=-1$.

\begin{example}\label{exSL3}
Suppose that the Kac-Moody matrix $A$ is $\begin{pmatrix}
2 & -1\\ -1 & 2
\end{pmatrix}$ (this is the case for example for $\mathbf{G}=\mathrm{SL}_3$). Write $\SCC=\{s,t\}$. Then $s.\alpha_t^\vee=t.\alpha_s^\vee=\alpha_s^\vee+\alpha_t^\vee$. Thus we have the following graph:
\[\xymatrix{
    I_{\tau_\epsilon}\ \ \ar[r]^{\not\simeq} \ar[rd]^{\not \simeq}   & I_{s.\tau_\epsilon\ \ }\ar[l]\ar[r]^{\simeq} & I_{ts.\tau_\epsilon}\ \ \ar[l]\ar[rd]^{\not\simeq} \\
   \ & I_{t.\tau_\epsilon}\ \  \ar[lu]\ar[r]^{\simeq} & I_{st.\tau_\epsilon}\ \ \ar[l] \ar[r]^{\not\simeq} & \ar[l]\ar[lu]I_{tst.\tau_{\epsilon}}=I_{sts.\tau_\epsilon} &\  
  }\]
  
  By Proposition~\ref{propWeights_intertwiners}, one has $\Wt(M_{s,\tau_\epsilon})=\{s.\tau_\epsilon,ts.\tau_{\epsilon}, sts.\tau_{\epsilon}=tst.\tau_\epsilon\}$, $\Wt(M_{t,\tau_\epsilon})=\{t.\tau_\epsilon,st.\tau_\epsilon,sts.\tau_\epsilon=tst.\tau_\epsilon\}$.
 Let $\Gamma=(I_{sts.\tau_\epsilon},I_{ts.\tau_\epsilon},I_{s.\tau_\epsilon},I_{\tau_\epsilon})$ and $\Gamma'= (I_{tst.\tau_\epsilon},I_{st.\tau_\epsilon},I_{t.\tau_\epsilon},I_{\tau_\epsilon})$. Then $\Gamma$ and $\Gamma'$ are intertwining maps and thus $M_{sts,\tau_\epsilon}$ and thus $M_{sts,\tau_\epsilon}\subset M_{s,\tau_\epsilon}\cap M_{t.\tau_\epsilon}$. By Lemma~\ref{lemWeights_intersection_modules regular_case}, we deduce that $\{0\}\subsetneq \Wt(M_{sts,\tau_\epsilon})\subset \Wt(M_{s,\tau_\epsilon})\cap \Wt(M_{t,\tau_\epsilon})$. Consequently $\Wt(M_{sts,\tau_\epsilon})=\{sts.\tau_\epsilon\}$. The proper submodules of $I_{\tau_\epsilon}$ are $M_{s,\tau_\epsilon}$, $M_{t,\tau_\epsilon}$, $M_{sts.\tau_\epsilon}$ and $M_{s,\tau_\epsilon}+M_{t,\tau_\epsilon}$. Note that $M_{s,\tau_\epsilon}+M_{t,\tau_\epsilon}$ is indecomposable, but not strongly indecomposable.

\end{example}

\begin{example}\label{exSteinberg_representation_right_angled}
We assume that the order of $st$ is infinite for all $s,t\in \SCC$ such that $s\neq t$.  Then every element of $W^v$ admits a unique reduced writing. Let $\epsilon\in \{-1,1\}$.

\begin{enumerate}
\item\label{itProper_indecomposable_steinberg} The proper strongly indecomposable submodules of $I_{\tau_\epsilon}$ are exactly the $M_{s,\tau_{\epsilon}}=A_{s,1,\tau_\epsilon}(I_{s.\tau_\epsilon})$, for $s\in \SCC$. If $s\in \SCC$, then $\Wt(M_{s,\tau_\epsilon})$ is the set of $w.\tau_\epsilon$ such that the reduced writing of $w\in W^v$ ends up with  an $s$.

\item\label{itProper_submodules_steinberg}  The proper submodules of $I_{\tau_\epsilon}$ are exactly the $\bigoplus_{s\in \SCC'}M_{s,\tau_\epsilon}$ such that $\SCC'\subset \SCC$.
\end{enumerate}
\end{example}

\begin{proof}
As $\Phi^\vee\subset \bigoplus_{s\in \SCC} \N \alpha_s^\vee\cup -\bigoplus_{s\in \SCC} \N \alpha_s^\vee$, one has \[\{\alpha^\vee\in \Phi^\vee|\tau_\epsilon(\alpha^\vee)\in \{\sigma^2,\sigma^{-2}\}\}=\{\pm\alpha_s^\vee|s\in \SCC\}.\] Let $w\in W^v$ and $s\in \SCC$ be such that $I_{sw.\tau_\epsilon}$ is not isomorphic to $I_{w.\tau_\epsilon}$.
 Then by \cite[Lemma 5.4]{hebert2018principal}, $\tau_{\epsilon}(w.\alpha_s^\vee)\in \{\sigma^2,\sigma^{-2}\}$. Thus $w.\alpha_s^\vee=\eta \alpha_t^\vee$, where $\eta\in\{-1,1\}$ and $t\in \SCC$. If $\eta=1$, set $w'=w$ and if $\eta=-1$, set $w'=tw$. By \cite[1.3.11 Theorem (b5)]{kumar2002kac}, $w's=tw'$. Suppose $w'\neq 1$. Let $w'=s_1\ldots s_k$ be the reduced writing of $w'$, with $k\geq 1$ and $s_1,\ldots,s_k\in \SCC$. Then $sw'=w't=ss_1\ldots s_k=s_1\ldots s_k t$. If $\ell(w's)=\ell(w')+1$, then these writings are reduced and thus $s=s_1$ and $t=s_k$ by the uniqueness of the writing. This is impossible and thus $\ell(w's)=\ell(w')-1$. But then $s=s_1$ and $t=s_k$ and hence $w's=s_2\ldots s_k=s_1\ldots s_{k-1}$ is a reduced writing. Thus $s_1=s_2$: a contradiction. Therefore $w'=1$ and $s=t$. Thus $w\in \{1,s\}$. Moreover, the graph $\GC_{\tau_\epsilon}$ is a homogeneous tree with valency $|\SCC|+1$.  Therefore the  graph of isomorphisms $\tilde{\GC}_{\tau_\epsilon}$ of $\tau_\epsilon$ has exactly $|\SCC|+1$ connected components: the component containing $I_{\tau_\epsilon}$ and the components containing  $I_{s.\tau_{\epsilon}}$, for $s\in \SCC$. By Lemma~\ref{lemDescription_indecomposable_submodules_image_morphism} we deduce that the proper strongly indecomposable submodules of $I_{\tau_{\epsilon}}$ are exactly the $M_{s,\tau_\epsilon}$, for $s\in  \SCC$. Using Proposition~\ref{propWeights_intertwiners} we deduce (\ref{itProper_indecomposable_steinberg}), which implies (\ref{itProper_submodules_steinberg}), by Remark~\ref{rkRight-angled_case}. 
\end{proof}

\subsubsection{Some representations of $\widehat{\mathrm{SL}}_2$}

Suppose that $\A$ is associated with the affine Kac-Moody matrix $A=\begin{pmatrix}
2 & -2\\ -2 & 2
\end{pmatrix}$. Then $A$ is the affine Kac-Moody matrix associated with the Cartan matrix $(2)$. Let $\mathring{X}=\Z\alpha=\Z$ and $\mathring{Y}=\Z\alpha^\vee=\Z$ for some symbols $\alpha,\alpha^\vee$. Let $X=\mathring{X}\oplus \Z \delta\oplus \Z \delta'$ and $Y=\mathring{Y}\oplus \Z c\oplus \Z d$, where $\delta,\delta',c,d$ are symbols, $\delta(d)=1$, $\delta(c)=\delta(\alpha^\vee)=0$ and $\alpha(c)=\alpha(d)=0$. By \cite[13.1]{kumar2002kac}, we can take $\alpha_0=\delta-\alpha$ and $\alpha_0^\vee=c-\alpha^\vee$ and then $\Phi=\{\pm\alpha+k\delta|k\in\Z\}$, $\Phi^\vee=\{\pm\alpha^\vee+k c|k\in \Z\}$ and $\delta$ is invariant under the action of $W^v$.

Let $a\in C^v_f$ be such that $a-d\in \R\alpha^\vee\oplus \R c$. For $w\in W^v$, $w.a\in w.C^v_f$. Write $w.a=d+x_w\alpha^\vee+y_wc$. Let $w\neq w'\in W^v$. Then $w.a\in w.C^v_f$ and $w'.a\in w'.C^v_f$. Moreover, $w.C^v_f=w.C^v+\R c\neq w'.C^v_f=w'.C^v_f+\R c$ and   thus $(x_w)_{w\in W^v}$ is injective.

 Let $\tau\in T_\C$ be such that $\tau(\alpha^\vee)=\sigma^2$ and $\tau(c)=1$. Then $\big(w^{-1}.\tau(a)\big)_{w\in W^v}$ is injective and thus $\tau$ is regular. Moreover, $\tau(\beta^\vee)=\tau(\alpha^\vee)=\sigma^2$ for all $\beta^\vee\in \Phi^\vee$. Thus by \cite[Lemma 5.4]{hebert2018principal} for all $w\neq w'\in W^v$, $I_{w.\tau}$ and $I_{w'.\tau}$ are not isomorphic. Write $\SCC=\{s,t\}$. Then $W^v$ is the infinite dihedral group. The graph of $I_\tau$ is thus: \[\xymatrix{ \ldots\ar[r]^{\not\simeq} & I_{st.\tau}\ar[l]\ar[r]^{\not\simeq} & I_{t.\tau} \ar[l]\ar[r]^{\not\simeq}& I_{\tau}\ar[l]\ar[r]^{\not\simeq} & I_{s.\tau }\ar[l]\ar[r]^{\not\simeq}& I_{ts.\tau} \ar[l]\ar[r]^{\not\simeq} & \ldots\ar[l]  }.\] Therefore $I_\tau$ admits not irreducible submodule. The family $(M_{(st)^n,\tau})_{n\in \N}$ is a strictly decreasing sequence of submodules and thus $I_\tau$ is not artinian. By Proposition~\ref{propWeights_irreducible_components}, for every $w\in W^v$, $M_{w,\tau}^{\mathrm{irr}}$ is one dimensional and thus $w.\tau$ extends uniquely to a one-dimensional representation of $\AC_\C$.

 \section{Study of $I_\tau$ for $\tau\in \mathcal{U}_\C$}\label{secTau_in_UC}

We now assume that $\FC=\C$ and  that $|\sigma_s|>1, |\sigma_s'|>1$ for all $s\in \SCC$. The ring $\C[Y]$ is a  unique factorization domain. For $\alpha^\vee$, write $\zeta_{\alpha^\vee}=\frac{\zeta_{\alpha^\vee}^{\mathrm{num}}}{\zeta_{\alpha^\vee}^{\mathrm{den}}}$\index{$\zeta_{\alpha^\vee}^{\mathrm{den}}$, $\zeta_{\alpha^\vee}^{\mathrm{num}}$} where $\zeta_{\alpha^\vee}^{\mathrm{num}},\zeta_{\alpha^\vee}^{\mathrm{den}}\in \FC[Y]$ are pairwise coprime. For example if $\alpha^\vee\in \Phi^\vee$ is such that $\sigma_{\alpha^\vee}=\sigma_{\alpha^\vee}'$  we can take $\zeta_{\alpha^\vee}^{\mathrm{den}}=1-Z^{-\alpha^\vee}$ and in any case we will choose $\zeta_{\alpha^\vee}^{\mathrm{den}}$ among $\{1-Z^{-\alpha^\vee},1+Z^{-\alpha^\vee},1-Z^{-2\alpha^\vee}\}$. 

Let $\UC_\C$\index{$\UC_\C$} be the set of $\tau\in T_\C$ such that for all $\alpha^\vee\in \Phi^\vee$,  $\tau(\zeta_{\alpha^\vee}^{\mathrm{num}})\neq 0$. When $\sigma_s=\sigma_s'=\sqrt{q}$ for all $s\in \SCC$, then $\UC_\C=\{\tau\in T_\C|\tau(\alpha^\vee)\neq q,\ \forall \alpha^\vee\in \Phi^\vee\}$.  
  By \cite[Lemma 5.4]{hebert2018principal}, if $\tau\in \UC_\C$, then $I_{w.\tau}\simeq I_\tau$ for all $w\in W^v$. 
 
 \medskip
 
Let $\tau\in \UC_\C$. The aim of this section is to study the submodules of $I_\tau$ (see Theorem~\ref{thmBijection_modules_ideals}) and then to deduce a description of the irreducible representations of $\AC_\C$ admitting $\tau$ as a weight (see Theorem~\ref{thmIrreducible_representations_tau_UC}). 

\medskip

The proof of Theorem~\ref{thmBijection_modules_ideals} is based on the study of the weights of the submodules  of $I_\tau$. Let $M$ be a submodule of $I_\tau$. As $I_\tau\simeq I_{w.\tau}$ for every $w\in W^v$, it suffices to study $M(\tau)$. In order to study it, we first study $M(\tau,\mathrm{gen})$ and $I_\tau(\tau,\mathrm{gen})$. 

To describe $I_\tau(\tau,\mathrm{gen})$, we begin by proving a decomposition $W_\tau=R_\tau\ltimes \Wta$, where $\Wta$ is some reflection subgroup of $W_\tau$ and $R_\tau$ is the generalization of the $R$-group (see Lemma~\ref{lemDecomposition_Wtau_Rgroup}). The group $\Wta$ is a Coxeter group for some set of simple reflections $\SCC_\tau$. We proved in \cite[Lemma 6.21]{hebert2018principal} that if $r=r_{\beta^\vee}\in \SCC_\tau$,  then $K_r:=F_{r_\beta^\vee}-\zeta_{\beta^\vee}$ is  an element of $\AC(T_\C)_\tau$. Using products of $K_r$, for $r\in \SCC_\tau$, we describe the   ``$\Wta$-part'' $I_\tau(\tau,\mathrm{gen},\Wta)$ of $I_\tau(\tau,\mathrm{gen})$. We prove that if $w_R\in R_\tau$, then $F_{w_R}\in \AC(T_\C)_\tau$ (see Lemma~\ref{lemFwr_well_defined_at_tau}), which enables us  to define an element $\psi_{w_R}\in \End(I_\tau)$.  Combining the $\psi_{w_R}$, $w_R\in R_\tau$ and $I_\tau(\tau,\mathrm{gen},\Wta)$, we deduce a description of $I_\tau(\tau,\mathrm{gen})$ (see Proposition~\ref{propDescription_generalized_weight_spaces}).  

In subsection~\ref{subRgroup}  we define, for $\tau\in T_\C$ a group $R_\tau$ such that $W_\tau$ decomposes as $W_\tau=R_\tau \ltimes \Wta$. We then associate to each $w_R\in R_\tau$ an element $\psi_{w_R}\in \End(I_\tau)$.

In subsection~\ref{subI_tau(tau,gen)} we study  $I_\tau(\tau,\mathrm{gen})$, for $\tau\in \UC_\C$. In the case where the Kac-Moody matrix has size $2$, we deduce a description of $I_\tau(\tau)$, using the $\psi_{w_R}$, $w_R\in R_\tau$.  We conjecture that this description remains valid in the general case (see Conjecture~\ref{conjItau}). We then restrict our study to the $\tau\in \UC_\C$ satisfying this conjecture.

In subsection~\ref{subGeneralized_weight_spaces_submodules}, we study the weight spaces and generalized weight spaces of the submodules and quotients of $I_\tau$.

In subsection~\ref{subStudy_End(Itau)}, we study $\End(I_\tau)$ and  describe it as  the group algebra of $R_\tau$ under some additional assumptions (for example when  $\AC_\C$ is associated to a split Kac-Moody group), using the $\psi_{w_R}$, $w_R\in R_\tau$.

In subsection~\ref{subSubmodules_Itau}, we establish a bijection between the right ideals of $\End(I_\tau)$ and the submodules of $I_\tau$.

In subsection~\ref{subIrreducible_representations} we describe the irreducible representations admitting $\tau$ as a weight.

\subsection{The $R$-group}\label{subRgroup}

\subsubsection{Definition of $R_\tau$ and decomposition of $W_\tau$}

In this subsection, we introduce a group generalizing the group called ``the Knapp-Stein $R$-group'' in \cite{keys1982decomposition}.

Recall that $\RCC=\{wsw^{-1}|w\in W^v,s\in \SCC\}$ is the set of reflections of $W^v$. For $\tau\in T_\C$, set $W_\tau=\{w\in W^v|\ w.\tau=\tau\}$\index{$W_\tau$}, $\Phi^\vee_{(\tau)}=\{\alpha^\vee\in \Phi^\vee_+| \zeta_{\alpha^\vee}^{\mathrm{den}}(\tau)=0\}$\index{$\Phi^\vee_{(\tau)}$},  $\RCC_{(\tau)}=\{r=r_{\alpha^\vee}\in \RCC|\alpha^\vee\in \Phi^\vee_{(\tau)}\}$\index{$\RCC_{(\tau)}$} and \[\Wta=\langle \RCC_{(\tau)}\rangle=\langle \{r=r_{\alpha^\vee}\in \RCC| \zeta_{\alpha^\vee}^{\mathrm{den}}(\tau)=0\}\rangle\subset W^v.\]\index{$\Wta$}

 By \cite[Remark 5.1]{hebert2018principal}, $\Wta\subset W_\tau$. When $\alpha_s(Y)=\Z$ for all $s\in \SCC$, then $\Wta=\langle W_\tau\cap \RCC\rangle$.

  By \cite[6.4.1]{hebert2018principal}, $(\Wta,\SCC_\tau)$ is a Coxeter system, where $\SCC_\tau\subset \RCC$ is the set introduced in \cite[Definition 6.11]{hebert2018principal}. We denote by $\ell_\tau$ the corresponding length and by $<_\tau$ the corresponding Bruhat order. By \cite[Lemma 6.12]{hebert2018principal}, for all $w,w'\in W^v$ such that $w\leq_\tau w'$, one has $w\leq w'$.

\begin{definition}
The \textbf{$R$-group of $\tau$} is the subgroup $R_\tau=\{w\in W_\tau|w.\Phi_{(\tau),+}^\vee=\Phi_{(\tau),+}^\vee\}$\index{$R_\tau$} of $W_\tau$. \end{definition}

\begin{lemma}\label{lemConjugation_Decomposition_Wtau} 
Let $w\in W^v$ and $\tau\in T_\C$. Then $w.W_{\tau}.w^{-1}=W_{w.\tau}$, $w.\Phi^\vee_{(\tau)}=\Phi^\vee_{(w.\tau)}$ and $w.W_{(\tau)}w^{-1}=W_{(w.\tau)}$. In particular, $\Wta$ is normal in $W_\tau$ and $W_\tau$ stabilizes $\Phi^\vee_{(\tau)}$.
\end{lemma}

\begin{proof}
The first equality is clear. Let $\beta^\vee\in \Phi^\vee_{(w.\tau)}$. Then \[\zeta^{\mathrm{den}}_{\beta^\vee}(w.\tau)=0=({^{w^{-1}}}\zeta_{\beta^\vee}^{\mathrm{den}})(\tau)=(\zeta_{w^{-1}.\beta^\vee}^{\mathrm{den}})(\tau).\] Thus $w^{-1}.\beta^\vee\in \Phi^\vee_{(\tau)}$ and hence $\Phi^\vee_{(w.\tau)}\subset w.\Phi^\vee_{(\tau)}$. Similarly $\Phi^\vee_{(\tau)}=\Phi^\vee_{(w^{-1}.w.\tau)}\subset w^{-1}.\Phi^\vee_{(w.\tau)}$ and so $\Phi^\vee_{(w.\tau)}= w.\Phi^\vee_{(\tau)}$. We deduce that $\RCC_{(w.\tau)}=w.\RCC_{(\tau)}.w^{-1}$. Consequently $W_{(w.\tau)}=w.\Wta.w^{-1}$.
\end{proof}

\begin{lemma}\label{lemDecomposition_Wtau_Rgroup}(see \cite[I § 3 Theorem 1]{keys1982decomposition})

One has the following decomposition: $W_\tau=R_\tau\ltimes W_{(\tau)}$. 

\end{lemma}

\begin{proof}
Let $w\in \Wta$. Write $w=r_1\ldots r_k$, with $k=\ell_\tau(w)$ and $r_i\in \SCC_\tau$ for all $i\in \llbracket 1,k\rrbracket$. Suppose $k\geq 1$. Let $w'=r_1\ldots r_{k-1}$. Then by \cite[Lemma 6.12]{hebert2018principal}, $w'<w$. One has $\ell(w')=\ell(wr_k)<\ell(w)$ and thus by \cite[1.3.13 Lemma]{kumar2002kac}, $w.\alpha_{r_k}^\vee\in \Phi^\vee_-$. Therefore, $w\notin R_\tau$ and we deduce that $R_\tau\cap \Wta=\{1\}$.

We now prove that $W_\tau=R_\tau.\Wta$. Let $n\in \N$. We assume that $\{w\in W_\tau|\ell(w)\leq n\}\subset R_\tau.\Wta$. Let $w\in W_\tau$ be such that $\ell(w)\leq n+1$.   Let us prove that $w\in R_\tau.\Wta$. If $w\in R_\tau$, there is nothing to prove. Suppose that $w \notin R_\tau$. Then there exists $\alpha^\vee\in \Phi^\vee_{(\tau),+}$ such that $w.\alpha^\vee\in \Phi^\vee_{(\tau),-}$. Let $w=s_{1}\ldots s_{k}$ be a reduced expression of $w$, where $k=\ell(w)$ and $s_{1},\ldots ,s_{k}\in \SCC$. Then by \cite[1.3.14 Lemma]{kumar2002kac}, there exists $j\in  \llbracket 1,k\rrbracket$ such that $\alpha^\vee=s_k\ldots s_{j+1}.\alpha_{s_j^\vee}$. Let $w'=s_1\ldots \hat{s_j}\ldots s_k$. Then $w=w'r_{\alpha^\vee}$. As $\alpha^\vee\in \Phi^\vee_{(\tau)}$, $r_{\alpha^\vee}\in \Wta$. As $w'\in \{w\in W_\tau |\ell(w)\leq n\}$,  $w'\in R_\tau.\Wta$ and thus $w\in R_\tau.\Wta$ , which concludes the proof of the lemma.
\end{proof}

\subsubsection{Bruhat order}

We now study how the Bruhat order behave when we multiply an element of $\Wta$ by an element of $R_\tau$. We will use it to prove that some family of $I_\tau(\tau,\mathrm{gen})$ is free and thus to describe $I_\tau(\tau,\mathrm{gen})$ (see Lemma~\ref{lemPropertiesKF} and Proposition~\ref{propDescription_generalized_weight_spaces}).
 
\begin{lemma}\label{lemBruhat_order_product_reflection_arbitrary_element}
Let $w\in W^v$ and $r\in \RCC$. Then either $wr>w$ or $wr<w$.
\end{lemma}

\begin{proof}
By \cite[1.3.13 Lemma]{kumar2002kac}, one has $wr>w$ if and only if $w.\alpha_r^\vee >0$. Suppose $w.\alpha_r^\vee<0$. Then $wr.\alpha_r^\vee=-w.\alpha_r^\vee>0$ and thus $wr<wr.r=w$, which proves the lemma.
\end{proof}

For $w\in W_\tau$, we set $N_{\Phi^\vee_{(\tau)}}(w)=N_{\Phi^\vee}(w)\cap \Phi^\vee_{(\tau)}$. By Lemma~\ref{lemConjugation_Decomposition_Wtau}, $N_{\Phi^\vee_{(\tau)}}(w)=\{\alpha^\vee\in \Phi^\vee_{(\tau),+}|w.\alpha^\vee\in \Phi^\vee_{(\tau),-}\}$.

\begin{lemma}\label{lemBruhat_order_in_W_tau}

 Let $w_R\in R_\tau$ and $v,w\in \Wta$ be such that $v\leq_\tau w$. Then $v w_R \leq w w_R$ and $w_Rv\leq w_Rw$.

\end{lemma}

\begin{proof}

Let $w'\in \Wta$ and $r\in \RCC_{(\tau)}$ be such that $w'r>_\tau w'$. Then by Lemma~\ref{lemConjugation_Decomposition_Wtau} and by definition of $R_\tau$: \[N_{\Phi^\vee_{(\tau)}}(w_Rw'r)=N_{\Phi^\vee}(w_Rw'r)\cap \Phi^\vee_{(\tau)}=N_{\Phi^\vee_{(\tau)}}(w'r)\supsetneq N_{\Phi^\vee_{(\tau)}} (w_Rw')= N_{\Phi^\vee_{(\tau)}} (w').\] Therefore $w_Rw'r \not < w_Rw'$ and by Lemma~\ref{lemBruhat_order_product_reflection_arbitrary_element},  $w_Rw'r >w_Rw'$.

By definition of the Bruhat order (see \cite[Definition 2.1.1]{bjorner2005combinatorics}), there exist $r_1,\ldots,r_k\in \RCC_{(\tau)}$ such that $v<_\tau v r_1<_\tau vr_1r_2<_\tau \ldots <_\tau v r_1\ldots r_k=w$,  which proves that $w_R w>w_R v$. By applying this result to $w_R^{-1}$, $v^{-1}$ and $w^{-1}$, we deduce that $w_R^{-1} w^{-1} > w_{R}^{-1} v^{-1}$ and thus $(w_R^{-1} w^{-1})^{-1}=ww_R > (w_{R}^{-1} v^{-1})^{-1}=vw_R$, which proves the lemma.
\end{proof}

\subsubsection{Endomorphisms associated to elements of $R_\tau$}

\begin{lemma}\label{lemDefinition_intertwining_operators_associated_Rgroup}
Let $w_R\in R_\tau$. Let  $w_R=s_k\ldots s_1$ be a reduced expression of $w_R$. Let  $j\in \llbracket 1,k\rrbracket$ and  $w_j=s_{j-1}\ldots s_1$. Then $w_j^{-1}.\alpha_{s_j}^\vee\notin \Phi^\vee_{(\tau)}$. 
\end{lemma}

\begin{proof}
One has $\ell (wr_{w_j^{-1}.\alpha_{s_j}^\vee})=k-1<\ell(w)=k$ and thus by \cite[1.3.13 Lemma]{kumar2002kac}, $w_j^{-1}.\alpha_{s_j}^\vee<0$ and by the definition of $R_\tau$, $w_j^{-1}.\alpha_{s_j}^\vee\notin \Phi^\vee_{(\tau)}$. 
\end{proof}

\begin{lemma}\label{lemFwr_well_defined_at_tau}
Let $w_R\in R_\tau$. Then $F_{w_R}\in \AC(T_\C)_\tau$.
 
\end{lemma}

\begin{proof}
Let $w_R=s_k\ldots s_1$ be a reduced writing of $w_R$, where $k=\ell(w_R)$ and $s_1,\ldots,s_k\in \SCC$.   For $j\in \llbracket 1,k\rrbracket$, set $w_j=s_{j-1}\ldots s_1\in W^v$. Then by Lemma~\ref{lemDefinition_intertwining_operators_associated_Rgroup} applied to $w_R^{-1}$,  $w_j^{-1}.\alpha_{s_j}^\vee\notin \Phi^\vee_{(\tau)}$ for all $j\in \llbracket 1,k\rrbracket$. 
Therefore $\tau\big(\zeta^{\mathrm{den}}_{w_j^{-1}.\alpha_{s_{j+1}}^\vee}\big)\neq 0$ 
and hence $\zeta_{w_j^{-1}.\alpha_{s_{j+1}}^\vee}\in \C(Y)_\tau$. Thus by Lemma~\ref{lemReeder 4.3}~(\ref{itDomain_Fw}), $F_{w_R}\in \AC(T_\C)_\tau$. Moreover, by Lemma~\ref{lemReeder 4.3} (\ref{itDomain_Fw}) and (\ref{itLeading_coefficient_Fw}), as $w_R\in W_\tau$, one has $F_{w_R}(\tau)\vb_\tau\in I_\tau(\tau)\cap (I_\tau^{\leq w_R}\setminus I_\tau^{<w_R})$. Using Lemma~\ref{lemFrobenius_reciprocity} we deduce  the lemma.
\end{proof}

For $w_R\in R_\tau$, we set $\psi_{w_R}=\Upsilon_{F_{w_R}(\tau)\vb_\tau}\in \mathrm{End}(I_\tau)$\index{$\psi_{w_R}$} (this is well defined by Lemma~\ref{lemFwr_well_defined_at_tau}). Then  there exists $a\in \C^*$ such that  $\psi_{w_R}(\vb_\tau)-aH_{w_R}\vb_\tau\in I_\tau^{< w_R}:= \bigoplus_{v<w_R} \C H_v \vb_\tau$.

\begin{lemma}\label{lemInvertibility_Psi_wr}
Let $\tau\in \UC_\C$ and $w_R\in R_\tau$. Then $\psi_{w_R}$ is invertible in $\mathrm{End}(I_\tau)$ and its inverse is in $\C^* \psi_{w_R^{-1}}$. 
\end{lemma}

\begin{proof}
By Lemma~\ref{lemFwr_well_defined_at_tau}, there exists $(\theta_v)\in (\C(Y)_\tau)^{(W^v)}$ such that $F_{w_R}=\sum_{v\in W^v} H_v \theta_v$. Then by Lemma~\ref{lemReeder 4.3} (\ref{itCommutation_relation}) \[F_{w_R}*F_{w_R^{-1}}=\sum_{v\in W^v} H_v\theta_v F_{w_R^{-1}}=\sum_{v\in W^v} H_v F_{w_R^{-1}}* {^{w_R^{-1}}\theta_v}.\] Let $v\in W^v$. Then ${^{w_R^{-1}}\theta_v}\in \C(Y)_{w_R.\tau}=\C(Y)_{\tau}$. By Lemma~\ref{lemFwr_well_defined_at_tau}, $F_{w_R^{-1}}\in \AC(T_\C)_\tau$. As $\AC(T_\C)_\tau$ is an $\HC_{W^v,\C}-\C(Y)_\tau$-bimodule we deduce that $F_{w_R}*F_{w_R^{-1}}\in \AC(T_\C)_\tau$. By \cite[Lemma 6.23]{hebert2018principal}, there exists $P\in \C(Y)_\tau$ such that $F_{w_R}*F_{w_R^{-1}}=P$ and $P(\tau)\neq 0$. Thus \[\psi_{w_R^{-1}}\big(\psi_{w_R}(\vb_\tau)\big)=\psi_{w_R^{-1}}\big(F_{w_R}(\tau)\vb_\tau\big)=F_{w_R}(\tau)*F_{w_R^{-1}}(\tau)\vb_\tau=P(\tau)\vb_\tau.\] As $I_\tau=\AC_\C.\vb_\tau$, we deduce that  $\psi_{w_R^{-1}}$ is surjective and  $\psi_{w_R}$ is injective. The lemma follows by symmetry.
\end{proof}

\subsection{Generalized weight spaces of $I_\tau$  for $\tau\in \UC_\C$}\label{subI_tau(tau,gen)}

 Let $\tau\in \UC_\C$. In this subsection, we describe $I_\tau(\tau,\mathrm{gen})$ (see Proposition~\ref{propDescription_generalized_weight_spaces}), using some elements of $\AC(T_\C)_\tau$.  Under some additional assumption, we deduce a description of $I_\tau(\tau)$ in terms of the $F_{w_R}(\tau)\vb_\tau$,  for $w_R\in R_\tau$. We conjecture (see Conjecture~\ref{conjItau}) that our assumption is satisfied for every $\tau\in \UC_\C$. As we shall see (\ref{lemConjItau_size2_case}), it is satisfied when $\AC_\C$ is associated with a size $2$ Kac-Moody matrix. This subsection extends the results of \cite[6.5]{hebert2018principal} (in which the case $\tau\in \UC_\C$ such that $R_\tau=\{1\}$ is treated) and is inspired by \cite{reeder1997nonstandard}. To generalize these results, we use the $\psi_{w_R}$, for $w_R\in R_\tau$.

For $r\in \RCC$, one sets $K_{r}=F_{r}-\zeta_{\alpha_r^\vee}\in \ATC$\index{$K_r$}. By Lemma~\ref{lemReeder 4.3} we have: \begin{equation}\label{eqRelation_commutation_K}
 \theta* K_r=K_r*\theta^r+(\theta^r-\theta)\zeta_r \text{ for all }\theta\in \C(Y).
 \end{equation} 
 
For each $w\in \Wta$ we fix a reduced writing $w=r_1\ldots r_k$, with $k=\ell(w)$ and $r_1,\ldots,r_k\in \SCC_\tau$ and we set $\underline{w}=(r_1,\ldots,r_k)$. Let $K_{\underline{w}}=K_{r_1}\ldots K_{r_k}\in \ATC$\index{$K_{\underline{w}}$}. In \cite[Lemma 6.25]{hebert2018principal}, generalizing results of Reeder (\cite[section 14]{reeder1997nonstandard}, we proved that  $K_{\underline{w}}\in \AC(T_\C)_\tau$, for every $w\in \Wta$.  We set $\KC_\tau(\tau)=\bigoplus_{w\in \Wta} \C K_{\underline{w}}(\tau)\subset \HCW$.\index{$\KC_\tau(\tau)$} 
 
 Recall that if $h=\sum_{v\in W^v} H_v\theta_v\in \AC(T_\C)_\tau$, $\ev_\tau(h)=\sum_{v\in W^v} \tau(\theta_v)H_v\in \HCW$.

\begin{lemma}\label{lemCompatibility_evalutation_weight_vectors}
Let $\theta\in \C(Y)_\tau$, $h\in \AC(T_\C)_\tau$ and $x\in I_\tau(\tau)$. Then $\theta*\ev_\tau(h).x=\ev_\tau(\theta*h).x$.
\end{lemma}

\begin{proof}
Let $w\in W^v$. Then by Lemma~\ref{lemCommutation relation}, one can write $\theta*H_w=\sum_{v\leq w} H_vP_{v,w,\theta}$, for some $P_{v,w,\theta}\in \C[Y]$. Let $Q\in \C(Y)_\tau$. One has $\ev_\tau(\theta*H_w Q)=\sum_{v\leq w} \tau(P_{v,w,\theta} Q)H_v$. One also has $\theta*\ev_\tau(H_w*Q)=\tau(Q)\sum_{v\leq w} H_v P_{v,w,\theta}$. Hence $\theta*\ev_\tau(H_w*Q).x=\sum_{v\leq w}\tau(P_{v,w,\theta}Q)H_v.x=\ev_\tau(\theta*H_w).x$, and the lemma follows by linearity.
\end{proof} 

\begin{notation}\label{notItau(tau,Wta)}
Let $I_\tau(\tau,\mathrm{gen},\Wta)=\KC_\tau(\tau).\vb_\tau=\bigoplus_{w\in \Wta} \C K_{\underline{w}}(\tau)\vb_\tau$ and  $I_\tau(\tau,\Wta)=I_\tau(\tau,\mathrm{gen},\Wta)\cap I_\tau(\tau)$. 
\end{notation}

We set $\mathfrak{m}_\tau=\{\theta\in \C[Y]|\tau(\theta)=0\}$\index{$\mathfrak{m}_\tau$}.

\begin{lemma}\label{lemItau_gen_Wtas_subset_I_tau_gen}
\begin{enumerate}
\item The space $I_\tau(\tau,\mathrm{gen},\Wta)$ is a $\C[Y]$-submodule of $I_\tau(\tau,\mathrm{gen})$.

\item Let $x\in I_\tau(\tau,\mathrm{gen},\Wta)\setminus\{0\}$. Write $x=\sum_{w\in \Wta} a_wK_{\underline{w}}(\tau)\vb_\tau$, where $(a_w)\in \C^{(\Wta)}$. Let $\ell_\tau(x)=\max \{\ell_\tau(w)|w\in \Wta\mathrm{\ and\ }a_w\neq 0\}$. Then for all $\theta_1,\ldots,\theta_{\ell_\tau(x)+1}\in \mathfrak{m}_\tau$, one has $\theta_1\ldots \theta_{\ell_\tau(x)+1}.x=0$.

\end{enumerate}
\end{lemma}

\begin{proof}
Let $\theta\in \C(Y)$ and $w\in \Wta$. Then by \cite[Lemma 6.25]{hebert2018principal}, there exists $k_{w,\theta}\in \bigoplus_{v<_\tau w} K_{\underline{v}}\C(Y)_\tau$ such that 
$\theta*K_{\underline{w}}=K_{\underline{w}}*{^{w^{-1}}\theta}+k_{w,\theta}$. Suppose $\theta\in \C(Y)_\tau$. Then by Lemma~\ref{lemCompatibility_evalutation_weight_vectors},
 \[\theta.(K_{\underline{w}}(\tau).\vb_\tau)=\big(\theta*\ev_\tau(K_{\underline{w}})\big).\vb_\tau=\ev_\tau(K_{\underline{w}}*{^{w^{-1}}\theta}+k_{w,\theta}).\vb_\tau=\big(\tau(\theta)K_{\underline{w}}(\tau)+k_{w,\theta}(\tau)\big).\vb_\tau,\] and $k_{w,\theta}(\theta)\in \bigoplus_{v<_\tau w} \C K_{\underline{v}}(\tau)$. Thus if $\theta\in \mathfrak{m}_\tau$, $\theta.K_{\underline{w}}(\tau)\in \bigoplus_{v<_\tau w} \C K_{\underline{v}}(\tau)$. By induction on $\ell_\tau(w)$, we deduce that $\bigoplus_{v\leq _\tau w} \C K_{\underline{v}}(\tau).\vb_\tau$ is a $\C[Y]$-submodule of $I_\tau$ and that  if $\theta_1,\ldots,\theta_{\ell_\tau(w)+1}\in \mathfrak{m}_\tau$, then $\theta_1\ldots \theta_{\ell_\tau(w)+1}.K_{\underline{w}}(\tau).\vb_\tau=0$. Therefore, $\KC_\tau(\tau).\vb_\tau\subset I_\tau(\tau,\Wta)$. By \cite[Proposition 2.2.9]{bjorner2005combinatorics}, for every finite subset $F$ of $\Wta$, there exists $w\in \Wta$ such that $v\leq_\tau w$, for every $v\in F$. Thus $I_\tau(\tau,\mathrm{gen},\Wta)$ is a $\C[Y]$-submodule of $I_\tau$. 

\end{proof}

\begin{lemma}\label{lemPropertiesKF}
Let $w\in \Wta$ and $w_R\in R_\tau$. Then: \begin{enumerate}
\item $K_{\underline{w}}*F_{w_R}\in \AC(T_\C)_{\tau}$,

\item $K_{\underline{w}}(\tau).\psi_{w_R}(\vb_\tau)= \ev_\tau(K_{\underline{w}}*F_{w_R}).\vb_\tau$,

\item\label{itSupport} $\max\bigg( \supp\big(K_{\underline{w}}.\psi_{w_R}(\vb_\tau)\big)\bigg)=\{ww_R\}$,

\end{enumerate} 
\end{lemma}

\begin{proof}
By \cite[Lemma 6.26]{hebert2018principal} $K_{\underline{w}}\in \AC(T_\C)_\tau$. Thus $K_{\underline{w}}=\sum_{v\in W^v} H_v*P_v$, where $(P_v)\in \big(\C(Y)_\tau\big)^{(W^v)}$. Then by Lemma~\ref{lemReeder 4.3}, $K_{\underline{w}}*F_{w_R}=\sum_{v\in W^v} H_v *F_{w_R}* {^{w_R}}P_v$. Moreover ${^{w_R}}P_v\in \C(Y)_{w_R^{-1}.\tau}=\C(Y)_\tau$, for $v\in W^v$. Thus by Lemma~\ref{lemFwr_well_defined_at_tau} and as $\AC(T_\C)_\tau$ is a right $\C(Y)_\tau$-submodule of $\AC(T_\C)$, we have $K_{\underline{w}}*F_{w_R}\in \AC(T_\C)_{\tau}$. Moreover, \[\ev_\tau(K_{\underline{w}}*F_{w_R}).\vb_\tau=\sum_{v\in W^v} \tau(P_v) H_v*F_{w_R}(\tau).\vb_\tau=K_{\underline{w}}(\tau).\psi_{w_R}(\vb_\tau).\] 

Write $\underline{w}=(r_1,\ldots, r_k)$, with $r_1,\ldots,r_k\in \SCC$. Then by definition of $K_{\underline{w}}$ and by \cite[Lemma 6.22]{hebert2018principal}, there exist $(\theta_v)\in \C(Y)^{[1,w]_\tau}$ such that $K_{\underline{w}}=\sum_{v\leq_\tau w} F_v*\theta_v$. By \cite[Lemma 6.23]{hebert2018principal}, we deduce that there exist $(\tilde{\theta}_v)\in \C(Y)^{[1,w]_\tau}$ such that $K_{\underline{w}}*F_{w_R}=\sum_{v\in [1,w]_\tau} F_{vw_R}\tilde{\theta}_v$.

 By Lemma~\ref{lemBruhat_order_in_W_tau}, for all $w'\in [1,w)_{\tau}$, $w'w_R<ww_R$, thus by Lemma~\ref{lemReeder 4.3}~(\ref{itLeading_coefficient_Fw}), the coordinate in  $H_{w_R}$ of $K_{\underline{w}}*F_{w_R}$ in the basis $(H_v)_{v\in W^v}$ is $\tilde{\theta}_w$   and hence $\tilde{\theta}_w\in \C(Y)_\tau$. By \cite[Lemma 6.26]{hebert2018principal}, $\theta_w(\tau)\neq 0$. Using \cite[Lemma 6.23]{hebert2018principal} we deduce that  $\tilde{\theta}_w(\tau)\neq 0$, which proves (\ref{itSupport}).
\end{proof}

\begin{proposition}\label{propDescription_generalized_weight_spaces}
\begin{enumerate}
\item\label{itGeneralized_weight_space} The family $\big(K_{\underline{w}}(\tau).\psi_{w_R}(\vb_\tau)\big)_{w\in \Wta,w_R\in R_\tau}$ is a $\C$-basis of  $I_\tau(\tau,\mathrm{gen})$ and one has the following decomposition of $\C[Y]$-modules: \[I_\tau(\tau,\mathrm{gen})=\bigoplus_{w_R\in R_\tau} \psi_{w_R}\big(I_\tau(\tau,\mathrm{gen},\Wta)\big)=\bigoplus_{w_R\in R_\tau} \psi_{w_R}(\KC_\tau(\tau).\vb_\tau).\]
\item\label{itWeight_vectors}  One has $I_\tau(\tau)=\bigoplus_{w_R\in R_\tau}\psi_{w_R}\big(I_\tau(\tau,\Wta)\big)$. In particular if $I_\tau(\tau,\Wta)=\C \vb_\tau$, then $I_\tau(\tau)=\bigoplus_{w_R\in R_\tau} \C F_{w_R}(\tau)\vb_\tau$. 

\end{enumerate}
\end{proposition}

\begin{proof}
By Lemma~\ref{lemPropertiesKF} (3), $\big(K_{\underline{w}}(\tau).\psi_{w_R}(\vb_\tau)\big)_{w\in \Wta,w_R\in R_\tau}$ is a free family. Moreover, if $w_R\in R_\tau$, $\psi_{w_R}$ is a $\AC_\C$-module morphism and thus it is a $\C[Y]$-module morphism. Therefore, by Lemma~\ref{lemItau_gen_Wtas_subset_I_tau_gen}, $K_{\underline{w}}(\tau).\psi_{w_R}(\vb_\tau)\in I_\tau(\tau,\mathrm{gen})$ for all $w\in \Wta$ and $w_R\in R_\tau$. Let $x \in I_\tau(\tau,\mathrm{gen})$ and $W_M=\max\big(\supp(x)\big)$. Then by \cite[Lemma 3.3]{hebert2018principal}, $M\subset W_\tau$. Write $W_M=\{w_1,\ldots,w_k\}$, with $k=|W_M|$. For $i\in \llbracket 1,k\rrbracket$, write $w_i=w^iw_R^i$, with $w^i\in \Wta$ and $w_R^i\in R_\tau$, which is possible by Lemma~\ref{lemDecomposition_Wtau_Rgroup}.  Then by Lemma~\ref{lemPropertiesKF} (\ref{itSupport}) there exist $\lambda_1,\ldots,\lambda_k\in \C^*$ such that if $y=x-\sum_{i=1}^k \lambda_i K_{\underline{w^i}}(\tau).\psi_{w_R^i}(\vb_\tau)$, then  for all $v\in \supp(y)$, there exists $w\in W^v$ such that $v<w$. Moreover, $y\in I_\tau(\tau,\mathrm{gen})$. Thus by decreasing induction on $\max \{\ell(w)|w\in W_M\}$ we deduce that $I_\tau(\tau,\mathrm{gen})\subset \bigoplus_{w\in \Wta,w_R\in R_\tau} K_{\underline{w}}(\tau).\psi_{w_R}(\vb_\tau)$. By Lemma~\ref{lemItau_gen_Wtas_subset_I_tau_gen}, $\psi_{w_R}\big(I_\tau(\tau,\mathrm{gen},\Wta)\big)$ is a $\C[Y]$-module for all $w_R\in R_\tau$,  which proves (\ref{itGeneralized_weight_space}).

Let $x\in I_\tau(\tau)\subset I_\tau(\tau,\mathrm{gen})$. 
Write $x=\sum_{w\in \Wta,w_R\in R_\tau} x_{w_R}$ with $x_{w_R}\in \psi_{w_R}\big(I_\tau(\tau,\Wta)\big)$, for $w_R\in R_\tau$. By (1), if $w_R\in R_\tau$, then $x_{w_R}\in I_\tau(\tau)\cap \psi_{w_R}\big(I_\tau(\tau,\mathrm{gen},\Wta)\big)=\psi_{w_R}\big(I_\tau(\tau,\Wta)\big)$, which  concludes the proof of the proposition.
\end{proof}

\begin{corollary}\label{corEnd(Itau)_spanned_invertible}
Suppose that $I_\tau(\tau,\Wta)=\C\vb_\tau$. Then for all $\phi\in \End(I_\tau)$, there exist $k\in \N$ and $\phi_1,\ldots,\phi_k\in \End(I_\tau)^\times$ such that $\phi=\sum_{i=1}^k \phi_i$.
\end{corollary}

\begin{proof}
Let $\phi\in \End(I_\tau)$ and $x=\phi(\vb_\tau)\in I_\tau(\tau)$. By Proposition~\ref{propDescription_generalized_weight_spaces}, one can write $x=\sum_{w_R\in R_\tau} a_{w_R} F_{w_R}(\tau)\vb_\tau$, where $(a_{w_R})\in \C^{(R_\tau)}$. Then $x=\sum_{w_R\in R_\tau} a_{w_R} \psi_{w_R}(\vb_\tau)$ and thus $\phi=\sum_{w_R\in R_\tau} a_{w_R}\psi_{w_R}$. Using Lemma~\ref{lemInvertibility_Psi_wr}, we deduce the corollary.
\end{proof}

\begin{lemma}\label{lemConjItau_size2_case}
Suppose that the Kac-Moody matrix $A$ has size $2$ and that $W^v$ is infinite. Then $I_\tau(\tau,\Wta)=\C \vb_\tau$. 
\end{lemma}

\begin{proof}
If $\Wta=\{1\}$, then $I_\tau(\tau,\Wta)=\C \vb_\tau$. If $(\Wta,\SCC_\tau)$ is isomorphic to the infinite dihedral group, then the proof of \cite[Lemma 6.36]{hebert2018principal} actually proves that
 $I_\tau(\tau,\Wta)=\C \vb_\tau$. By \cite[Lemma 6.37]{hebert2018principal}, as $\Wta$ is generated by reflections, the only remaining case is the case where $\SCC_\tau=\{r\}$ and  $\Wta=\langle r\rangle$, for some reflection $r$. Then $I_\tau(\tau,\mathrm{gen},\Wta)=\C\vb_\tau\oplus K_{r}(\tau)\vb_\tau$. Thus it suffices to prove that $K_{r}(\tau)\vb_\tau\notin I_\tau(\tau)$. Let $\theta\in \C(Y)_\tau$. Then by Lemma~\ref{lemReeder 4.3} (\ref{itCommutation_relation}) we have \[\theta* K_r(\tau)\vb_\tau=K_r(\tau)*\theta^r\vb_\tau+(\theta^r-\theta)\zeta_r\vb_\tau.\] Let $\lambda\in Y$ and suppose that $\theta=Z^\lambda$. Then by \cite[Lemma 6.32]{hebert2018principal}, there exists $a\in \C^*$ such that if $r=r_{\beta^\vee}$, $(\theta^r-\theta)\zeta_r\vb_\tau =a\tau(\lambda)\beta(\lambda)\vb_\tau$. Thus as $\beta(\beta^\vee)=2\neq 0$ we deduce that $\theta*K_r(\tau)\vb_\tau \notin \C K_{r}(\tau)\vb_\tau$, which concludes the proof of the lemma.
\end{proof}

\begin{conjecture}\label{conjItau}
Let $\tau\in \UC_\C$. Then $I_\tau(\tau,\Wta)=\C \vb_\tau$. 
\end{conjecture}

\subsection{Generalized weight spaces of submodules and quotients for $\tau\in \UC_\C$ such that $I_\tau(\tau,\Wta)=\C \vb_\tau$ }\label{subGeneralized_weight_spaces_submodules}

Let $M\subset I_\tau$ be a submodule. We now study $M(\tau,\mathrm{gen})$ and $I_\tau/M(\tau,\mathrm{gen})$  and we deduce results on $I_\tau/M(\tau)$. The main results are Lemma~\ref{lemModules_K_generated_weight_subspaces}, which in particular asserts that $M$ is generated by $M(\tau)$ and Proposition~\ref{propK_description_generalized_weight_spaces_quotient}, which enables to describe $I_\tau/M(\tau)$ as the image of $I_\tau(\tau)$ by the canonical projection.

\subsubsection{Description of $M(\tau,\mathrm{gen})$, for $M\subset I_\tau$}

\begin{lemma}\label{lemProduct_free_families_Itau(tau)}
Suppose that $I_\tau(\tau,\Wta)=\C \vb_\tau$. Let $(x_j)_{j\in J}$ be a free family of $I_\tau(\tau)$. Then $(K_{\underline{w}}(\tau). x_j)_{j\in J,w\in \Wta}$ is a free family of $I_\tau(\tau,\mathrm{gen})$. 
\end{lemma}

\begin{proof}
Let $w\in \Wta$. Using Proposition~\ref{propDescription_generalized_weight_spaces} we define  $\pi_w:I_\tau(\tau,\mathrm{gen})\rightarrow I_\tau(\tau)$ as follows: \[\pi_w\big(\sum_{v\in\Wta, w_R\in R_\tau} a_{v,w_R} K_{\underline{v}}(\tau).\psi_{w_R}(\vb_\tau)\big)=\sum_{w_R\in \Wta} a_{w,w_R}\psi_{w_R}(\vb_\tau),\] for $(a_{w,w_R})\in \C^{(\Wta\times R_\tau)}$. Let $(a_{v,j})\in \C^{(\Wta\times J)}$ be such that $\sum_{(v,j)\in \Wta\times J} a_{v,j} K_{\underline{v}}(\tau).x_j=0$. Let $w\in \Wta$. Then $\pi_w( \sum_{(v,j)\in \Wta\times J} a_{v,j} K_{\underline{w}}(\tau).x_j)=\sum_{j\in J}a_{w,j} x_j=0$ and thus $(a_{w,j})_{w\in \Wta,j\in J} =0$. Therefore $(K_{\underline{w}}(\tau). x_j)$ is free.
\end{proof}

\medskip 

 Recall that $\KC_\tau(\tau)=\bigoplus_{w\in \Wta} K_{\underline{w}}(\tau)\subset \HCW$.

\begin{lemma}\label{lemModules_K_generated_weight_subspaces}
Let $M$ be a submodule of $I_\tau$. Then $M(\tau,\mathrm{gen})=\mathrm{vect}_\C\big(\KC_\tau(\tau).M(\tau)\big)$.
\end{lemma}

\begin{proof}
 For $x\in I(\tau,\mathrm{gen})$, we denote by $x_{w_R}$ the projection of $x$ on $\psi_{w_R}\big(I_\tau(\tau,\mathrm{gen},\Wta)$ with respect to the decomposition of Proposition~\ref{propDescription_generalized_weight_spaces}. Let $x\in  M(\tau,\mathrm{gen})$. Let \[n(x)=|\{w_R\in R_\tau|\ x_{w_R}\neq 0\}|.\] 
We prove the lemma by induction on $n(x)$. Let $m\in \N$ be  such that for all $x\in M(\tau,\mathrm{gen})$ such that $n(x)\leq m$, one has $x\in \mathrm{vect}_\C\big(\KC_\tau(\tau).M(\tau)\big)$.

Let $x\in I_\tau(\tau,\mathrm{gen})$ be such that $n(x)\leq m+1$. Let \[k=\min \{k'\in \Ne|\forall (\theta_1,\ldots,\theta_{k'})\in (\mathfrak{m}_\tau)^{k'}, \theta_1\ldots \theta_{k'}.x=0\}-1,\] which is well defined by Lemma~\ref{lemItau_gen_Wtas_subset_I_tau_gen}~(2) and Proposition~\ref{propDescription_generalized_weight_spaces}.  Let $\theta_1,\ldots,\theta_k\in \mathfrak{m}_\tau$ be such that $y:=\theta_1\ldots\theta_k.x\neq 0$. By definition of $k$, $y\in M(\tau)$. Write $y=\sum_{w_R\in R_\tau} \psi_{w_R}(a_{w_R}\vb_\tau)$, with $(a_{w_R})\in \C^{(R_\tau)}$, which is possible by Proposition~\ref{propDescription_generalized_weight_spaces}. 
Let $w_R\in R_\tau$ be such that $a_{w_R}\neq 0$. By Proposition~\ref{propDescription_generalized_weight_spaces}, there exists $h\in \KC_\tau(\tau)$ such that $h.\psi_{w_R}(\vb_\tau)=x_{w_R}$. Set $\tilde{x}=x-\frac{1}{a_{w_R}}h.y$. Then $\tilde{x}\in M$ and   
\[\tilde{x}=x-\frac{1}{a_{w_R}}\sum_{v_R\in R_\tau}\psi_{v_R}(a_{v_R}h.\vb_\tau).\] Thus $\tilde{x}\in M(\tau,\mathrm{gen})$ (by Proposition~\ref{propDescription_generalized_weight_spaces}) and $\{v_R\in R_\tau|\tilde{x}_{v_R}\neq 0\}\subset \{v_R\in R_\tau|x_{v_R}\neq 0\}\setminus\{w_R\}$. By the induction assumption we deduce that $\tilde{x}\in \KC_\tau(\tau).M(\tau)$.  Therefore $x\in \mathrm{vect}_\C\big(\KC_\tau(\tau).M(\tau)\big)$ and the lemma follows by induction.

\end{proof}

\subsubsection{Description of $I_\tau/M(\tau,\mathrm{gen})$ for $M\subset I_\tau$}

\begin{lemma}\label{lemQuotient_generalized_weight_spaces}
Let $M\subset M'\subset I\tau$ be two $\AC_\C$-modules. Let $\pi_M:I_\tau\twoheadrightarrow I_\tau/M$ be the canonical projection. Then the restriction $g:M'(\tau,\mathrm{gen})\rightarrow M'/M(\tau,\mathrm{gen})$ of $\pi_M$  is well defined and induces an isomorphism of $\C[Y]$-modules \[M'(\tau,\mathrm{gen})/M(\tau,\mathrm{gen})\overset{\sim}{\rightarrow} M'/M(\tau,\mathrm{gen}).\]
\end{lemma}

\begin{proof}
 As $\pi_M$ is a $\AC_\C$-module morphism, it is in particular a $\C[Y]$-module morphism and thus $\pi_M\big(I_\tau(\tau,\mathrm{gen})\big)\subset M'/M(\tau,\mathrm{gen})$, which proves  that $g$ is well defined. Let $\overline{x}\in M'/M(\tau,\mathrm{gen})$. Let $x\in \pi_{M}^{-1}(\overline{x})\cap M'$. 
By \cite[Lemma 3.3]{hebert2018principal}, we can write $x=\sum_{\overline{w}\in W^v/W_\tau} x_{w.\tau}$, where $x_{w.\tau}\in M'(w.\tau,\mathrm{gen})$, for all $\overline{w}\in W^v/W_\tau$. Then for all $\overline{w}\in W^v/W_\tau\setminus \{\overline{1}\}$, $\pi_M(x_{w.\tau})\in M'(w.\tau,\mathrm{gen})$ and thus $\pi_M(x_{w.\tau})=0$. Therefore $\pi_M(x)=\pi_M(x_\tau)$ and thus ${\pi}_M\big(M'(\tau,\mathrm{gen})\big)=M'/M(\tau,\mathrm{gen})$. Moreover, $\ker(g)=M\cap M'(\tau,\mathrm{gen})=M(\tau,\mathrm{gen})$, which proves the lemma.
\end{proof}

\begin{proposition}\label{propK_description_generalized_weight_spaces_quotient}
Let $M$ be a submodule of $I_\tau$. Let $H$ be a complement of $M(\tau)$ in $I_\tau(\tau)$. \begin{enumerate}
\item Then we have the following decomposition of $\C[Y]$-submodules: \[I_\tau(\tau,\mathrm{gen})=\mathrm{vect}_\C\big(\KC_\tau(\tau).M(\tau)\big)\oplus \mathrm{vect}_\C\big(\KC_\tau(\tau). H\big).\]

\item Let $\pi_M:I_\tau\twoheadrightarrow I_\tau/M$ be the canonical projection. Then the restriction $f:\vect_\C\big(\KC_\tau(\tau).H\big)\rightarrow I_\tau/M(\tau,\mathrm{gen})$ of $\pi_M$ is well defined and is an isomorphism of $\C[Y]$-modules. 

\item One has $I_\tau/M(\tau)=\pi_M(H)=\pi_M\big(I_\tau(\tau)\big)$, $\dim H=\dim I_\tau/M(\tau)$  and $\dim I_\tau/M(\tau,\mathrm{gen})=|\Wta|\dim I_\tau/M (\tau)$.
\end{enumerate}
\end{proposition}

\begin{proof}
(1) By Lemma~\ref{lemModules_K_generated_weight_subspaces}, Lemma~\ref{lemProduct_free_families_Itau(tau)} and Proposition~\ref{propDescription_generalized_weight_spaces}, \[\begin{aligned} \mathrm{vect}_\C \big( \KC_\tau(\tau).H\big)+\mathrm{vect}_\C\big(\KC_\tau(\tau).M(\tau)\big)= &  \mathrm{vect}_\C \big( \KC_\tau(\tau).H\big)\oplus \mathrm{vect}_\C\big(\KC_\tau(\tau).M(\tau)\big)
\\ 
=& \mathrm{vect}_\C\bigg(\KC_\tau(\tau).\big(H\oplus M(\tau)\big)\bigg)\\
 = &\mathrm{vect}_\C\big(\KC_\tau(\tau).I_\tau( \tau)\big)=I_\tau(\tau,\mathrm{gen}).\end{aligned}.\]

Therefore  $\mathrm{vect}_\C\big(\KC_\tau(\tau).H\big)$ is a complement of $M(\tau,\mathrm{gen})=\mathrm{vect}_\C\big(\KC_\tau(\tau).M(\tau)\big)$. For $w\in \Wta$, set $\KC_\tau^{< w}=\bigoplus_{v\in [1,w)_\tau} K_{\underline{v}}\C(Y)_\tau$ and $\KC_\tau^{<w}= \bigoplus_{v\in [1,w)_\tau} \C K_{\underline{v}}(\tau)$. Let $\theta\in \C(Y)_\tau$ and $w\in W^v$. Then by \cite[Lemma 6.27]{hebert2018principal}, there exists $k_{w,\theta}\in \KC_\tau^{<w}$ such that $\theta*K_{\underline{w}}=K_{\underline{w}}*{^{w^{-1}}\theta}+k_{w,\theta}$. By Lemma~\ref{lemCompatibility_evalutation_weight_vectors}, we deduce that if $x\in I_\tau(\tau)$, then  $\theta*K_{\underline{w}}(\tau).x=\ev_\tau(\theta*K_{\underline{w}}).x=\tau(\theta)K_{\underline{w}}(\tau).x+\ev_\tau(k_{w,\theta}).x$. By induction on $\ell(w)$, we deduce that if $x\in I_\tau(\tau)$, then $\KC_\tau(\tau).x$ is a $\C[Y]$-submodule of $I_\tau$, which proves (1).

(2) is a consequence of (1), of Lemma~\ref{lemModules_K_generated_weight_subspaces} and  of Lemma~\ref{lemQuotient_generalized_weight_spaces} applied with $M'=I_\tau$. 

(3) By (2), $f^{-1}\big(I_\tau/M(\tau)\big)=\vect_\C(\KC_\tau(\tau).H)$  Let $x\in I_\tau(\tau)\cap  \vect_\C(\KC_\tau(\tau).H)$. Write $x=h+m$, with $h\in H$ and $m\in M(\tau)$. Then $x-h\in \vect_\C(\KC_\tau(\tau).H)\cap \KC_\tau(\tau).M(\tau)=0$ (by (1)) and thus $x=h\in H$. Therefore $f(H)=\pi_M(H)=\pi_M\big(I_\tau(\tau)\big)=I_\tau/M(\tau)$. Therefore, $\dim H=\dim I_\tau/M(\tau)$. By (2) and by Lemma~\ref{lemProduct_free_families_Itau(tau)}, $\dim I_\tau/M(\tau,\mathrm{gen})=\dim \vect_\C(\KC_\tau(\tau).H)=|\Wta|.|H|$, which concludes the proof of the lemma.
\end{proof}

\begin{corollary}\label{corProjection_endomorphisms}
Let $M$ be a submodule of $I_\tau$ and $\pi_M:I_\tau\twoheadrightarrow I_\tau/M$ be the canonical projection. Then the map $\End(I_\tau)\twoheadrightarrow \Hom(I_\tau,I_\tau/M)$ defined by $\phi\mapsto \pi_M\circ \phi$ is surjective.
\end{corollary}

\begin{proof}
Let $\overline{\phi}\in \Hom(I_\tau,I_\tau/M)$. Let $\overline{x}=\overline{\phi}(\vb_\tau)\in I_\tau/M(\tau)$. Using Proposition~\ref{propK_description_generalized_weight_spaces_quotient}~(3), we choose $x\in I_\tau(\tau)$ such that $\pi_M(x)=\overline{x}$. Let $\phi\in \End(I_\tau)$ be such that $\phi(\vb_\tau)=x$, whose existence if provided by Lemma~\ref{lemFrobenius_reciprocity}. Then $\pi_M\circ\phi(\vb_\tau)=\overline{\phi}(\vb_\tau)$ and as $I_\tau=\AC_\C.\vb_\tau$, we deduce that $\overline{\phi}=\pi_M\circ\phi$.
\end{proof}

\subsubsection{Invariance of the dimensions of the weight spaces under the action of $W^v$}
\begin{lemma}\label{lemInvariance_dimension_weights_Weyl_group}
Let $M$ be a $\AC_\C$-module and $w\in W^v$. Then $\dim M(\tau)=\dim M (w.\tau)$.
\end{lemma}

\begin{proof}
As $\tau\in \UC_\C$, there exists an isomorphism $\phi:I_{w.\tau}\rightarrow I_{\tau}$. Let $g:\Hom (I_\tau,M)\rightarrow  \Hom(I_{w.\tau},M)$ be defined by $g(f)=f\circ \phi$ for $f\in \Hom (I_\tau,M)$. Then $g$ is a vector space isomorphism. By Lemma~\ref{lemFrobenius_reciprocity} we deduce that   $\dim \Hom (I_\tau,M) =\dim M(\tau)=\dim M (w.\tau)=\dim \Hom(I_{w.\tau},M)$.
\end{proof}

\begin{lemma}\label{lemInvariance_dimension_generalized_weight_spaces}
Let $M$ be a submodule of $I_\tau$. Then for all $w\in W^v$, $\dim M(\tau,\mathrm{gen})=\dim M(w.\tau,\mathrm{gen})$ and $\dim I_\tau/ M(\tau,\mathrm{gen})=\dim I_\tau/M(w.\tau,\mathrm{gen})$.
\end{lemma}

\begin{proof}
As $\tau\in \UC_\C$, there exists an isomorphism $\phi:I_\tau\rightarrow I_{w.\tau}$. Let $M'=\phi(M)$. Then by Lemma~\ref{lemModules_K_generated_weight_subspaces}, \[\dim M'(w.\tau,\mathrm{gen})=\dim M(w.\tau,\mathrm{gen})=|W_{(w.\tau)}|.\dim M'(w.\tau)=|W_{(w.\tau)}|\dim M(w.\tau).\] By Lemma~\ref{lemConjugation_Decomposition_Wtau}, $|W_{(w.\tau)}|=|W_{(\tau)}|$ and by Lemma~\ref{lemInvariance_dimension_weights_Weyl_group}, $\dim M(\tau)=\dim M(\tau)$, which proves that $\dim M(w.\tau,\mathrm{gen})=|W_{(\tau)}|\dim M(w.\tau)=\dim M(\tau,\mathrm{gen})$.

The map $\phi$ induces an isomorphism $\overline{\phi} :I_\tau/M\overset{\sim}{\rightarrow} I_{w.\tau}/M'$. By Proposition~\ref{propK_description_generalized_weight_spaces_quotient}, Lemma~\ref{lemInvariance_dimension_weights_Weyl_group} and Lemma~\ref{lemConjugation_Decomposition_Wtau}, \[\dim I_{w.\tau}/M' (w.\tau,\mathrm{gen})= |W_{(w.\tau)}|\dim I_{w.\tau}/M' (w.\tau)=|\Wta| \dim I_\tau/M(\tau)=\dim I_{\tau}/M(\tau,\mathrm{gen}).\]
\end{proof}

\begin{lemma}\label{lemCharacterization_equality_generalized_weight_spaces}
Let $M\subset M'\subset I_\tau$ be $\AC_\C$-modules. Then $M=M'$ if and only if $M(\tau,\mathrm{gen})=M'(\tau,\mathrm{gen})$ if and only if $M(\tau)=M'(\tau)$.
\end{lemma}

\begin{proof}
It is clear that $M=M'\implies M(\tau,\mathrm{gen})=M'(\tau,\mathrm{gen})$ and that $M(\tau,\mathrm{gen})=M'(\tau,\mathrm{gen})\implies M(\tau)=M'(\tau)$. By Lemma~\ref{lemModules_K_generated_weight_subspaces} , $M(\tau)=M'(\tau)$ if and only if $M(\tau,\mathrm{gen})=M'(\tau,\mathrm{gen})$. Suppose that $M(\tau,\mathrm{gen})=M'(\tau,\mathrm{gen})$.
Let $M'\supset M$ be a submodule of $I_\tau$ such that  $M'(\tau,\mathrm{gen})=M(\tau,\mathrm{gen})$. Then by Lemma~\ref{lemQuotient_generalized_weight_spaces},  $M'/M(\tau,\mathrm{gen})=\{0\}$ and in particular, $M'/M(\tau)=\{0\}$. Using  Lemma~\ref{lemInvariance_dimension_weights_Weyl_group} we deduce that $M'/M(w.\tau)=\{0\}$, for every $w\in W^v$. Therefore $\Wt(M'/M)\cap W^v.\tau=\emptyset$. Moreover by \cite[Lemma 3.3]{hebert2018principal}, $M'=\bigoplus_{\tau'\in \Wt(M')} M'(\tau',\mathrm{gen})$ and  $\Wt(M')\subset W^v.\tau$. Therefore, $M'/M=\bigoplus_{\tau'\in \Wt(M')} M'/M(\tau',\mathrm{gen})=\bigoplus_{\tau'\in W^v.\tau} M'/M(\tau',\mathrm{gen})=\{0\}$.  Thus $M'=M$. 
\end{proof}

\subsection{Study of $\End(I_\tau)$ for $\tau\in \UC_\C$ such that $I_\tau(\tau,\Wta)=\C\vb_\tau$}\label{subStudy_End(Itau)}
In this subsection, we study  the algebra $\End(I_\tau)$. We prove that  $\End(I_\tau)$ is isomorphic to $\C[R_\tau]$ when $\AC_\C$ is associated to a split Kac-Moody group or when the order of $st$ is infinite for every $s,t\in \SCC$ such that $s\neq t$, for $\tau\in \UC_\C$ satisfying Conjecture~\ref{conjItau} (see Proposition~\ref{propDescription_endomorphism_algebra}). Our proof relies on the fact that one has $\psi_{v_R}\circ \psi_{w_R}\in \C\psi_{v_Rw_R}$ for all $v_R,w_R\in R_\tau$. We normalize the $\psi_{w_R}$ suitably to obtain the desired isomorphism.

We then give criteria for an element of $\End(I_\tau)$ to be surjective or injective (see Lemma~\ref{lemCharacterization_injectivity_UC} and Lemma~\ref{lemCharacterization_surjectivity_UC}.

\subsubsection{Description of $\End(I_\tau)$ in the split case or in the right-angled case}

 For $s\in \SCC$, we set $F_s'=F_s*\frac{1}{\zeta_s}\in \AC(T_\C)$.

\begin{lemma}\label{lemFs'2}
Let $s\in \SCC$. Then $(F_s')^2=1$.
\end{lemma} 

\begin{proof}
By Lemma~\ref{lemReeder 4.3} (\ref{itCommutation_relation}) and \cite[Lemma 5.3]{hebert2018principal}, one has: \[(F_s')^2=F_s*\frac{1}{\zeta}*F_s*\frac{1}{\zeta_s}=F_s^2*\frac{1}{\zeta_s*{^s}\zeta_s}=1.\] 
\end{proof}

For $G$ a group,  $a,b\in G$ and $m\in \N$, we denote by  $\Pi(a,b,m)$ the product $abab\ldots $ having $m$ factors.

\begin{lemma}\label{lemF'w0}
We assume that there exists $\sigma\in \C^*$ such that  $\sigma_s=\sigma_s'=\sigma$ for all $s\in \SCC$. Let $s_1,s_2\in \SCC$. We assume that the order $m(s_1,s_2)$ of $s_1s_2$ is finite. Then: \[\Pi\big(F_{s_1}',F_{s_2}',m(s_1,s_2)\big)=\Pi\big(F_{s_2}',F_{s_1}',m(s_1,s_2)\big).\]
\end{lemma}

\begin{proof}
By assumption on the $\sigma_s,\sigma_s'$, one has $\zeta_s=\frac{1-\sigma^2 Z^{-\alpha_s^\vee}}{1-Z^{-\alpha_s^\vee}}$ for $s\in \SCC$. Let $m=m(s_1,s_2)$ By Lemma~\ref{lemReeder 4.3} (\ref{itCommutation_relation}), \[ \Pi\big(F_{s_1}',F_{s_2}',m\big) = \Pi\big(F_{s_1},F_{s_2},m\big)*\prod_{\alpha^\vee\in N_{\Phi^\vee}\big(\prod(s_1,s_2,m)\big)}\frac{1-\sigma^2 Z^{-\alpha^\vee}}{1-Z^{-\alpha^\vee}}.\]

By Lemma~\ref{lemReeder 4.3} (\ref{itWell_definedness_Fw}), $\Pi\big(F_{s_1},F_{s_2},m\big)= \Pi\big(F_{s_2},F_{s_1},m\big)$. Moreover, $\prod(s_1,s_2,m)=\prod(s_2,s_1,m)$, which proves the lemma. 

\end{proof}

\begin{proposition}\label{propDescription_endomorphism_algebra}
Let $\tau\in  \UC_\C$. We make the following assumptions:\begin{enumerate}
\item  the order of $st$ is infinite for every $s,t\in \SCC$ such that $s\neq t$ or  there exists $\sigma\in \C^*$ such that $\sigma_s=\sigma_s'=\sigma$ for all $s\in \SCC$,

\item $I_\tau(\tau,\Wta)=\C\vb_\tau$.  
\end{enumerate} Then $\mathrm{End}(I_\tau)$ is isomorphic to $\C[R_\tau]$.
\end{proposition}

\begin{proof}
By \cite[1.1]{bjorner2005combinatorics} and Lemmas~\ref{lemFs'2} and \ref{lemF'w0}, there exists a unique morphism $F':W^v\rightarrow \AC(T_\C) $ such that $F'(s)=F'_s$ for all $s\in \SCC$. We denote $F'_w$ instead of $F'(w)$, for $w\in W^v$. Let $w_R\in R_\tau$. Let $w_R=s_1\ldots s_k$ be a reduced expression of $w_R$. Then by Lemma~\ref{lemReeder 4.3} (\ref{itCommutation_relation}), there exist $w_1,\ldots,w_k\in W^v$ such that \[F'_{w_R}=F_{s_1}*\ldots*F_{s_k}*\prod_{i=1}^k {^{w_i}\zeta_{\alpha_{s_i}^\vee}^{\mathrm{den}}}\prod_{i=1}^k {\frac{1}{^{w_i}\zeta_{\alpha_{s_i}^\vee}^{\mathrm{num}}}}=F_{w_R}\prod_{i=1}^k {^{w_i}\zeta_{\alpha_{s_i}^\vee}^{\mathrm{den}}}*\prod_{i=1}^k {\frac{1}{^{w_i}\zeta_{\alpha_{s_i}^\vee}^{\mathrm{num}}}}.\] Then $\prod_{i=1}^k {^{w_i}\zeta_{\alpha_{s_i}^\vee}^{\mathrm{den}}}\in \C[Y]\subset \C(Y)_\tau$ and by definition of $\UC_\C$, $\prod_{i=1}^k {\frac{1}{^{w_i}\zeta_{\alpha_{s_i}^\vee}^{\mathrm{num}}}}\in \C(Y)_\tau$. By Lemma~\ref{lemFwr_well_defined_at_tau} we deduce that $F'_{w_R}\in \AC(T_\C)_\tau$.

Let $v_R,w_R\in R_\tau$. Write $F'_{v_R}=\sum_{u\in W^v} H_u*\theta_{u,v_R}$ and $F'_{w_R}=\sum_{u\in W^v} H_u*\theta_{u,w_R}$, where $(\theta_{u,v_R}),(\theta_{u,w_R})\in (\C(Y)_\tau)^{(W^v)}$. By Lemma~\ref{lemReeder 4.3} (\ref{itCommutation_relation}), as $F_{w_R}'\in F_{W_R}\C(Y)$, one has \[F'_{v_R}*F'_{w_R}=\sum_{u,\in W^v} H_u\theta_{u,v_R}F_{w_R}'=\sum_{u\in W^v} H_u F_{w_R}'*({^{w_R}\theta_{u,v_R}})=\sum_{u,u'\in W^v} H_u*H_{u'}*\theta_{u,w_R}*{^{w_R}\theta_{u,v_R}}.\] Thus \[(F'_{v_R}*F'_{w_R})(\tau)=\sum_{u,u'\in W^v}\tau(\theta_{u',w_R})\tau({^{w_R}\theta_{u,v_R}}) H_{u}*H_{u'}=\sum_{u,u'\in W^v}\tau(\theta_{u',w_R})\tau(\theta_{u,v_R}) H_{u}*H_{u'}.\] Therefore $(F'_{v_R}*F'_{w_R})(\tau)=F_{v_R}'(\tau)*F_{w_R}'(\tau)\in \HC_{W^v,\C}$.

Write $\C[R_\tau]=\bigoplus_{w_R\in R_\tau}\C e^{w_R}$, where the $e^{w_R}$ are symbols such that $e^{v_R}e^{w_R}=e^{v_Rw_R}$ for all $v_R,w_R\in R_\tau$. For $w_R\in R_\tau$, set $\psi'_{w_R}=\Upsilon_{F'_{w_R}(\tau)\vb_\tau}\in \End(I_\tau)$, where $\Upsilon$ is defined in Lemma~\ref{lemFrobenius_reciprocity}.  Let $f:\C[R_\tau]\rightarrow \End(I_\tau)$ be the linear map such that $f(e^{w_R})=\psi'_{w_R^{-1}}$, for $w_R\in R_\tau$. Let $v_R,w_R\in R_\tau$. Then \[f(e^{v_R})\circ f(e^{w_R})(\vb_\tau)=\psi'_{v_R^{-1}}\big(\psi'_{w_R^{-1}}(\vb_\tau)\big)=F'_{w_R^{-1}}(\tau)\psi'_{v_R^{-1}}(\vb_\tau)=F'_{w_R^{-1}}(\tau)*F'_{v_R^{-1}}(\tau)\vb_\tau,\] 
  thus $f(e^{v_R})\circ f(e^{w_R})(\vb_\tau)=f(e^{v_Rw_R})(\vb_\tau)$, which proves that $f(e^{v_R})\circ f(e^{w_R})=f(e^{v_Rw_R})$. Therefore $f$ is an algebra morphism. By Proposition~\ref{propDescription_generalized_weight_spaces},  the map $\C[R_\tau]\rightarrow I_\tau(\tau)$ sending each $x\in \C[R_\tau]$ to $f(x)(\vb_\tau)$ is a bijection and by Lemma~\ref{lemFrobenius_reciprocity} we deduce that $f$ is bijective.
 \end{proof}

In \cite[Section 6]{keys1987indistinguishability}, Keys gives  an example where $\mathrm{End}(I_\tau)\not \simeq \C[R_\tau]$.

\subsubsection{Study of injectivity and surjectivity}

\begin{lemma}\label{lemCharacterization_injectivity_UC}
Let $\tau\in \UC_\C$ and let $f\in \End(I_\tau)$. Then $f$ is injective if and only if for every $g\in \End(I_\tau)$, $f\circ g\neq 0$.
\end{lemma}

\begin{proof}
Suppose that $f$ is not injective. Let $M=\ker(f)\subset I_\tau$. By \cite[Lemma 3.3]{hebert2018principal}, there exists $\tau'\in W^v.\tau\cap \Wt(M)$.  As $\tau\in \UC_\C$, $I_{\tau'}\simeq I_\tau$ and thus by  Lemma~\ref{lemStructure_Weights_module}, $\tau\in \Wt(M)$. Let $x\in M(\tau)\setminus\{0\}$. By Lemma~\ref{lemFrobenius_reciprocity}, there exists $g\in \End(I_\tau)$ such that $g(\vb_\tau)=x$. Then $f\circ g(\vb_\tau)=0$. As $I_\tau=\AC_\C.\vb_\tau$, we deduce that $f\circ g=0$. 
\end{proof}

\begin{remark}
As we shall see in~\ref{subsubWtau=Z}, there can exist $f\in \End(I_\tau)$ injective such that for all $g\in \End(I_\tau)$, $g\circ f\neq \Id$. 
\end{remark}

\begin{lemma}\label{lemCharacterization_surjectivity_UC}
Let $\tau\in \UC_\C$ be such that $I_\tau(\tau,\Wta)=\C \vb_\tau$. Let $f\in \End(I_\tau)$. Then $f$ is surjective if and only if there exists $g\in \End(I_\tau)$ such that $f\circ g= \Id$. In particular if $\End(I_\tau)$ is commutative, then $f$ is surjective if and only if $f$ is invertible. 
\end{lemma}

\begin{proof}
Suppose that $f$ is surjective. Let $M=\ker(f)$. Then $f$ induces an isomorphism $\overline{f}:I_\tau/M\overset{\simeq}{\rightarrow} I_\tau$. By Corollary~\ref{corProjection_endomorphisms}, we can write $\overline{f}^{-1}=\pi_M\circ \phi$, where $\phi\in \End(I_\tau)$. Then $f\circ\phi=\overline{f}\circ \pi_M\circ\phi=\Id$. 
\end{proof}

\subsection{Submodules of $I_\tau$ when $I_\tau(\tau,\Wta)=\C \vb_\tau$}\label{subSubmodules_Itau}

In this subsection, we describe the submodules of $I_\tau$ by using right ideal of $\End(I_\tau)$ (see Theorem~\ref{thmBijection_modules_ideals}).

A \textbf{right ideal $J$ of }$\mathrm{End}(I_\tau)$ (resp. left ideal) is a vector subspace $J$ of $\mathrm{End}(I_\tau)$ such that $f\circ g\in J$ (resp. $g\circ f\in J$), for all $f\in J$ and $g\in \mathrm{End}(I_\tau)$. A \textbf{two-sided ideal of }$\End(I_\tau)$ is a right ideal of $\End(I_\tau)$ which is also a left ideal.

\begin{notation}\label{not_ideal_modules}
For a right ideal $J\subset \End(I_\tau)$, we set $J(I_\tau)=\sum_{\phi\in J} \phi(I_\tau)$. For $M\subset I_\tau$ a submodule, we set $J_M=\{\phi\in \End(I_\tau)|\phi(\vb_\tau)\in M\}$.
\end{notation}

If $M$ is a submodule of $I_\tau$, then $J_M$ is a right ideal of $\End(I_\tau)$. Indeed, let $\phi\in J_M$ and $\phi'\in \End(I_\tau)$. Then $\phi'(\vb_\tau)\in I_\tau$ and thus there exists $h\in \AC_\C$ such that $\phi'(\vb_\tau)=h.\vb_\tau$. Then $\phi\circ\phi'(\vb_\tau)=h.\phi(\vb_\tau)$ and as $\phi(\vb_\tau)\in M$, $h.\phi(\vb_\tau)\in M$.

\begin{lemma}\label{lemElements_generalized_weight_spaces_image_intertwiners}
Let $M$ be a submodule of $I_\tau$ and $x\in M$. Then there exists a right ideal $J_x$ of $\End(I_\tau)$ such that $x\in J_x(I_\tau)\subset M$.
\end{lemma}

\begin{proof}
We first assume that $x\in M(\tau',\mathrm{gen})$, for some $\tau'\in W^v.\tau$. Let  $f:I_\tau\rightarrow I_{\tau'}$  be an isomorphism. Let $M'=f(M)$ and $x'=f(x)$. Then by Lemma~\ref{lemModules_K_generated_weight_subspaces},  there exist $n\in \N$ and $(h_i)\in (\AC_\C)^n$, $(x_i)\in M'(\tau')^n$  such that $x'=\sum_{i=1}^n h_i.x_i$. For $i\in \llbracket 1,n\rrbracket$, let $\phi_i\in \mathrm{End}(I_{\tau'})$ be such that $\phi_i(\vb_{\tau'})=x_i$, which exists by Lemma~\ref{lemFrobenius_reciprocity}. Then $x'=\sum_{i=1}^n \phi_i(h_i.\vb_{\tau'})\in \sum_{i=1}^n \phi_i(I_{\tau'})$ and thus $x\in \sum_{i=1}^n f^{-1}\circ \phi_i\circ f(I_\tau)$. Moreover for $i\in \llbracket 1,n\rrbracket$, $\phi_i(\vb_\tau)\in M'$, thus $\sum_{i=1}^n  \phi_i (I_{\tau'})\subset M'$ and hence $\sum_{i=1}^n f^{-1}\circ \phi_i\circ f(I_\tau)\subset M$. Set $J_x=\sum_{i=1}^n \big(f^{-1}\circ \phi_i\circ f\big)\circ \mathrm{End}(I_\tau)$. Then $x\in J_x(I_\tau)\subset M$. 

We no longer assume that $x\in M(\tau',\mathrm{gen})$, for some $\tau'\in W^v.\tau$. By \cite[Lemma 3.3]{hebert2018principal}, one has $M=\sum_{\tau'\in \Wt(M)} M(\tau',\mathrm{gen})$. For $\tau'\in \Wt(M)$ and $x_{\tau'}\in M(\tau',\mathrm{gen})$, choose a right ideal $J_{x_{\tau'}}\subset \mathrm{End}(I_\tau)$ such that $x_{\tau'}\subset J_{x_{\tau'}}(I_\tau)\subset M$. Then $x\in (\sum_{\tau'\in \Wt(M)} J_{x_{\tau'}})(I_\tau)\subset M$ and thus one can choose $J_x=\sum_{\tau'\in \Wt(M)} J_{x_{\tau'}}$.
\end{proof}

\begin{lemma}\label{lemProjection_on_Itau(tau)_and_intertwiners}
Let $\pi^{R_\tau}:I_\tau(\tau,\mathrm{gen})\rightarrow I_\tau(\tau)$ be the linear map  defined by $\pi^{R_\tau}\big(K_{\underline{w}}.\psi_{w_R}(\vb_\tau)\big)=0$  and $\pi^{R_\tau}\big(\psi_{w_R}(\vb_\tau)\big)=\psi_{w_R}(\vb_\tau)$, for $w\in \Wta\setminus\{1\}$ and $w_R\in \Wta$. Then for all $\phi\in \mathrm{End}(I_\tau)$, one has $(\phi\circ \pi^{R_\tau})|I_\tau(\tau,\mathrm{gen})=(\pi^{R_\tau}\circ \phi )|I_\tau(\tau,\mathrm{gen})$.
\end{lemma}

\begin{proof}
The map $\pi^{R_\tau}$ is well defined by Proposition~\ref{propDescription_generalized_weight_spaces}. Let $\phi\in \mathrm{End}(I_\tau)$ and $w_R\in R_\tau$. Then by Lemma~\ref{lemReeder 4.3}, $\phi(F_{w_R}(\tau)\vb_\tau)\in I_\tau(\tau)$. By Proposition~\ref{propDescription_generalized_weight_spaces}, as we assumed $I_\tau(\tau,\Wta)=\C \vb_\tau$, we have: \[\pi^{R_\tau}\circ\phi\big(\psi_{w_R}(\vb_\tau)\big)=\phi\big(\psi_{w_R}(\vb_\tau)\big)=\phi\circ \pi^{R_\tau}\big(\psi_{w_R}(\vb_\tau)\big).\]

Write $\phi\big(\psi_{w_R}(\vb_\tau)\big)=\sum_{v_R\in R_\tau} a_{v_R} \psi_{v_R}(\vb_\tau)$, where $(a_{v_R})\in \C^{(R_\tau)}$. Let $w\in \Wta\setminus\{1\}$. Then $\phi\big(K_{\underline{w}}(\tau).\psi_{w_R}(\vb_\tau)\big)=\sum_{v_R\in R_\tau}a_{v_R} K_{\underline{w}}(\tau).\psi_{v_R}(\vb_\tau)$. Therefore $\pi^{R_\tau}\circ \phi\big(K_{\underline{w}} (\tau)\psi_{w_R}(\vb_\tau)\big)=0=\phi\circ \pi^{R_\tau}\big(K_{\underline{w}}(\tau)\psi_{w_R}(\vb_\tau)\big)$, which proves the lemma. 
\end{proof}

\begin{theorem}\label{thmBijection_modules_ideals}
Let $\tau\in \UC_\C$ be such that $I_\tau(\tau,\Wta)=\C  \vb_\tau$. We use Notation~\ref{not_ideal_modules}. Then the assignment $M\mapsto J_M$ defines a bijection between the set of submodules of $I_\tau$ and the set of right ideals of $\End(I_\tau)$. Its inverse is the map $J\mapsto J(I_\tau)$.
\end{theorem}

\begin{proof}
Let $M\subset I_\tau$ be  a submodule. Then \[J_M(I_\tau)=\sum_{\phi\in J_M} \phi(I_\tau)=\sum_{\phi\in J_M}\phi(\AC_\C. \vb_\tau)=\sum_{\phi\in J_M} \AC_\C .\phi(\vb_\tau)\subset M,\] by definition of $J_M$. By \cite[Lemma 3.3]{hebert2018principal}, one has $M=\sum_{\tau'\in \Wt(M)} M(\tau',\mathrm{gen})$. For  $x\in M$, choose a right ideal $J_x\subset \mathrm{End}(I_\tau)$ such that $x\in J_x(I_\tau)\subset M$, whose existence is provided by Lemma~\ref{lemElements_generalized_weight_spaces_image_intertwiners}. Then $M\subset \sum_{x\in M} J_x(I_\tau)\subset M$. Moreover $J_x\subset J_M$ for all $x\in M$ and hence $M\subset J_M(I_\tau)\subset M$. 

Let $J$ be a right ideal of $\mathrm{End}(I_\tau)$. Let $\phi\in J$. Then $\phi(\vb_\tau)\in J(I_\tau)$ and thus $\phi\in J_{J(I_\tau)}$. Hence $J\subset J_{J(I_\tau)}$. Let $\phi\in J_{J(I_\tau)}$. Then $\phi(\vb_\tau)\in J(I_\tau)$ and thus there exist $k\in \Ne$, $\phi_1,\ldots,\phi_k\in J$ and $x_1,\ldots,x_k\in I_\tau$ such that $\phi(\vb_\tau)=\sum_{i=1}^k \phi_i(x_i)$. By \cite[Lemma 3.3]{hebert2018principal}, we may assume that $x_i\in I_\tau(\tau,\mathrm{gen})$ for all $i\in \llbracket 1,k\rrbracket$. By Lemma~\ref{lemProjection_on_Itau(tau)_and_intertwiners}, one has $\phi(\vb_\tau)=\pi^{R_\tau}\circ\phi (\vb_\tau)=\sum_{i=1}^n\phi_i(y_i)$, where $y_i=\pi^{R_\tau}(x_i)\in I_\tau(\tau)$, for $i\in \llbracket 1,n\rrbracket$.  For $i\in \llbracket 1,n\rrbracket$, let $\phi_i'\in \mathrm{End}(I_\tau)$ be  such that $\phi_i'(\vb_\tau)=y_i$, which exists by Lemma~\ref{lemFrobenius_reciprocity}.  Then $\phi(\vb_\tau)=(\sum_{i=1}^n \phi_i\circ\phi_i)'(\vb_\tau)$. As $\sum_{i=1}^n\phi_i\circ\phi_i'\in \mathrm{End}(I_\tau)$, we deduce that $\phi=\sum_{i=1}^n \phi_i\circ \phi_i'$ and hence $\phi\in J$. Therefore  $J= J_{J(I_\tau)}$, which proves the theorem.
\end{proof}

\begin{remark}
From the definition of $M\mapsto J_M$ and $J\mapsto J(I_\tau)$, it is clear that these maps are (strictly) increasing. If $M,M'$ are submodules of $I_\tau$, then $J_{M\cap M'}=J_M\cap J_{M'}$ and if $J,J'$ are right ideals of $\End(I_\tau)$, then $(J+J')(I_\tau)=J(I_\tau)+J'(I_\tau)$. 
\end{remark}

\begin{corollary}\label{corTwo-sided_ideals}
Let $M\subset I_\tau$ be a submodule. Then the following are equivalent: \begin{enumerate}
\item for every $\phi\in \End(I_\tau)$, $\phi(M)\subset M$,

\item $J_M$ is a two-sided ideal.
\end{enumerate}

If these conditions  hold, then we have a natural map $\End(I_\tau)\rightarrow \End(I_\tau/M)$. 
\end{corollary}

\begin{proof}
Suppose that $J_M$ is a two-sided ideal. Let $x\in M$ and $\phi\in \End(I_\tau)$. Then by Theorem~\ref{thmBijection_modules_ideals}, there exist $k\in \N$, $\phi_1,\ldots,\phi_k\in J_M$ and $x_1,\ldots,x_k\in I_\tau$ such that $x=\sum_{i=1}^k \phi_i(x_i)$. Then $\phi(x)=\sum_{i=1}^k \phi\circ \phi_i(x_i)$. By assumption, $\phi\circ\phi_i(x_i)\in M$ for all $i\in \llbracket 1,k\rrbracket$ and thus $\phi(x)\in M$, which proves that $\phi(M)\subset M$. Reciprocally suppose (1). Let $\phi\in J_M$ and $\phi'\in \End(I_\tau)$. Then $\phi(\vb_\tau)\in M$, thus $\phi'\circ\phi(\vb_\tau)\in M$ and hence $\phi'\circ\phi\in J_M$, which proves (2).
\end{proof}

\subsection{Irreducible representations admitting $\tau$ as a weight when $I_\tau(\tau,\Wta)=\C\vb_\tau$}\label{subIrreducible_representations}

We now study $I_\tau/M$ for $M$ a maximal submodule of $I_\tau$ and we give a criterion for $M\mapsto I_\tau/M$ to be a bijection between the maximal submodules of $I_\tau$ and the irreducible representations admitting $\tau$ as a weight (see Theorem~\ref{thmIrreducible_representations_tau_UC}).

\begin{lemma}\label{lemSufficient_condition_uniqueness_submodule_irreducible_component}
Let $M\subset I_\tau$ be a submodule. Then the following properties are equivalent:

\begin{enumerate}
\item $\dim I_\tau/M(\tau) =1$, 

\item $I_\tau/M$ is irreducible and $M$ is the unique submodule $M'$ of $I_\tau$ such that $I_\tau/M'$ is isomorphic to $I_\tau/M$.

\item $J_M$ is a maximal right ideal of $\End(I_\tau)$ and is two sided.
\end{enumerate}

In particular if $N$  is a $\AC_\C$-module such that $\dim N(\tau)=1$, then $\AC_\C.N(\tau)$ is an irreducible representation of $\AC_\C$. 
\end{lemma}

\begin{proof}
Suppose that $\dim I_\tau/M(\tau)=1$. Let $M'$ be a submodule of $I_\tau$ such that there exists an isomorphism $f:I_\tau/M'\rightarrow I_\tau/M$. Let $\pi_M:I_\tau\twoheadrightarrow I_\tau/M$ and  $\pi_{M'}:I_\tau\twoheadrightarrow I_\tau/{M'}$ be the canonical projections. By assumption, one has $I_\tau/M(\tau)=\C\pi_M(\vb_\tau)$ and $I_\tau/M'(\tau)=\C \pi_{M'}(\vb_\tau)$. Thus maybe considering $af$ for some $a\in \C^*$, we may assume that $f(\pi_{M'}(\vb_\tau)\big)=\pi_M(\vb_\tau)$. Let $m\in M'$. Write $m=h.\vb_\tau$, for some $h\in \AC_\C$. Then $f\big(\pi_{M'}(m)\big)=0=h.f\big(\pi_{M'}(\vb_\tau)\big)=h.\pi_{M}(\vb_\tau)=\pi_M(m)$ and thus $m\in M$. Consequently, $M'\subset M$ and by symmetry, $M'=M$.

Let $M'\supsetneq M$ be a submodule of $I_\tau$. By Lemma~\ref{lemCharacterization_equality_generalized_weight_spaces} $M'(\tau)\supsetneq M(\tau)$. By Proposition~\ref{propK_description_generalized_weight_spaces_quotient} (3), $M(\tau)$ is a one-codimensional  subspace of $I_\tau(\tau)$, thus $M'(\tau)=I_\tau(\tau)$ and hence by Lemma~\ref{lemCharacterization_equality_generalized_weight_spaces}, $M'=I_\tau$. Therefore, $M$ is a maximal submodule of $I_\tau$ and thus $I_\tau/M$ is irreducible. 

Suppose (2). Suppose that $I_\tau/M(\tau)\neq \C \pi_M(\vb_\tau)$. Let $\overline{x}\in I_\tau/M(\tau)\setminus \C\pi_M(\vb_\tau)$. Let $f\in \Hom(I_\tau,I_\tau/M)$ be such that $f(\vb_\tau)=\overline{x}$, which exists by Lemma~\ref{lemFrobenius_reciprocity}. As $I_\tau/M$ is irreducible, $f(I_\tau)=I_\tau/M$ and thus $f$ induces an isomorphism $\overline{f}:I_\tau/\ker(f)\overset{\sim}{\rightarrow} I_\tau/M$. Thus $\ker(f)=M$. Moreover $\overline{f}\big(\pi_M(\vb_\tau)\big)=\overline{x}$. Thus $\End(I_\tau/M)\neq \C \Id$ and by Schur's Lemma (\cite[B.II Théorème]{renard2010representations}), $I_\tau/M$ is reducible: a contradiction. Therefore $\dim I_\tau/M(\tau)=1$.

Suppose (3). Then by Theorem~\ref{thmBijection_modules_ideals}, $I_\tau/M$ is irreducible. Let $\overline{x}\in I_\tau/M(\tau)$. By Proposition~\ref{propK_description_generalized_weight_spaces_quotient}~(3), there exists $x\in I_\tau(\tau)$ such that $\overline{x}=\pi_M(x)$. By Lemma~\ref{lemFrobenius_reciprocity}, there exists $\phi\in \End(I_\tau)$ such that $\phi(\vb_\tau)=x$. By Corollary~\ref{corTwo-sided_ideals}, $\phi(M)\subset M$. Therefore, $\phi$ induces a map $\overline{\phi}:I_\tau/M\rightarrow I_\tau/M$ such that $\overline{\phi}(\vb_\tau)=\overline{x}$. Therefore $\overline{\phi}$ is an isomorphism. By Schur's Lemma (\cite[B.II Théorème]{renard2010representations}), $\overline{\phi}\big(\pi_M(\vb_\tau)\big)=\overline{x}\in \C^* \pi_M(\vb_\tau)$ and thus $\dim I_\tau/M(\tau)=1$, which proves that (3) implies (1).

Suppose now that $J_M$ is not two sided. There exists $\phi\in \End(I_\tau),\psi\in J_M$ such that $\phi\circ \psi\notin J_M$. Therefore $\psi(\vb_\tau)\in M$ and $\phi\circ\psi(\vb_\tau)\notin M$: $\phi(M)\not \subset M$. Using Corollary~\ref{corEnd(Itau)_spanned_invertible}, we may assume that $\phi$ is invertible. Then $\phi(M)$ is a maximal submodule of $I_\tau$. Then $\phi$ induces an nonzero map $\overline{\phi}:I_\tau/M\rightarrow I_\tau/\phi(M)$, and $\overline{\phi}$ is an isomorphism, which contradicts (2). Thus (2) implies (3).

Let now $N$ be a  $\AC_\C$-module such that $\dim N(\tau)=1$. Let $x\in N(\tau)\setminus\{0\}$. By Lemma~\ref{lemFrobenius_reciprocity}, there exists $f:I_\tau\rightarrow \AC_\C.x=\AC_\C.N(\tau)$ such that $f(\vb_\tau)=x$. Then $f$ induces an isomorphism $\overline{f}:I_\tau/\ker(f)\overset{\sim}{\rightarrow} \AC_\C.N(\tau)$. Then $\dim I_\tau/\ker(f)(\tau)=1$, and hence $I_\tau/\ker(f)$ is irreducible.
\end{proof}

\begin{theorem}\label{thmIrreducible_representations_tau_UC}
Let $\tau\in \UC_\C$.

\begin{enumerate}
\item For every irreducible  $\AC_\C$-module $N$, $\tau\in \Wt(N)$ if and only if $\Wt(N)=W^v.\tau$.

\item The assignment $\Xi:M\mapsto I_\tau/M$ is a surjective map from the set of maximal submodules of $I_\tau$ to the set isomorphism classes of irreducible representations of $\AC_\C$ admitting the weight $\tau$. 

\item Suppose that $I_\tau(\tau,\Wta)=\C\vb_\tau$. Let $[N]$ be the isomorphism class of an irreducible representation of $\AC_\C$ admitting the weight $\tau$. Then $|\Xi^{-1}([N])|=1$ if and only if $\dim N(\tau)=1$ if and only if $J_M$ is a two-sided ideal, for any $M\in \Xi^{-1}([N])$. In particular, $\Xi$ is a bijection if and only if every maximal right ideal of $\End(I_\tau)$ is two-sided.

\item Suppose that $I_\tau(\tau,\Wta)=\C\vb_\tau$ and that $\Xi$ is a bijection, then for every irreducible representation $N$ admitting $\tau$ as a weight, one has $\dim N(\tau)=1=\dim N(w.\tau)$, $\dim N(\tau,\mathrm{gen})=|\Wta|=\dim N(w.\tau,\mathrm{gen})$, for every $w\in W^v$ and $\dim N=|\Wta||W^v/W_\tau|$.
\end{enumerate}
\end{theorem}

\begin{proof}
(1) is a a consequence of Lemma~\ref{lemStructure_Weights_module}. 

(2) Let $N$ be an irreducible representation of $\AC_\C$ admitting the weight $\tau$.  By \cite[Proposition 3.7]{hebert2018principal}, there exists  a surjective morphism $\phi:I_\tau \twoheadrightarrow N$. Then $\ker(f)$ is a maximal submodule of $I_\tau$, which proves (2).

(3) is a consequence of Lemma~\ref{lemSufficient_condition_uniqueness_submodule_irreducible_component}.

(4) By \cite[Lemma 3.3]{hebert2018principal}, $N=\bigoplus_{\overline{w}\in W^v/W_\tau} N(w.\tau,\mathrm{gen})$.  By Lemma~\ref{lemInvariance_dimension_weights_Weyl_group}, Lemma~\ref{lemInvariance_dimension_generalized_weight_spaces} and Proposition~\ref{propK_description_generalized_weight_spaces_quotient}, we deduce that $\dim N=|\Wta||W^v/W_\tau|$.

\end{proof}

\subsection{Case where the Kac-Moody matrix \texorpdfstring{$A$}{A} has size $2$}\label{subExamples}

In this section, we study the case where the Kac-Moody matrix defining the generating root system is not a Cartan matrix and has size $2$. We begin by studying all the possibilities for the triple $W_\tau$, $\Wta$, $R_\tau$ and then we study examples of $I_\tau$, for $\tau\in \UC_\C$.

We assume that there exists $\sigma\in \C$ such that  $\sigma_s=\sigma_s'=\sigma$ for all $s\in \SCC$. In particular, $\zeta_s=\frac{1-\sigma^2 Z^{-\alpha_s^\vee}}{1-Z^{-\alpha_s^\vee}}$, for all $s\in \SCC$.

\subsubsection{Possibilities for $W_\tau$, $\Wta$, $R_\tau$}

We write $\SCC=\{s_1,s_2\}$. Recall that $Q^\vee_\Z=\Z\alpha_{s_1}^\vee\oplus \Z \alpha_{s_2}^\vee$.

\begin{lemma}\label{lemEigenvalues_st}
Let $A=\begin{pmatrix}
2 & a_{1,2}\\ a_{2,1} & 2
\end{pmatrix}$ be a Kac-Moody matrix which is not a Cartan matrix. Let $k\in \Z\setminus\{0\}$ and $w=(s_1s_2)^k$. Then $\vect_\Q\big((w-\Id)(Y)\big)=\vect_\Q(Q^\vee_\Z)$. 
\end{lemma}

\begin{proof}
For all $\lambda\in Y$, $w.\lambda-\lambda\in Q^\vee_\Z$ and thus $\vect_\Q \big((w-\Id)(Y))\subset \vect_\Q(Q^\vee_\Z)$. 

In the basis $\alpha_{s_1}^\vee$, $\alpha_{s_2}^\vee$ of $Q^\vee_\Z$, the matrix of $s_1$, $s_2$ and $s_1s_2$ are $\begin{pmatrix}
-1 & -a_{2,1}\\ 0 & 1
\end{pmatrix}$ and $\begin{pmatrix}
1&0\\-a{1,2} & -1
\end{pmatrix}$ and $\begin{pmatrix}
-1 & -a_{2,1} \\a_{1,2} & a_{1,2}a_{2,1}+1
\end{pmatrix}$. The characteristic polynomial of $s_1s_2$ is $T^2-aT+1$, where $a=a_{1,2}a_{2,1}$ and $T$ is an indeterminate. Thus the eigenvalues of $s_1s_2$ are $\frac{a\pm \sqrt{a^2-4}}{2}\neq \pm 1$. 
\end{proof}

We denote by  $D_\infty=\langle s,t|s^2=t^2=1\rangle$\index{$D_\infty$}  the infinite dihedral group.

\begin{lemma}\label{lemList_possibilities_Wtau_Rtau_Wta}
The possibilities for the triple $R_\tau$, $\Wta$, $W_\tau$ are exactly:\begin{enumerate}
\item $W_\tau=\Wta=R_\tau=\{1\}$,

\item $W_\tau\simeq \Z/2\Z$, $\Wta=W_\tau$ and $R_\tau=\{1\}$,

\item $W_\tau\simeq \Z/2\Z$, $\Wta=\{1\}$ and $R_\tau=W_\tau$,

\item $W_\tau\simeq \Z$, $\Wta=\{1\}$ and $R_\tau=W_\tau$,

\item  $W_\tau\simeq D_\infty$, $\Wta=W_\tau$ and $R_\tau=\{1\}$,

\item $W_\tau\simeq D_\infty$, $\Wta=\{1\}$ and $R_\tau=W_\tau$, 

\item $W_\tau\simeq D_\infty$, $\Wta\simeq D_\infty$ and $R_\tau\simeq \Z/2\Z$.

Moreover, if $\tau\in T_\C\setminus \UC_\C$, then $W_\tau=\{1\}$ or $W_\tau\simeq \Z/2\Z$.
\end{enumerate}

\end{lemma}

\begin{proof}
We begin by proving the existence of size $2$ Kac-Moody matrices $A$, of root generating system  $\mathcal{S}=(A,X,Y,(\alpha_i)_{i\in I},(\alpha_i^\vee)_{i\in I})$ and of $\tau\in T_\C$ for (1) to (7). We assume that $\alpha_{s_1}(Y)=\alpha_{s_2}(Y)=2\Z$, which is possible by  taking the ``donnée radicielle simplement connexe'' of \cite[7.1.2]{remy2002groupes}. By  \cite[Lemma 6.2]{hebert2018principal}, for all $\gamma_1,\gamma_2\in \C^*$, there exists $\tau_\gamma\in T_\C$ such that $\tau_\gamma(\alpha_{s_i})=\gamma_i$, for $i\in \{1,2\}$.

(1) This is a consequence of  \cite[Lemma  6.5]{hebert2018principal}.

(2) Set $\gamma_1=1$, choose $\gamma_2\in \C^*$ a transcendental number. Then $s_1\in W_\tau$ and by \cite[Lemma 6.18]{hebert2018principal}, $W_\tau\subset \{1,s_1\}$. Then $s_1\in \Wta$ and thus $W_\tau=\Wta$.

(3) Set $\gamma_1=-1$, choose $\gamma_2\in \C^*$ a transcendental number. A similar proof as in (2) proves that $W_\tau=\langle s_1\rangle$ and $R_\tau=W_\tau$.

(4) By \cite[Lemma B1]{hebert2018principal}, we can have $W_\tau\simeq \Z$. As $\Wta$ is generated by reflections, we have $\Wta=\{1\}$ and thus $R_\tau=W_\tau$.

(5) Set $\gamma_1=1$, $\gamma_2=2$. Then $\tau_\gamma$ satisfies (5).

(6) Suppose that the Kac-Moody matrix  $A=\begin{pmatrix}
2 & a_{1,2}\\ a_{2,1} & 2
\end{pmatrix}$ is such that $a_{1,2},a_{2,1}\in \Z_{\leq -2}$ are even. Let $\gamma_1=\gamma_2=-1$ and $\tau=\tau_\gamma$. 
Let $\mathrm{ht}:\Z \alpha_{s_1}^\vee\oplus \Z\alpha_{s_2}^\vee\rightarrow \Z$ be defined by $\mathrm{ht}( n_1\alpha_{s_1}^\vee+n_{2}\alpha_{s_2}^\vee)=n_1+n_2$, for $n_1,n_2\in \Z$. Let $\lambda\in \Z\alpha_{s_1}\oplus \Z\alpha_{s_2}^\vee$ be such that $\mathrm{ht}(\lambda)$ is odd. Let $i\in \{1,2\}$. Then $\mathrm{ht}(s_i.\lambda)=\mathrm{\lambda}-\mathrm{ht}(\alpha_i(\lambda)\alpha_i^\vee)=\mathrm{ht}(\lambda)-\alpha_i(\lambda)$. Write $\lambda=n_1\alpha_{s_1}^\vee+n_2\alpha_{s_2}^\vee$, with $n_1,n_2\in \Z$. Let $j\in \{1,2\}\setminus\{i\}$. Then $\alpha_{s_i}(\lambda)=2n_i+n_ja_{j,i} $ is even and thus $\mathrm{ht}(s_i.\lambda)$ is odd. By induction we deduce that for all $\alpha^\vee\in \Phi^\vee$, $\mathrm{ht}(\alpha^\vee)$ is odd. Therefore $\tau(\Phi^\vee)=\{-1\}$ and hence $\Wta=\{1\}$.

(7) Let $A=\begin{pmatrix}
2 & a_{1,2}\\ a_{2,1} & 2
\end{pmatrix}$ be a Kac-Moody matrix such that $a_{1,2}$ is even. Let $\gamma_1=1$ and $\gamma_2=-1$ and $\tau=\tau_\gamma$. Then $s_1,s_2\in W_\tau$ and thus $W_\tau=W^v$. Then $s_2.\tau(\alpha_{s_1}^\vee)=1$ and thus $s_2s_1s_2\in \Wta$. Therefore $\langle s_1,s_2s_1s_2\rangle\subset \Wta\subsetneq W_\tau$ ($s_2\notin \Wta$). Moreover, $\langle s_1,s_2s_1s_2\rangle$ is a normal subgroup of $W^v$ and $W^v/\langle s_1,s_2s_1s_2\rangle\simeq \Z/2\Z$  and hence $\Wta=\langle s_1,s_2s_1s_2\rangle$. Moreover by Lemma~\ref{lemDecomposition_Wtau_Rgroup}, $R_\tau\simeq W_\tau/\Wta\simeq \Z/2\Z$. 

Let us prove that there are no other possibilities. By \cite[Lemma 6.36]{hebert2018principal}, we made the list of all the possible $W_\tau$.  As $\Wta$ is generated by reflections, if $W_\tau\simeq \Z$, then $\Wta=\{1\}$ and $R_\tau=W_\tau$. Suppose that $W_\tau\simeq D_\infty$. By Lemma~\ref{lemDecomposition_Wtau_Rgroup}, $\Wta$ is normal in $W_\tau$. If $w\in W^v$ and $s\in \SCC$, then $\langle wsw^{-1}\rangle$ is not normal in $W_\tau$ (if $i\in \{1,2\}$ is such that the reduced writing of $w$ does not begin with $s_i$, then $s_iwsw^{-1}s_i\notin \langle wsw^{-1}\rangle$). By \cite[Lemma 6.36]{hebert2018principal} we deduce that if $\Wta\neq \{1\}$, one has $\Wta\simeq \Z$ or $\Wta\simeq D_\infty$ and thus $\Wta\simeq D_\infty$. In particular, $R_\tau=W_\tau/\Wta$ is finite. By  \cite[Theorem 4.5.3]{bjorner2005combinatorics}, we deduce that $R_\tau\simeq \Z/2\Z$. 

Let now $\tau\in T_\C$ be such that $W_\tau\neq \{1\}$ is not isomorphic to $\Z/2\Z$. Then there exists $w\in W_\tau\setminus\{1\}$ such that $w$ is not a reflection. Then $w=(s_1s_2)^n$, for some $n\in \Z\setminus\{0\}$. By Lemma~\ref{lemEigenvalues_st}, there exists $(y_1,y_2)\in (w^{-1}-\Id)(Y)$  such that $(y_1,y_2)$ is a $\Q$-basis of  $Q^\vee\otimes \Q$. For $i\in \{1,2\}$, write $y_i=w^{-1}.x_i-x_i$, with $x_i\in Y$. Then $w.\tau(y_i)= w.\tau(x_i)\tau^{-1}(x_i)=1$. Let $k\in \N^*$ be such that $ky_1,ky_2\in Y$. Then $\tau(\alpha_{s_1}^\vee)^k=\tau(\alpha_{s_2}^\vee)^k=1$ and thus $|\tau(\lambda)|=1$ for all $\lambda\in Q^\vee$. Therefore $\tau\in \UC_\C$. 
\end{proof}

\subsubsection{The case $R_\tau\simeq \Z/2\Z$}

Let $\tau\in \UC_\C$ be such that $R_\tau\simeq \Z/2\Z$. Let $r\in W^v$ be such that $R_\tau=\langle r\rangle$. Then by Proposition~\ref{propDescription_endomorphism_algebra}, $\End(I_\tau)\simeq \C[T]/(T^2-1)\simeq \C\times \C$, where $T$ is an indeterminate. Let $\psi'=\psi_{r}'$, with the notation of the proof of Proposition~\ref{propDescription_endomorphism_algebra}. Then the following map $\End(I_\tau)\rightarrow \C\times \C$ is an isomorphism: $a\psi'+b\mapsto (a+b,a-b)$, for $a,b\in \C$. The ideals of $\C\times \C$ are $\{0\}$, $\C\times \{0\}$, $\{0\}\times \C$ and $\C\times \C$. Therefore the nontrivial submodules of $I_\tau$ are $M_{(1,0)}:=(\psi'+\Id)(I_\tau)$ and $M_{(0,1)}:=(\psi'-\Id)(I_\tau)$. These submodules are irreducible. If $x\in I_\tau$, then $x=\frac{1}{2}\big((\psi'(x)+\Id(x))-(\psi'(x)-\Id(x)\big)$ and thus $M_{(1,0)}+ M_{(0,1)}=I_\tau$. Moreover, $M_{(1,0)}\cap M_{(0,1)}$ is a submodule of $M_{(1,0)}$ and as $M_{(1,0)}\not\subset M_{(0,1)}$, one has $M_{(1,0)}\cap M_{(0,1)}=\{0\}$. Therefore $M_{(1,0)}\oplus  M_{(0,1)}=I_\tau$.  By Theorem~\ref{thmIrreducible_representations_tau_UC}, $M_{(1,0)}\simeq I_\tau/M_{(0,1)}$ and $M_{(0,1)}\simeq I_\tau/M_{(1,0)}$ are not isomorphic, $M_{(1,0)}(\tau)=\C (\psi'(\vb_\tau)+\vb_\tau)$ and $M_{(0,1)}(\tau)=\C (\psi'(\vb_\tau)-\vb_\tau)$. 

\subsubsection{The case $W_\tau=R_\tau\simeq \Z$}\label{subsubWtau=Z}

Let $\tau\in \UC_\C$ be such that $W_\tau=R_\tau\simeq\Z$. Then by Proposition~\ref{propDescription_endomorphism_algebra}, $\End(I_\tau)\simeq \C[\Z]=\C[T,T^{-1}]$, where $T$ is an indeterminate and thus $\End(I_\tau)$ is commutative. The ideals of $\End(I_\tau)$ are the $P\C[T,T^{-1}]$ such that $P\in \C[T,T^{-1}]$ and the maximal ideals are the $(T+a)\C[T,T^{-1}]$ such that $a\in \C^*$. Write $R_\tau=\langle (s_1s_2)^k\rangle$, where $k\in \Ne$. Let $\psi=\psi_{(s_1s_2)^k}$. The maximal submodules of $I_\tau$ are the $(\psi+a)(I_\tau)$, for $a\in \C^*$. The group $W^v/W_\tau$ has $2k$ elements, $\Wta=\{1\}$ and thus the irreducible representations $M$ having the weight $\tau$ decompose as $M=\bigoplus_{\overline{w}\in W^v/W_\tau} M(w.\tau)=\bigoplus_{w\in W^v|\ell(w)<k}M(w.\tau)\oplus M((s_1s_2)^k.\tau)$. In particular, they have dimension $2k$. 

By Lemma~\ref{lemCharacterization_injectivity_UC}, every nonzero element of $\End(I_\tau)$ is injective and by Lemma~\ref{lemCharacterization_surjectivity_UC}, the only surjective elements of $\End(I_\tau)$ are the invertible ones. Let $M\subset I_\tau$ be a nonzero submodule. As $\C[T,T^{-1}]=\End(I_\tau)$ is principal, $J_M$ is principal and there exists $\phi\in J_M$ such that $M=\phi(I_\tau)$. Then $\phi:I_\tau\rightarrow M$ is an isomorphism: every nonzero submodule of $I_\tau$ is isomorphic to $I_\tau$. Thus we can construct an infinite strictly decreasing sequence $(M_i)_{i\in \N}$ of submodules of $I_\tau$ and no submodule of $I_\tau$ is irreducible. 
 
 \subsubsection{The case $R_\tau=D_\infty$}

Recall that $D_\infty=\langle s,t|s^2=t^2=1\rangle$ is the infinite dihedral group. We now study $\C[D_\infty]$. We determine its maximal right ideals which are two-sided (see Lemma~\ref{lemMaximal_two_sided_ideals}) and we prove the existence of maximal right ideals which are not two sided (see Lemma~\ref{lemExistence_maximal_right_ideal_not_two-sided}) . Let $S=e^s,T{}=e^t\in \C[D_\infty]$. If $(a,b)\in \{-1,1\}^2$, we denote by $\ev_{(a,b)}:\C[D_\infty]\rightarrow \C$ the $\C$-algebra morphism such that $\ev_{(a,b)}(S)=a$ and 
$\ev_{(a,b)}(T{})=b$. Recall that if $a,b\in \C[D_\infty]$ and $m\in \N$, we denote by  $\Pi(a,b,m)$ the product $abab\ldots $ having $m$ factors.
 For $Q\in \C[D_\infty]\setminus\{0\}$, $Q=\sum_{k\in \N} a_k \Pi(S,T{},k)+b_k\Pi(T{},S,k)$ where $(a_k),(b_k)\in \C^{(\N)}$, we set $\deg(Q)=\max \{k\in \N||a_k|+|b_k|\neq 0\}$.

\begin{lemma}\label{lemQuotient_CW_finite_dimensional}
Let $P\in \C[D_\infty]\setminus\C$ and $J$ be the two-sided ideal $\C[D_\infty]P\C[D_\infty]$. Then $\C[D_\infty]/J$ is a finite dimensional vector space over $\C$.
\end{lemma}

\begin{proof}
If $Q\in \C[D_\infty]$, we denote by $\overline{Q}$ its image in $\C[D_\infty]/J$. For $k\in \N$, set $\mathcal{A}_k=\sum_{j=0}^k \big(\C \Pi(\overline{S},\overline{T{}},j)+ \C \Pi(\overline{T{}},\overline{S},j)\big)$. 
Let $n=\deg(P)$. Write $P=\sum_{k\leq n} a_k \Pi(S,T{},k)+b_k\Pi(T{},S,k)$, where $(a_k),(b_k)\in \C^{n+1}$. Maybe considering $aP$, for some $a\in\C^*$ and exchanging the roles of $S$ and $T{}$ we may assume that $a_n=1$. 

 First assume that $b_n\neq 0$. Then \[\begin{aligned}&  \Pi(\overline{S},\overline{T{}},n+1)\\
  &=\overline{S}\Pi(\overline{T{}},\overline{S},n)\\ 
&=-\frac{1}{b_n}\bigg(\overline{S}\Pi(\overline{S},\overline{T{}},n)+\sum_{k=0}^{n-1} \big((a_k\Pi(\overline{S},\overline{T{}},k)+b_k\Pi(\overline{T{}},\overline{S},k)\big)\bigg)\\ 
&=-\frac{1}{b_n}\big(\Pi(\overline{T{}},\overline{S},n-1)+\sum_{k=1}^{n-1} (a_k\Pi(\overline{T{}},\overline{S},k-1)+b_k\Pi(\overline{S},\overline{T{}},k+1)+a_0\overline{S}+b_0\overline{S}\big),  \end{aligned}\] thus $\Pi(\overline{S},\overline{T{}},n+1)\in \mathcal{A}_n$. Symmetrically, $\Pi(\overline{T{}},\overline{S},n+1)\in \mathcal{A}_n$.

Now assume that $b_n=0$. Then $\Pi(\overline{S},\overline{T{}},n+1)=\Pi(\overline{S},\overline{T{}},n)A$, for some $A\in \{\overline{S},\overline{T{}}\}$. As $\Pi(\overline{S},\overline{T{}},n)\in \mathcal{A}_{n-1}$, we deduce that $\Pi(\overline{S},\overline{T{}},n+1)\in \mathcal{A}_n$. 

One has $\Pi(\overline{T{}},\overline{S},n+1)=\overline{T{}}\Pi(\overline{S},\overline{T{}},n)$. As $\Pi(\overline{S},\overline{T{}},n)\in \mathcal{A}_{n-1}$, we deduce that $\Pi(\overline{T{}},\overline{S},n+1)\in \sum_{k=0}^{n} \mathcal{A}_n$. 

In both cases ($b_n\neq 0$ and $b_n=0$), $\C \Pi(\overline{S},\overline{T{}},n+1)\subset \mathcal{A}_n$ and thus $\mathcal{A}_{n+1}\subset \mathcal{A}_n$. 
Let $m\in \llbracket n+1,+\infty\llbracket$ be such that $\mathcal{A}_m\subset \mathcal{A}_n$. 
Then $\Pi(\overline{S},\overline{T{}},m+1)=\overline{S}\Pi(\overline{T{}},\overline{S},m)\in \mathcal{A}_{n+1}\subset \mathcal{A}_n$. 
Symmetrically, $\Pi(\overline{T{}},\overline{S},m+1)\in \mathcal{A}_n$
 and thus $\mathcal{A}_{m+1}\subset \mathcal{A}_n$. Therefore $\C[D_\infty]/J=\bigcup_{m\in \N} \mathcal{A}_m=\mathcal{A}_n$ is finite dimensional.
\end{proof}

\begin{lemma}\label{lemMaximal_two_sided_ideals}
The maximal right ideals of $\C[D_\infty]$ which are two-sided ideals are exactly the $\ev_{(a,b)}^{-1}(\{0\})$ such that $(a,b)\in \{-1,1\}$. 
\end{lemma}

\begin{proof}
Let $J$ be a maximal two-sided ideal of $\C[D_\infty]$. Then $\C[D_\infty]/J$ is a field and by Lemma~\ref{lemQuotient_CW_finite_dimensional}, it is a finite dimensional $\C$-algebra. By Frobenius theorem, we deduce that $\C[D_\infty]/J$ is either isomorphic to $\C$ or isomorphic to the division algebra $\mathbb{H}$ of quaternions. Let $f:\C[D_\infty]\rightarrow \mathbb{H}$ be an algebra morphism. Then $f(S^2)=f(T{}^2)=1$ and thus $f(S),f(T{})\in \{-1,1\}$ and $f(\C[D_\infty])=\C$. Therefore the algebra morphisms from $\C[D_\infty]$ to $\C$ are exactly the $\ev_{(a,b)}$ such that $(a,b)\in \{-1,1\}$. Consequently the maximal two-sided ideals of $\C[D_\infty]$ are exactly the $\ev_{(a,b)}^{-1}(\{0\})$ such that $(a,b)\in \{-1,1\}$. Let $(a,b)\in \{-1,1\}$ and $J=\ev_{(a,b)}^{-1}(\{0\})$.  We regard $J$ as a right ideal. As $J$ has codimension $1$, it is maximal as a right ideal  which proves the lemma.
\end{proof}

For example if $\tau\in \UC_\C$ is such that  $W^v=W_\tau=R_\tau$, the lemma above prove that there are exactly four one dimensional representations admitting $\tau$ as a weight.

\begin{lemma}\label{lemExistence_maximal_right_ideal_not_two-sided}
Let $a\in \C^*$ and $P=1-a(ST{}-T{}S)$. Then $P\C[D_\infty]$ is  a proper right ideal of $\C[D_\infty]$ which is not contained in any proper two-sided ideal. Therefore, any maximal right ideal containing $P\C[D_\infty]$ is not two-sided.
\end{lemma}

\begin{proof}
Let us prove that $1\notin \C[D_\infty]$. Let $Q\in \C[D_\infty]\setminus \C$. Let $d=\deg(Q)$. Write $Q=b\Pi(S,T{},d)+c\Pi(T{},S,d)+\tilde{P}$, with $\tilde{P}\in \C[D_\infty]$ such that $\deg (\tilde{P})\leq d-1$ and $b,c\in \C$. 
Then $PQ=-abST{}\Pi(S,T{},d)+acT{}S\Pi(T{},S,d)+\tilde{Q}$, where $Q\in \C[D_\infty]$ is such that $\deg(\tilde{Q})\leq d+1$. Thus $\deg (QP)=d+2$ and hence $QP\neq 1$. Therefore $P\C[D_\infty]$ is a proper right ideal. Let $J$ be a two-sided ideal containing $P$. Suppose that $J$ is proper. Let $J'$ be a maximal two-sided ideal containing $J$. Then by Lemma~\ref{lemMaximal_two_sided_ideals}, $J'\ni ST{}-T{}S$ and thus $1\in J'$: a contradiction. Lemma follows. 
\end{proof}

\printindex

\bibliography{/home/auguste/Documents/Projets/bibliographie.bib}
\bibliographystyle{plain}
\end{document}